\documentclass[leqno,12pt]{amsart}
\usepackage{amssymb,enumerate,eepic}
\usepackage[dvipdfmx]{graphicx}
\numberwithin{equation}{subsection}
\allowdisplaybreaks[4]
\theoremstyle{plain}
\newtheorem{theorem}{\indent\sc Theorem}[subsection]
\newtheorem{lemma}[theorem]{\indent\sc Lemma}
\newtheorem{proposition}[theorem]{\indent\sc Proposition}
\theoremstyle{definition}
\newtheorem{definition}[theorem]{\indent\sc Definition}
\newtheorem{remark}[theorem]{\indent\sc Remark}

\newtheorem*{acknowledgments}{\indent\sc Acknowledgments} 

\begin{document}

\title[Pluriharmonic maps and para-pluriharmonic maps]
{On a relation between potentials for pluriharmonic maps and para-pluriharmonic 
maps}

\author[N.~Boumuki]{Nobutaka Boumuki}
\author[J.~F.~Dorfmeister]{Josef F.~Dorfmeister}

\date{}


\address{Nobutaka Boumuki\endgraf  
Osaka City University Advanced  Mathematical Institute\endgraf
3-3-138, Sugimoto, Sumiyoshi-ku, Osaka, 558-8585, Japan}
\email{boumuki@sci.osaka-cu.ac.jp}

\address{Josef F.~Dorfmeister\endgraf    
TU M\"{u}nchen, Zentrum Mathematik (M8)\endgraf  
Boltzmannstr.\ 3, 85748, Garching, Germany}
\email{dorfm@ma.tum.de}
\maketitle

\begin{abstract} 
   In this paper, we show that one can interrelate pluriharmonic maps with 
para-pluriharmonic maps by means of the loop group method.   
   As an appendix, we give examples for the interrelation between pluriharmonic 
maps and para-pluriharmonic maps. 
   Moreover, we investigate the relation among CMC-surfaces by use of such 
maps.  
\end{abstract}

\section{Introduction}\label{sec-1} 
   Let $f_1:(M_1,J)\to G_1/H_1$ be a pluriharmonic map from a complex manifold 
$(M_1,J)$, and let $f_2:(M_2,I)\to G_2/H_2$ be a para-pluriharmonic map from a 
para-complex manifold $(M_2,I)$, where $G_i/H_i$ are affine symmetric spaces. 
   Then, the loop group method enables us to obtain a pluriharmonic potential 
$(\eta_\lambda,\tau_\lambda)$ and a para-pluriharmonic potential 
$(\eta_\theta,\tau_\theta)$ from $f_1$ and $f_2$, respectively; and 
furthermore, the method enables us to construct pluriharmonic maps and 
para-pluriharmonic maps from their potentials, respectively (see Section 
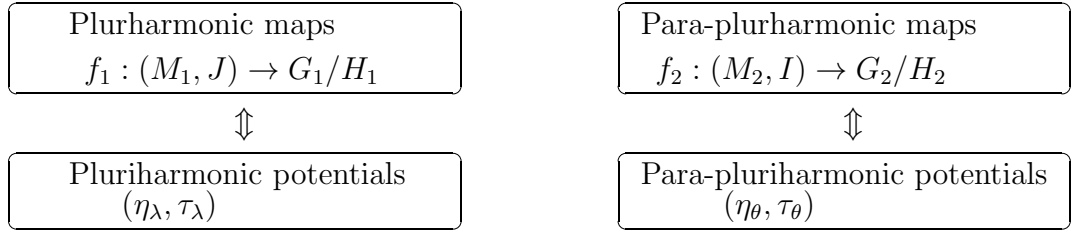
\ref{sec-3}).  
\begin{center}
\maxovaldiam=2mm
\unitlength=1mm
\begin{picture}(136,30)
\put(25,24){\oval(60,12)}
\put(3,26){Plurharmonic maps}
\put(5,20){$f_1:(M_1,J)\to G_1/H_1$}
\put(25,13){$\Updownarrow$}
\put(106,13){$\Updownarrow$}
\put(106,24){\oval(60,12)}
\put(79,26){Para-plurharmonic maps}
\put(81,20){$f_2:(M_2,I)\to G_2/H_2$}
\put(25,5){\oval(60,10)}
\put(3,6){Pluriharmonic potentials}
\put(10,2){$(\eta_\lambda,\tau_\lambda)$}
\put(106,5){\oval(60,10)}
\put(79,6){Para-pluriharmonic potentials}
\put(90,2){$(\eta_\theta,\tau_\theta)$}
\end{picture}   
\end{center} 
   The goal of this paper is to interrelate $f_1:(M_1,J)\to G_1/H_1$ with 
$f_2:(M_2,I)\to G_2/H_2$ by interrelating $(\eta_\lambda,\tau_\lambda)$ with 
$(\eta_\theta,\tau_\theta)$. 
   In this paper, we demonstrate that one can indeed locally interrelate a 
pluriharmonic map with a para-pluriharmonic map in the case where its potential 
satisfies the {\it morphing condition} \eqref{M} (see Theorem \ref{thm-4.3.1}).\par 

   The notions of a pluriharmonic map and a para-pluriharmonic map are  
generalized notions of a harmonic map from a Riemann surface $\Sigma^2$ and 
a Lorentz harmonic map from a Lorentz surface $\Sigma^2_1$, respectively. 
   Consequently, Theorem \ref{thm-4.3.1} enables us to interrelate harmonic maps 
from $\Sigma^2$ with Lorentz harmonic maps from $\Sigma^2_1$. 
   Harmonic maps $f_1$ from $\Sigma^2$ or Lorentz harmonic maps $f_2$ from 
$\Sigma^2_1$ into $S^2$, $H^2$ or $S^2_1$ give rise to constant mean curvature 
surfaces (CMC-surfaces, for short) in $\mathbb{R}^3$, spacelike CMC-surfaces 
in $\mathbb{R}^3_1$ or timelike CMC-surfaces in $\mathbb{R}^3_1$; and vice 
versa.  
   For this reason, one can interrelate CMC-surfaces in $\mathbb{R}^3$ or 
$\mathbb{R}^3_1$ with other CMC-surfaces in $\mathbb{R}^3$ or $\mathbb{R}^3_1$ 
by means of Theorem \ref{thm-4.3.1}. 
   In the appendix, we present concrete examples of the method developed in 
this paper; and moreover, we investigate the relation among CMC-surfaces by use 
of such maps.\par

   This paper is organized as follows: 
   In Section \ref{sec-2} we recall the basic definitions and results 
concerning para-complex manifolds, para-pluriharmonic maps and pluriharmonic 
maps. 
   In Section \ref{sec-3} we review elementary facts and results about the 
loop group method; and we study the relation between para-pluriharmonic or 
pluriharmonic maps and loop groups. 
   In Section \ref{sec-4} we prove the main Theorem \ref{thm-4.3.1}. 
   Finally, in Section \ref{sec-5} we actually interrelate some pluriharmonic 
maps with para-pluriharmonic maps by means of Theorem \ref{thm-4.3.1}. 
\begin{acknowledgments}
   Many thanks are due to the members of GeometrieWerkstatt at the 
Universit\"{a}t T\"{u}bingen. 
   The first named author would like to express his sincere gratitude to 
Wayne Rossman, Hui Ma, Yoshihiro Ohnita, and David Brander for their 
encouragement; and he is grateful to Lars Sch\"{a}fer for his valuable advice.    
\end{acknowledgments}

\section{Pluriharmonic maps and para-pluriharmonic maps}\label{sec-2} 

\subsection{Para-complex manifolds}\label{subsec-2.1} 
   We first recall the notion of a para-complex manifold, in order to 
introduce the notion of a para-pluriharmonic map. 
\begin{definition}[{cf.\ Libermann \cite{Li1}, \cite[p.\ 82, p.\ 83]{Li2}}]
\label{def-2.1.1}\quad\par 
   (i) Let $M$ be a $2n$-dimensional real smooth manifold, and let $\frak{X}M$ 
denote the Lie algebra of smooth vector fields on $M$.  
   Then $M$ is called a {\it para-complex manifold}, if there exists a smooth 
$(1,1)$-tensor field $I$ on $M$ such that    
\begin{enumerate}
\item 
  $I^2=\operatorname{id}$; 
\item 
  $\dim_\mathbb{R}T_p^+M=n=\dim_\mathbb{R}T_p^-M$ for each 
$p\in M$; 
\item 
  $[IX,IY]-I[IX,Y]-I[X,IY]+[X,Y]=0$ for any 
$X,Y\in\frak{X}M$,  
\end{enumerate} 
where $T^\pm_pM$ denotes the $\pm$-eigenspace of $I_p$ ($=$ the value of $I$ 
at $p$) in $T_pM$.\par 
   (ii) Let $(M,I)$ and $(M',I')$ be two para-complex manifolds. 
   Then a smooth map $f:(M,I)\to(M',I')$ is called {\it para-holomorphic} 
(resp.\ {\it para-antiholomorphic}), if it satisfies $df\circ I=I'\circ df$ 
(resp.\ $df\circ I=-I'\circ df$).   
\end{definition}   

   Every para-complex manifold can be endowed with a set of special, local 
coordinates $(x_\alpha^1,\cdots,x_\alpha^n,y_\alpha^1,\cdots,y_\alpha^n)$ 
which are called {\it para-holomorphic coordinates}: 
\begin{proposition}[{cf.\ Kaneyuki-Kozai \cite[p.\ 83]{Ka-Ko}}]
\label{prop-2.1.2} 
   Let $(M,I)$ be a para-complex manifold with $\dim_\mathbb{R}M=2n$. 
   Then, $M$ has an atlas $\{(U_\alpha,\varphi_\alpha)\}_{\alpha\in A}$ with 
$U_\alpha$ open and 
$\varphi_\alpha=(x_\alpha^1,\cdots,x_\alpha^n,y_\alpha^1,\cdots,y_\alpha^n)$
a coordinate map satisfying  
\begin{enumerate}
\item[{\rm (1)}] 
 $I(\partial/\partial x_\alpha^a)=\partial/\partial x_\alpha^a$ 
and 
 $I(\partial/\partial y_\alpha^a)=-\partial/\partial y_\alpha^a$ 
for all $1\leq a\leq n;$
\item[{\rm (2)}]
 $\partial y_\beta^b/\partial x_\alpha^a
  =0=\partial x_\beta^b/\partial y_\alpha^a$ 
on $U_\alpha\cap U_\beta\neq\emptyset$ for all $1\leq a,b\leq n$. 
\end{enumerate}  
\end{proposition}  

   A Lorentz surface and a one sheeted hyperboloid are one of the examples of 
para-complex manifold.

\subsection{Para-pluriharmonic maps}\label{subsec-2.2}
\subsubsection{}\label{subsec-2.2.1} 
   Now, let us recall the notion of a para-pluriharmonic map: 
\begin{definition}[{cf.\ Sch\"{a}fer \cite[p.\ 72]{Sc1}}]
\label{def-2.2.1}
  Let $(M,I)$ be a para-complex manifold with $\dim_\mathbb{R}M=2n$, and let 
$N$ be a smooth manifold with a torsion-free affine connection $\nabla^N$. 
  Then a smooth map $f:(M,I)\to(N,\nabla^N)$ is called {\it para-pluriharmonic}, 
if it satisfies   
\begin{equation}\label{P}\tag{P}
  (\nabla df)\bigl(\frac{\partial}{\partial y^a},
                   \frac{\partial}{\partial x^b}\bigr)=0 
\quad \mbox{for all $1\leq a,b\leq n$},     
\end{equation}
for any local para-holomorphic coordinate $(x^1,\cdots,x^n,y^1,\cdots,y^n)$ on 
$(M,I)$.   
   Here $\nabla$ denotes the connection on $\operatorname{End}(TM,f^{-1}TN)$ 
which is induced from $D$ and $\nabla^N$, where $D$ is any para-complex (i.e., 
$DI=0$) torsion-free affine connection on $(M,I)$.    
\end{definition}

\begin{remark}\label{rem-2.2.2} 
  Every para-complex manifold admits a para-complex torsion-free affine 
connection (cf.\ \cite[p.\ 64]{Sc1}). 
\end{remark}

   The following lemma implies that the equation \eqref{P} in Definition 
\ref{def-2.2.1} is independent of the choice of para-complex torsion-free 
affine connections on $(M,I)$: 
\begin{lemma}\label{lem-2.2.3}
   Let $(M,I)$ be a para-complex manifold with $\dim_\mathbb{R}M=2n$, and let 
$D$ be any para-complex torsion-free affine connection on $(M,I)$.
   Then, every local para-holomorphic coordinate 
$(x^1,\cdots,x^n,y^1,\cdots,y^n)$ on $(M,I)$ satisfies 
$D_{\partial_+^a}\partial_-^b=0=D_{\partial_-^a}\partial_+^b$ for all 
$1\leq a,b\leq n$. 
   Here, $\partial_+^a:=\partial/\partial x^a$ and 
$\partial_-^a:=\partial/\partial y^a$.  
\end{lemma}
\begin{proof}
   It follows from $DI=0$ that for any $1\leq a,b\leq n$,  
\[
 I(D_{\partial_+^a}\partial_-^b)=D_{\partial_+^a}I(\partial_-^b)
  -(D_{\partial_+^a}I)\partial_-^b
 =D_{\partial_+^a}I(\partial_-^b)=-D_{\partial_+^a}\partial_-^b. 
\] 
   This yields $D_{\partial_+^a}\partial_-^b\in T^-M$. 
   Similarly one has $D_{\partial_-^b}\partial_+^a\in T^+M$. 
   Therefore we conclude  
\[
T^-M\ni 
 D_{\partial_+^a}\partial_-^b
=D_{\partial_-^b}\partial_+^a+[\partial_+^a,\partial_-^b]
=D_{\partial_-^b}\partial_+^a
 \in T^+M        
\]
because the torsion of $D$ is free. 
   Thus $D_{\partial_+^a}\partial_-^b=0=D_{\partial_-^b}\partial_+^a$. 
\end{proof}

\subsubsection{}\label{subsec-2.2.2}
   Our goal in this subsection is to show Proposition \ref{prop-2.2.4} (below) 
which will play an important role in Section \ref{sec-3}.   
   First, let us fix the setting and the notation of the proposition.\par    

   Let $G$ be a connected matrix group, and let $\sigma$ be an involution of 
$G$. 
   We denote by $H$ the fixed point set of $\sigma$ in $G$, and get an affine 
symmetric space $(G/H,\sigma)$. 
   Let $(M,I)$ be a para-complex manifold of dimension $2n$, and let $F$ be a 
smooth map from $(M,I)$ into $G$. 
   Then we consider: 
\begin{enumerate}[(2.2.1)] 
\item 
  $\pi$: the projection from $G$ onto $G/H$, 
\item 
  $\nabla^1$: the canonical affine connection on $(G/H,\sigma)$ 
(see \cite[p.\ 54]{No} for the definition of the canonical affine connection),
\item 
  $\alpha:=F^{-1}\cdot dF$: the pullback of the left-invariant Maurer-Cartan 
form on $G$ along $F$, 
\item 
 $\frak{g}:=\operatorname{Lie}G$,\quad   
 $\frak{h}:=\operatorname{Fix}(\frak{g},d\sigma)$,\quad 
 $\frak{m}:=\operatorname{Fix}(\frak{g},-d\sigma)$,
\item 
 $\alpha_\frak{h}$ (resp.\ $\alpha_\frak{m}$): the $\frak{h}$-component 
(resp.\ the $\frak{m}$-component) of $\alpha$ with respect to 
$\frak{g}=\frak{h}\oplus\frak{m}$, 
\item 
 $\alpha_\frak{h}^\pm:=(1/2)\cdot(\alpha_\frak{h}\pm{}^tI(\alpha_\frak{h}))$, 
\quad 
 $\alpha_\frak{m}^\pm:=(1/2)\cdot(\alpha_\frak{m}\pm{}^tI(\alpha_\frak{m}))$,
\item 
 $\partial_-\alpha_\frak{m}^++[\alpha_\frak{h}^-\wedge\alpha_\frak{m}^+]=0$ 
as an abbreviation for  
$\partial_-^a(\alpha_\frak{m}(\partial_+^b))
  +[\alpha_\frak{h}(\partial_-^a),\alpha_\frak{m}(\partial_+^b)]=0$ 
for all $1\leq a,b\leq n$, where $(x^1,\cdots,x^n,y^1,\cdots,y^n)$ is 
any local para-holomorphic coordinate system on $(M,I)$,
\item 
 $T^\pm M$: the subbundle of the tangent bundle $TM$ determined by the 
$\pm 1$-eigenspace of $I$ in $TM$,
\item
 $[\alpha_\frak{m}^\pm\wedge\alpha_\frak{m}^\pm]=0$ as an abbreviation of: 
$[\alpha_\frak{m}\wedge\alpha_\frak{m}]\equiv 0$ on $T^+ M\times T^+ M$ and 
on $T^- M\times T^- M$,
\item
$\mathbb{C}^*:=\mathbb{C}\setminus\{0\}$.         
\end{enumerate}
\setcounter{equation}{10}
   Now, we are in a position to state  
\begin{proposition}\label{prop-2.2.4}
   With the above setting and notation, the following statements {\rm (a)} and 
{\rm (b)} are equivalent$:$ 
\begin{enumerate}
\item[{\rm (a)}] 
   A map $f:=\pi\circ F:(M,I)\to(G/H,\nabla^1)$ is para-pluriharmonic and 
satisfies 
$[\alpha_\frak{m}^+\wedge\alpha_\frak{m}^+]=0
 =[\alpha_\frak{m}^-\wedge\alpha_\frak{m}^-];$ 
\item[{\rm (b)}] 
 $d\alpha^\mu+(1/2)\cdot[\alpha^\mu\wedge\alpha^\mu]=0$ for any 
$\mu\in\mathbb{C}^*$, where  
$\alpha^\mu:=\alpha_\frak{h}+\mu^{-1}\cdot\alpha_\frak{m}^+
 +\mu\cdot\alpha_\frak{m}^-$. 
\end{enumerate}  
\end{proposition}    

   In order to prove the above proposition, we first show 
\begin{lemma}\label{lem-2.2.5}
   $f=\pi\circ F:(M,I)\to(G/H,\nabla^1)$ is a para-pluriharmonic map if and 
only if 
$\partial_-\alpha_\frak{m}^++[\alpha_\frak{h}^-\wedge\alpha_\frak{m}^+]=0$ 
$($cf.\ $(2.2.7))$.    
\end{lemma}
\begin{proof} 
   Lemma \ref{lem-2.2.3} allows us to reduce the equation \eqref{P} in 
Definition \ref{def-2.2.1} as follows: 
$(\nabla df)({\partial_-^a},\partial_+^b)
 =\nabla^1_{\partial_-^a}\bigl(df(\partial_+^b)\bigr)$. 
   This implies that 
\[
\mbox{$f=\pi\circ F$ is para-pluriharmonic if and only if 
$\beta\bigl(\nabla^1_{\partial_-^a}\bigl(df(\partial_+^b)\bigr)\bigr)=0$}
\]
because $\beta:T(G/H)\to G/H\times\frak{g}$ is injective (see \cite{Bu-Ra} or 
\cite[p.\ 403]{Hg} for $\beta$). 
   Accordingly, it suffices to show that 
\begin{equation}\label{eq-2.2.11} 
\mbox{
$\beta\bigl(\nabla^1_{\partial_-^a}\bigl(df(\partial_+^b)\bigr)\bigr)=0$ 
if and only if 
$\partial_-\alpha_\frak{m}^++[\alpha_\frak{h}^-\wedge\alpha_\frak{m}^+]=0$}.
\end{equation}
   To prove this we note first that it is known that $\nabla^1$ coincides with 
the canonical affine connection of the second kind (cf.\ \cite[p.\ 53]{No}). 
   Therefore, Proposition 1.4 and Lemma 1.1 in \cite[p.\ 404, p.\ 403]{Hg} 
assure that 
\[
 \beta\bigl(\nabla^1_{\partial_-^a}\bigl(df(\partial_+^b)\bigr)\bigr) 
=\partial_-^a\bigl(\beta\bigl(df(\partial_+^b)\bigr)\bigr)
 -\big[\beta\bigl(df(\partial_-^a)\bigr),
  \beta\bigl(df(\partial_+^b)\bigr)\big].  
\]
   Let us compute each term on the right-hand side of the above equation. 
   We note that $f^*\beta=\operatorname{Ad}F\cdot\alpha_\frak{m}$ (cf.\ 
\cite[p.\ 409]{Hg}) implies   
\begin{multline*}
\partial_-^a\bigl(\beta\bigl(df(\partial_+^b)\bigr)\bigr)
=\partial_-^a\bigl((f^*\beta)(\partial_+^b)\bigr) 
=\partial_-^a\bigl(F\cdot\alpha_\frak{m}(\partial_+^b)\cdot F^{-1}\bigr)\\
=(\partial_-^a F)\cdot\alpha_\frak{m}(\partial_+^b)\cdot F^{-1}
 +F\cdot\partial_-^a\bigl(\alpha_\frak{m}(\partial_+^b)\bigr)\cdot F^{-1} 
 -F\cdot\alpha_\frak{m}(\partial_+^b)
   \cdot F^{-1}\cdot(\partial_-^a F)\cdot F^{-1}\\
=\operatorname{Ad}F\cdot\left\{
 \partial_-^a\bigl(\alpha_\frak{m}(\partial_+^b)\bigr)
   +\bigl[F^{-1}\cdot(\partial_-^a F),
     \alpha_\frak{m}(\partial_+^b)\bigr] \right\}\\
=\operatorname{Ad}F\cdot
 \left\{\partial_-^a\bigl(\alpha_\frak{m}(\partial_+^b)\bigr)
  +\bigl[\alpha(\partial_-^a),
      \alpha_\frak{m}(\partial_+^b)\bigr]
     \right\}.      
\end{multline*}
   Moreover, $f^*\beta=\operatorname{Ad}F\cdot\alpha_\frak{m}$ yields  
\[
\big[\beta\bigl(df(\partial_-^a)\bigr),\beta\bigl(df(\partial_+^b)\bigr)\big]
=\big[(f^*\beta)(\partial_-^a),(f^*\beta)(\partial_+^b)\big]
=\operatorname{Ad}F\cdot\big[\alpha_\frak{m}(\partial_-^a),
   \alpha_\frak{m}(\partial_+^b)\big]. 
\] 
  Therefore we obtain   
\[
 \beta\bigl(\nabla^1_{\partial_-^a}\bigl(df(\partial_+^b)\bigr)\bigr)
=\operatorname{Ad}F\cdot
 \left\{\partial_-^a\bigl(\alpha_\frak{m}(\partial_+^b)\bigr)
  +\bigl[\alpha_\frak{h}(\partial_-^a),\alpha_\frak{m}(\partial_+^b)\bigr]
     \right\}.
\]
  Hence we have shown \eqref{eq-2.2.11}.      
\end{proof}        

\begin{proof}[Proof of Proposition {\rm \ref{prop-2.2.4}}]
   First we rewrite the expression  
$d\alpha^\mu+(1/2)\cdot[\alpha^\mu\wedge\alpha^\mu]$. 
   Since $\alpha=F^{-1}\cdot dF$ we have  
$d\alpha+(1/2)\cdot[\alpha\wedge\alpha]=0$. 
   From 
$[\frak{h},\frak{h}]\subset\frak{h}$,  $[\frak{h},\frak{m}]\subset\frak{m}$ and 
$[\frak{m},\frak{m}]\subset\frak{h}$ we obtain   
$d\alpha_\frak{h}
 +(1/2)\cdot\bigl([\alpha_\frak{h}\wedge\alpha_\frak{h}]
  +[\alpha_\frak{m}\wedge\alpha_\frak{m}]\bigr)=0=
d\alpha_\frak{m}+[\alpha_\frak{h}\wedge\alpha_\frak{m}]$. 
  Thus we can assert that 
\begin{equation}\label{eq-2.2.12}
\begin{split}
&
d\alpha_\frak{h}
 +\frac{1}{2}\cdot[\alpha_\frak{h}\wedge\alpha_\frak{h}]
  +[\alpha_\frak{m}^+\wedge\alpha_\frak{m}^-]
=-\frac{1}{2}\cdot
   \left([\alpha_\frak{m}^+\wedge\alpha_\frak{m}^+]
    +[\alpha_\frak{m}^-\wedge\alpha_\frak{m}^-]\right),\\
&
d\alpha_\frak{m}+[\alpha_\frak{h}\wedge\alpha_\frak{m}]=0.
\end{split} 
\end{equation}
   By a direct computation we obtain  
\[
\begin{split}
d\alpha^\mu+\frac{1}{2}\cdot[\alpha^\mu\wedge\alpha^\mu]
=& d\alpha_\frak{h}+\frac{1}{2}\cdot[\alpha_\frak{h}\wedge\alpha_\frak{h}]
  +[\alpha_\frak{m}^+\wedge\alpha_\frak{m}^-]\\
&+\mu^{-1}\cdot\left(d\alpha_\frak{m}^+
   +[\alpha_\frak{h}\wedge\alpha_\frak{m}^+]\right)
 +\mu\cdot\left(d\alpha_\frak{m}^-
   +[\alpha_\frak{h}\wedge\alpha_\frak{m}^-]\right)\\
&+\frac{1}{2}\cdot\mu^{-2}\cdot[\alpha_\frak{m}^+\wedge\alpha_\frak{m}^+]
 +\frac{1}{2}\cdot\mu^2\cdot[\alpha_\frak{m}^-\wedge\alpha_\frak{m}^-].    
\end{split}  
\] 
   Consequently, by virtue of \eqref{eq-2.2.12} one can rewrite 
$d\alpha^\mu+(1/2)\cdot[\alpha^\mu\wedge\alpha^\mu]$ as follows: 
\begin{equation}\label{eq-2.2.13}
\begin{split}
d\alpha^\mu+\frac{1}{2}\cdot[\alpha^\mu\wedge\alpha^\mu]
=&\mu^{-1}\cdot\left(d\alpha_\frak{m}^+
   +[\alpha_\frak{h}\wedge\alpha_\frak{m}^+]\right)
 +\mu\cdot\left(d\alpha_\frak{m}^-
   +[\alpha_\frak{h}\wedge\alpha_\frak{m}^-]\right)\\
&+\frac{1}{2}\cdot(\mu^{-2}-1)\cdot[\alpha_\frak{m}^+\wedge\alpha_\frak{m}^+]
 +\frac{1}{2}\cdot(\mu^2-1)\cdot[\alpha_\frak{m}^-\wedge\alpha_\frak{m}^-].
\end{split}
\end{equation}

  (a)$\to$(b): Suppose that $f=\pi\circ F$ is para-pluriharmonic and 
satisfies 
$[\alpha_\frak{m}^+\wedge\alpha_\frak{m}^+]=0
  =[\alpha_\frak{m}^-\wedge\alpha_\frak{m}^-]$. 
   Then, \eqref{eq-2.2.13} yields   
\[
d\alpha^\mu+\frac{1}{2}\cdot[\alpha^\mu\wedge\alpha^\mu]
 =\mu^{-1}\cdot\left(d\alpha_\frak{m}^+
   +[\alpha_\frak{h}\wedge\alpha_\frak{m}^+]\right)
 +\mu\cdot\left(d\alpha_\frak{m}^-
   +[\alpha_\frak{h}\wedge\alpha_\frak{m}^-]\right).  
\]
   So it suffices to show 
\[
d\alpha_\frak{m}^++[\alpha_\frak{h}\wedge\alpha_\frak{m}^+]
 =0= 
  d\alpha_\frak{m}^-+[\alpha_\frak{h}\wedge\alpha_\frak{m}^-].
\] 
   Equation \eqref{P} in Definition \ref{def-2.2.1} is symmetric with respect 
to the variables $y^a$ and $x^b$. 
   This and Lemma \ref{lem-2.2.5} imply that   
\[
\begin{split}
\mbox{$f=\pi\circ F$ is para-pluriharmonic}
&\mbox{ if and only if 
  $\partial_-\alpha_\frak{m}^++[\alpha_\frak{h}^-\wedge\alpha_\frak{m}^+]=0$}\\
&\mbox{ if and only if 
 $\partial_+\alpha_\frak{m}^-+[\alpha_\frak{h}^+\wedge\alpha_\frak{m}^-]=0$}.
\end{split}
\]
  Therefore it follows from 
$d\alpha_\frak{m}+[\alpha_\frak{h}\wedge\alpha_\frak{m}]=0$ (cf.\ 
\eqref{eq-2.2.12}) that 
$d\alpha_\frak{m}^++[\alpha_\frak{h}\wedge\alpha_\frak{m}^+]
 =0=d\alpha_\frak{m}^-+[\alpha_\frak{h}\wedge\alpha_\frak{m}^-]$.
\par 
   
   (b)$\to$(a): Suppose that 
$d\alpha^\mu+(1/2)\cdot[\alpha^\mu\wedge\alpha^\mu]=0$ for any 
$\mu\in\mathbb{C}^*$. 
   We obtain $d\alpha_\frak{m}^++[\alpha_\frak{h}\wedge\alpha_\frak{m}^+]=0$ 
and    
$[\alpha_\frak{m}^+\wedge\alpha_\frak{m}^+]
 =0=[\alpha_\frak{m}^-\wedge\alpha_\frak{m}^-]$ 
from \eqref{eq-2.2.13}. 
   So Lemma \ref{lem-2.2.5} allows us to obtain the conclusion, if one has 
$\partial_-\alpha_\frak{m}^++[\alpha_\frak{h}^-\wedge\alpha_\frak{m}^+]=0$. 
   But, this equation is immediate from 
$d\alpha_\frak{m}^++[\alpha_\frak{h}\wedge\alpha_\frak{m}^+]=0$. 
\end{proof}

\subsubsection{}\label{subsec-2.2.3}
   We recall the notion of the extended framing of a para-pluriharmonic map 
(cf.\ Definition \ref{def-2.2.6}). 
   One will see that the framing is an element of the loop group 
$\widetilde{\Lambda}G_\sigma$ in Section \ref{sec-3}.  
 
   Let $G^\mathbb{C}$ be a simply connected, simple, complex linear algebraic 
subgroup of $SL(m,\mathbb{C})$, let $\sigma$ be a holomorphic involution of 
$G^\mathbb{C}$, and let $\nu$ be an antiholomorphic involution of 
$G^\mathbb{C}$ such that $[\sigma,\nu]=0$ (i.e., 
$\sigma\circ\nu=\nu\circ\sigma$). 
   Define $H^\mathbb{C}$, $G$ and $H$ by 
\begin{equation}\label{eq-2.2.14}
\begin{array}{lll} 
   H^\mathbb{C}:=\operatorname{Fix}(G^\mathbb{C},\sigma), 
&  G:=\operatorname{Fix}(G^\mathbb{C},\nu), 
&  H:=\operatorname{Fix}(G,\sigma)=\operatorname{Fix}(H^\mathbb{C},\nu). 
\end{array}  
\end{equation}
   Note that $(G/H,\sigma|_G)$ is an affine symmetric space. 
   Now, let $p_o$ be a base point in a simply connected para-complex manifold 
$(M,I)$. 
   Then, Proposition \ref{prop-2.2.4} assures that for any para-pluriharmonic 
map $f=\pi\circ F:(M,I)\to (G/H,\nabla^1)$ with $F(p_o)=\operatorname{id}$ and 
$[\alpha_\frak{m}^\pm\wedge\alpha_\frak{m}^\pm]=0$, the 
$\frak{g}^\mathbb{C}$-valued $1$-form 
$\alpha^\mu=\alpha_\frak{h}+\mu^{-1}\cdot\alpha_\frak{m}^+
 +\mu\cdot\alpha_\frak{m}^-$ 
on $(M,I)$, parameterized by $\mu\in\mathbb{C}^*$, is integrable; and 
furthermore, one can obtain a smooth map 
\[
\begin{array}{ll}
 F:M\times\mathbb{C}^*\to G^\mathbb{C}, & (p,\mu)\mapsto F_\mu(p),
\end{array}
\]
from the integrability condition 
$d\alpha^\mu+(1/2)\cdot[\alpha^\mu\wedge\alpha^\mu]=0$ and 
$F_\mu^{-1}\cdot dF_\mu=\alpha^\mu$ with $F_\mu(p_o)\equiv\operatorname{id}$. 
   The above map $F=F_\mu:\mathbb{C}^*\to G^\mathbb{C}$ satisfies 
\begin{enumerate} 
\item[(2.2.15)] 
   $\sigma(F_\mu)=F_{-\mu}$ for all $\mu\in\mathbb{C}^*$, 
\item[(2.2.16)] 
   $F_\lambda:=F|_{S^1}:S^1\to G^\mathbb{C}$, where 
$S^1:=\{\lambda\in\mathbb{C}^*\,|\,|\lambda|=1\}$, 
\item[(2.2.17)] 
   $F_\theta:=F|_{\mathbb{R}^+}:\mathbb{R}^+\to 
G=\operatorname{Fix}(G^\mathbb{C},\nu)$ ($\subset G^\mathbb{C}$), 
where $\mathbb{R}^+:=\{\theta\in\mathbb{R}\,|\,\theta>0\}$.     
\end{enumerate} 
\setcounter{equation}{17}
   Indeed, (2.2.15) follows from $d\sigma(\alpha^\mu)=\alpha^{-\mu}$ and 
$\sigma(F_\mu(p_o))=F_{-\mu}(p_o)$; (2.2.16) is obvious; and (2.2.17) follows 
from $\alpha^\theta$ being $\frak{g}$-valued for any $\theta\in\mathbb{R}^+$. 

\begin{definition}\label{def-2.2.6} 
    The map $F_\theta$ is called the {\it extended framing} of the 
para-pluriharmonic map $f=\pi\circ F:(M,I)\to(G/H,\nabla^1)$; and 
$\{f_\theta\}_{\theta\in\mathbb{R}^+}$ is called an {\it associated family} of 
$f$, where $f_\theta:=\pi\circ F_\theta$.   
   Here, we remark that $f_1=f$ and $F_1=F$ are immediate from 
$\alpha^1=\alpha$ and $F_1(p_o)=F(p_o)$.    
\end{definition}
   
\begin{remark}\label{rem-2.2.7}
   Throughout this paper we consider that for the extended framing $F_\theta$ 
of a para-pluriharmonic map, its variable $\theta$ varies in the whole 
$\mathbb{C}^*$ which contains not only $\mathbb{R}^+$ but also $S^1$.       
\end{remark}

\subsection{Pluriharmonic maps}\label{subsec-2.3} 
\subsubsection{}\label{subsec-2.3.1} 
   In this subsection we will survey some basic facts and results about 
pluriharmonic maps. 
   First, let us recall the notion of a pluriharmonic map:  
\begin{definition}\label{def-2.3.1}
   Let $(M,J)$ be a real $2n$-dimensional complex manifold, and let $N$ be a 
smooth manifold with a torsion-free affine connection $\nabla^N$. 
   Then a smooth map $f:(M,J)\to(N,\nabla^N)$ is called {\it pluriharmonic}, 
if it satisfies 
\begin{equation}\label{H}\tag{H}
  (\nabla df)\bigl(\frac{\partial}{\partial \bar{z}^a}, 
   \frac{\partial}{\partial z^b}\bigr)=0 
\quad \mbox{for all $1\leq a,b\leq n$},     
\end{equation}
for any local holomorphic coordinate 
$(z^1,\cdots,z^n,\bar{z}^1,\cdots,\bar{z}^n)$ on $(M,J)$.    
   Here $\nabla$ denotes the connection on $\operatorname{End}(TM,f^{-1}TN)$ 
which is induced from $D$ and $\nabla^N$, where $D$ is any complex torsion-free 
affine connection on $(M,J)$. 
\end{definition} 

\begin{remark}\label{rem-2.3.2}
   (i) We utilize the terminology ``pluriharmonic map,'' in a sense that is 
more general than the one originally given by Siu \cite{Si}. \par 
   (ii) Any complex manifold admits a complex torsion-free affine connection 
(cf.\ \cite[p.\ 145]{Ko-No}).\par 
   (iii) The equation \eqref{H} in Definition \ref{def-2.3.1} is independent 
of the choice of complex torsion-free affine connections $D$ on $(M,J)$ (ref.\ 
the proof of Lemma \ref{lem-2.2.3}).        
\end{remark}

\subsubsection{}\label{subsec-2.3.2} 
   In Section \ref{sec-3} we will study the relation between pluriharmonic maps 
and the loop group method. 
   For this we will use a result of Ohnita \cite{Oh} about pluriharmonic maps   
(see Proposition \ref{prop-2.3.3}). 
   First, let us fix the setting for Proposition \ref{prop-2.3.3}.\par 
   
   Let $(G/H,\sigma)$ denote the affine symmetric space defined in 
Subsection \ref{subsec-2.2.2}, and let $F$ be a smooth map from a real 
$2n$-dimensional complex manifold $(M,J)$ into $G$. 
   Then, we consider:      
\begin{enumerate}[(2.3.1)] 
\item 
  $\pi$: the same as in (2.2.1),  
\item 
  $\nabla^1$: the same as in (2.2.2),
\item 
  $\alpha$: the same as in (2.2.3), 
\item 
 $\frak{g}$, $\frak{h}$, $\frak{m}$: the same as in (2.2.4),
\item 
 $\alpha_\frak{h}$, $\alpha_\frak{m}$: the same as in (2.2.5),  
\item 
 $\alpha_X':=(-i/2)\cdot(i\alpha_X+{}^tJ(\alpha_X))$,
\quad 
 $\alpha_X'':=(-i/2)\cdot(i\alpha_X-{}^tJ(\alpha_X))$ for 
$X=\frak{h}$, $\frak{m}$,  
\item 
 $\overline{\partial}\alpha_\frak{m}'
 +[\alpha_\frak{h}''\wedge\alpha_\frak{m}']=0$ 
as an abbreviation for  
$\overline{\partial}^a(\alpha_\frak{m}(\partial^b))
  +[\alpha_\frak{h}(\overline{\partial}^a),\alpha_\frak{m}(\partial^b)]=0$ 
for all $1\leq a,b\leq n$, where $(z^1,\cdots,z^n,\bar{z}^1,\cdots,\bar{z}^n)$ 
is any local holomorphic coordinate system on $(M,J)$, and where 
$\partial^b:=\partial/\partial z^b$ and 
$\overline{\partial}^a:=\partial/\partial\bar{z}^a$, 
\item
 $[\alpha_\frak{m}'\wedge\alpha_\frak{m}']=0$ as an abbreviation for 
$[\alpha_\frak{m}(\partial^a),\alpha_\frak{m}(\partial^b)]=0$ 
for all $1\leq a,b\leq n$, where $(z^1,\cdots,z^n,\bar{z}^1,\cdots,\bar{z}^n)$ 
is any local holomorphic coordinate system on $(M,J)$.         
\end{enumerate}
\setcounter{equation}{8}    

\begin{proposition}[{cf.\ Ohnita \cite{Oh}}]\label{prop-2.3.3}
   With the above notation,  a map $f:=\pi\circ F:(M,J)\to(G/H,\nabla^1)$ 
is pluriharmonic if and only if 
$\overline{\partial}\alpha_\frak{m}'
 +[\alpha_\frak{h}''\wedge\alpha_\frak{m}']=0$.  
   Moreover, the following statements {\rm (a)} and {\rm (b)} are equivalent$:$ 
\begin{enumerate}
\item[{\rm (a)}] 
   $f=\pi\circ F$ is pluriharmonic and satisfies 
$[\alpha_\frak{m}'\wedge\alpha_\frak{m}']=0;$ 
\item[{\rm (b)}] 
 $d\alpha^\mu+(1/2)\cdot[\alpha^\mu\wedge\alpha^\mu]=0$ 
for any $\mu\in\mathbb{C}^*$, where 
$\alpha^\mu:=\alpha_\frak{h}+\mu^{-1}\cdot\alpha_\frak{m}'
  +\mu\cdot\alpha_\frak{m}''$. 
\end{enumerate}     
\end{proposition} 

\subsubsection{}\label{subsec-2.3.3} 
   We will first recall the notion of the extended framing of a pluriharmonic 
map (cf.\ Definition \ref{def-2.3.4}), and afterwards point out a crucial 
difference between the extended framings of pluriharmonic maps and 
para-pluriharmonic maps in view of the loop group method (cf.\ Remark 
\ref{rem-2.3.5}).\par

   The arguments below will be similar to those in Subsection 
\ref{subsec-2.2.3}. 
   Let $G^\mathbb{C}$, $H^\mathbb{C}$, $G$ and $H$ denote the same Lie groups 
as in \eqref{eq-2.2.14}. 
   Fix a base point $p_o$ in a simply connected complex manifold $(M,J)$. 
   For a pluriharmonic map $f=\pi\circ F:(M,J)\to(G/H,\nabla^1)$ with 
$F(p_o)=\operatorname{id}$ and $[\alpha_\frak{m}'\wedge\alpha_\frak{m}']=0$, 
Proposition \ref{prop-2.3.3} shows that the $\frak{g}^\mathbb{C}$-valued 
$1$-form 
$\alpha^\mu=\alpha_\frak{h}+\mu^{-1}\cdot\alpha_\frak{m}'
  +\mu\cdot\alpha_\frak{m}''$ 
on $(M,J)$ parameterized by $\mu\in\mathbb{C}^*$ is integrable. 
   Then there exists a unique map 
\[ 
\begin{array}{ll} 
  F:M\times\mathbb{C}^*\to G^\mathbb{C},  
& (p,\mu)\mapsto F_\mu(p),
\end{array}
\]
such that $F_\mu^{-1}\cdot dF_\mu=\alpha^\mu$ and 
$F_\mu(p_o)\equiv\operatorname{id}$, by virtue of the integrability condition 
$d\alpha^\mu+(1/2)\cdot[\alpha^\mu\wedge\alpha^\mu]=0$.   
   Here we remark that $F=F_\mu:\mathbb{C}^*\to G^\mathbb{C}$ satisfies 
\begin{enumerate}
\item[(2.3.9)] 
   $\sigma(F_\mu)=F_{-\mu}$ for all $\mu\in\mathbb{C}^*$, 
\item[(2.3.10)]  
   $F_\lambda:=F|_{S^1}:S^1\to G=\operatorname{Fix}(G^\mathbb{C},\nu)$ 
($\subset G^\mathbb{C}$).     
\end{enumerate} 
\setcounter{equation}{10}
   Indeed, (2.3.10) follows from $\alpha^\lambda$ being $\frak{g}$-valued for 
any $\lambda\in S^1$.
\begin{definition}\label{def-2.3.4} 
    The map $F_\lambda$ is called the {\it extended framing} of the 
pluriharmonic map $f=\pi\circ F:(M,J)\to(G/H,\nabla^1)$; and 
$\{f_\lambda\}_{\lambda\in S^1}$ is called an {\it associated family} of $f$, 
where $f_\lambda(p):=\pi\circ F_\lambda(p)$ for $(p,\lambda)\in M\times S^1$.    
\end{definition}

\begin{remark}\label{rem-2.3.5}
   The map $F=F_\mu:\mathbb{C}^*\to G^\mathbb{C}$ defined above becomes 
$G$-valued if its variable $\mu$ varies in $S^1$; and $F_\lambda=F|_{S^1}$ is 
the extended framing of a pluriharmonic map. 
   By contrast, the map $F=F_\mu:\mathbb{C}^*\to G^\mathbb{C}$ in 
Subsection \ref{subsec-2.2.3} becomes $G$-valued if its variable $\mu$ varies 
in $\mathbb{R}^+$; and $F_\theta=F|_{\mathbb{R}^+}$ is the extended framing of 
a para-pluriharmonic map.     
\end{remark}

\section{The loop group method}\label{sec-3}
   First, we introduce three kinds of loop groups $\Lambda G^\mathbb{C}_\sigma$, 
$\Lambda G_\sigma$ and $\widetilde{\Lambda} G_\sigma$, and review their 
decomposition theorems. 
   Next, we explain the relation between para-pluriharmonic maps and the loop 
group method, and interrelate para-pluriharmonic maps with para-pluriharmonic 
potentials. 
   Finally, we treat the pluriharmonic case.

\subsection{Decomposition theorems of loop groups}\label{subsec-3.1}
\subsubsection{}\label{subsec-3.1.1}  
   Let $G^\mathbb{C}$ be a simply connected, simple, complex linear algebraic 
subgroup of $SL(m,\mathbb{C})$, and let $\sigma$ be a holomorphic involution 
of $G^\mathbb{C}$.  
   In this case the twisted loop group $\Lambda G^\mathbb{C}_\sigma$ is 
defined as follows: 
\[
\Lambda G^\mathbb{C}_\sigma:=\left\{
  \begin{array}{@{\,}l|r@{\,}}  
  A_\lambda:S^1\to G^\mathbb{C} 
  & A_\lambda=\sum_{k\in\mathbb{Z}}A_k\lambda^k, \, 
    \sum||A_k||<\infty,\\
  & \mbox{$\sigma(A_\lambda)=A_{-\lambda}$ for all 
          $\lambda\in S^1$} 
  \end{array}\right\}, 
\]
where $||\cdot||$ denotes some matrix norm satisfying 
$||A\cdot B||\leq||A||\cdot||B||$ and $||\operatorname{id}||=1$. 
   Then $\Lambda G^\mathbb{C}_\sigma$, with this norm 
$||A_\lambda||=\sum||A_k||$,  is a complex Banach Lie group (see \cite{Ba-Do}, 
\cite{Go-Wa} and \cite{Pr-Se} for more details). 
  Here, the Lie algebra $\Lambda\frak{g}^\mathbb{C}_\sigma$ of 
$\Lambda G^\mathbb{C}_\sigma$ is given by 
\begin{equation}\label{eq-3.1.1}
\Lambda\frak{g}^\mathbb{C}_\sigma
:=\left\{
  \begin{array}{@{\,}l|r@{\,}}  
  X_\lambda:S^1\to\frak{g}^\mathbb{C} 
  & X_\lambda=\sum_{k\in\mathbb{Z}}X_k\lambda^k, \, 
    \sum||X_k||<\infty,\\
  & \mbox{$d\sigma(X_\lambda)=X_{-\lambda}$ for all 
          $\lambda\in S^1$} 
  \end{array}\right\}.   
\end{equation}
   Define four subgroups $\Lambda^\pm G^\mathbb{C}_\sigma$ and 
$\Lambda^\pm_* G^\mathbb{C}_\sigma$ of $\Lambda G^\mathbb{C}_\sigma$ by 
\[
\begin{array}{l}
\Lambda^\pm G^\mathbb{C}_\sigma
 :=\{ A_\lambda\in\Lambda G^\mathbb{C}_\sigma \,|\, 
     \mbox{$A_\lambda$ has a holomorphic extension 
      $\widehat{A}_z:\mathbb{D}_\pm\to G^\mathbb{C}$} \},\\
\Lambda^+_* G^\mathbb{C}_\sigma
  :=\{ A_\lambda\in\Lambda^+ G^\mathbb{C}_\sigma \,|\, 
       \widehat{A}_0=\operatorname{id} \},\quad  
\Lambda^-_* G^\mathbb{C}_\sigma
  :=\{ A_\lambda\in\Lambda^- G^\mathbb{C}_\sigma \,|\, 
       \widehat{A}_\infty=\operatorname{id} \},        
\end{array}
\]
where $\mathbb{D}_+:=\{z\in\mathbb{C}\,|\,|z|<1\}$ and 
$\mathbb{D}_-:=\{z\in\mathbb{C}\,|\,|z|>1\}\cup\{\infty\}$.   
      With this notation, we can state the following two Theorems 
\ref{thm-3.1.1} and \ref{thm-3.1.2}, which are called the {\it Iwasawa 
decomposition} of 
$\Lambda G^\mathbb{C}_\sigma\times\Lambda G^\mathbb{C}_\sigma$ and the 
{\it Birkhoff decomposition} of $\Lambda G^\mathbb{C}_\sigma$, respectively 
(see \cite{Ba-Do}, \cite{Go-Wa}, \cite{Pr-Se}): 
\begin{theorem}[Iwasawa decomposition of 
$\Lambda G^\mathbb{C}_\sigma\times\Lambda G^\mathbb{C}_\sigma$]
\label{thm-3.1.1} 
   The multiplication maps 
\[
\begin{array}{l}
 \triangle(\Lambda G^\mathbb{C}_\sigma
   \times\Lambda G^\mathbb{C}_\sigma)
 \times 
 (\Lambda^-_* G^\mathbb{C}_\sigma\times
  \Lambda^+ G^\mathbb{C}_\sigma)  
 \to 
 \Lambda G^\mathbb{C}_\sigma
  \times\Lambda G^\mathbb{C}_\sigma,\\
 \triangle(\Lambda G^\mathbb{C}_\sigma
   \times\Lambda G^\mathbb{C}_\sigma)
 \times 
 (\Lambda^+_* G^\mathbb{C}_\sigma\times
  \Lambda^- G^\mathbb{C}_\sigma)  
 \to 
 \Lambda G^\mathbb{C}_\sigma\times\Lambda G^\mathbb{C}_\sigma
\end{array} 
\] 
are holomorphic diffeomorphisms onto open subsets of 
$\Lambda G^\mathbb{C}_\sigma\times\Lambda G^\mathbb{C}_\sigma$, 
respectively. 
   Here 
$\triangle(\Lambda G^\mathbb{C}_\sigma\times\Lambda G^\mathbb{C}_\sigma)$ 
denotes the diagonal subgroup of 
$\Lambda G^\mathbb{C}_\sigma\times\Lambda G^\mathbb{C}_\sigma$.   
\end{theorem}

\begin{theorem}[Birkhoff decomposition of $\Lambda G^\mathbb{C}_\sigma$]
\label{thm-3.1.2}  
   The multiplication maps 
\[
\begin{array}{ll}
 \Lambda^-_*G^\mathbb{C}_\sigma\times\Lambda^+G^\mathbb{C}_\sigma 
 \to\Lambda G^\mathbb{C}_\sigma, 
& 
 \Lambda^+_*G^\mathbb{C}_\sigma\times\Lambda^-G^\mathbb{C}_\sigma 
 \to\Lambda G^\mathbb{C}_\sigma
\end{array}
\] 
are holomorphic diffeomorphisms onto the open subsets 
$\mathcal{B}_\mp^\mathbb{C}:=
 \Lambda^\mp_*G^\mathbb{C}_\sigma\cdot\Lambda^\pm G^\mathbb{C}_\sigma$ 
of $\Lambda G^\mathbb{C}_\sigma$, respectively. 
   In particular, each element 
$A_\lambda\in\mathcal{B}^\mathbb{C}:=\mathcal{B}_-^\mathbb{C}
 \cap\mathcal{B}_+^\mathbb{C}$ 
can be uniquely factorized$:$  
\[
\begin{array}{lrr}
  A_\lambda=A^-_\lambda\cdot B^+_\lambda=A^+_\lambda\cdot B^-_\lambda, 
& A^\pm_\lambda\in\Lambda^\pm_* G^\mathbb{C}_\sigma, 
& B^\pm_\lambda\in\Lambda^\pm G^\mathbb{C}_\sigma.   
\end{array}
\]    
\end{theorem}

\subsubsection{Almost split real forms of $\Lambda G^\mathbb{C}_\sigma$}
\label{subsec-3.1.2} 
   Now, let $\nu$ be an antiholomorphic involution of $G^\mathbb{C}$ such that 
$[\sigma,\nu]=0$ (i.e., $\sigma\circ\nu=\nu\circ\sigma$). 
   Then one can define an antiholomorphic involution $\nu_S$ of 
$\Lambda G^\mathbb{C}_\sigma$ by setting 
\begin{equation}\label{eq-3.1.2}
\begin{array}{ll}
  \nu_S(A_\lambda):=\nu(A_{\overline{\lambda}}) 
& \mbox{for $A_\lambda\in\Lambda G^\mathbb{C}_\sigma$}. 
\end{array}   
\end{equation} 
   This involution $\nu_S$ is said to be of {\it the first kind}, and its 
fixed point set 
$\Lambda G_\sigma:=\operatorname{Fix}(\Lambda G^\mathbb{C}_\sigma,\nu_S)$ 
is called an {\it almost split real form} of $\Lambda G^\mathbb{C}_\sigma$. 
   Note that $\nu_S$ satisfies 
$\nu_S(\Lambda^\pm G^\mathbb{C}_\sigma)=\Lambda^\pm G^\mathbb{C}_\sigma$ and 
$\nu_S(\Lambda^\pm_* G^\mathbb{C}_\sigma)=\Lambda^\pm_* G^\mathbb{C}_\sigma$. 
   That allows us to define four subgroups $\Lambda^\pm G_\sigma$ and 
$\Lambda^\pm_*G_\sigma$ as follows:  
\[
\begin{array}{ll}
\Lambda^\pm G_\sigma
 :=\operatorname{Fix}(\Lambda^\pm G^\mathbb{C}_\sigma,\nu_S),  
&
\Lambda^\pm_* G_\sigma
 :=\operatorname{Fix}(\Lambda^\pm_* G^\mathbb{C}_\sigma,\nu_S). 
\end{array}
\]   
   With this notation, one can state the following theorems (see \cite{Br}, 
\cite{Br-Do}): 
\begin{theorem}[Iwasawa decomposition of 
$\Lambda G_\sigma\times\Lambda G_\sigma$]\label{thm-3.1.3} 
   The multiplication maps 
\[
\begin{array}{l}
 \triangle(\Lambda G_\sigma\times\Lambda G_\sigma)
 \times(\Lambda^-_* G_\sigma\times\Lambda^+ G_\sigma)  
 \to 
 \Lambda G_\sigma\times\Lambda G_\sigma,\\
 \triangle(\Lambda G_\sigma\times\Lambda G_\sigma)
 \times(\Lambda ^+_* G_\sigma\times\Lambda^- G_\sigma)  
 \to 
 \Lambda G_\sigma\times\Lambda G_\sigma
\end{array} 
\] 
are holomorphic diffeomorphisms onto open subsets of 
$\Lambda G_\sigma\times\Lambda G_\sigma$, respectively.    
\end{theorem}

\begin{theorem}[Birkhoff decomposition of $\Lambda G_\sigma$]\label{thm-3.1.4}  
   The multiplication maps 
\[
\begin{array}{ll}
 \Lambda^-_*G_\sigma\times\Lambda^+G_\sigma\to\Lambda G_\sigma, 
& 
 \Lambda^+_*G_\sigma\times\Lambda^-G_\sigma\to\Lambda G_\sigma
\end{array}
\] 
are holomorphic diffeomorphisms onto the open subsets 
$\mathcal{B}_\mp:=\Lambda^\mp_*G_\sigma\cdot\Lambda^\pm G_\sigma$ 
of $\Lambda G_\sigma$, respectively. 
   In particular, each element 
$A_\lambda\in\mathcal{B}:=\mathcal{B}_-\cap\mathcal{B}_+$ can be uniquely 
factorized$:$  
\[
\begin{array}{lrr}
  A_\lambda=A^-_\lambda\cdot B^+_\lambda=A^+_\lambda\cdot B^-_\lambda, 
& A^\pm_\lambda\in\Lambda^\pm_* G_\sigma, 
& B^\pm_\lambda\in\Lambda^\pm G_\sigma.   
\end{array}
\]    
\end{theorem}

\subsubsection{}\label{subsec-3.1.3}
   For a general element $A_\lambda\in\Lambda G_\sigma^\mathbb{C}$, its 
variable $\lambda$ only varies in $S^1$. 
   However, for the framing $F_\lambda$ of a para-pluriharmonic map, the 
variable $\lambda$ of $F_\lambda$ can vary in the whole $\mathbb{C}^*$ 
(cf.\ Subsection \ref{subsec-2.2.3}). 
   Toda \cite{To} has addressed this relevant point, since in her work 
$\lambda$ is for all geometric purposes a positive real number. 
   She proposed to consider the following subgroup 
$\widetilde{\Lambda}G_\sigma$ of $\Lambda G_\sigma$:  
\begin{equation}\label{eq-3.1.3}
  \widetilde{\Lambda}G_\sigma
 :=\{A_\lambda\in\Lambda G_\sigma \,|\, 
     \mbox{$A_\lambda$ has an analytic extension 
      $\widetilde{A}_\mu:\mathbb{C}^*\to G^\mathbb{C}$} \}.
\end{equation}
   One equips $\widetilde{\Lambda}G_\sigma$ with the induced topology from 
$\Lambda G_\sigma$, where $\Lambda G_\sigma$ is considered as a loop group 
with $\lambda\in S^1$; and in a similar way, one defines four subgroups 
$\widetilde{\Lambda}^\pm G_\sigma$ and $\widetilde{\Lambda}^\pm_*G_\sigma$ of 
$\Lambda^\pm G_\sigma$ and $\Lambda^\pm_*G_\sigma$, respectively. 
   Then, the following two decomposition theorems hold (cf.\ \cite{Br}, 
\cite{Do-In-To}, \cite{To}):     
\begin{theorem}[Iwasawa decomposition of 
$\widetilde{\Lambda}G_\sigma\times\widetilde{\Lambda}G_\sigma$]\label{thm-3.1.5} 
   The multiplication maps 
\[
\begin{array}{l}
 \triangle(\widetilde{\Lambda}G_\sigma\times\widetilde{\Lambda}G_\sigma)
 \times(\widetilde{\Lambda}^-_* G_\sigma\times\widetilde{\Lambda}^+ G_\sigma)  
 \to 
 \widetilde{\Lambda}G_\sigma\times\widetilde{\Lambda}G_\sigma,\\
 \triangle(\widetilde{\Lambda}G_\sigma\times\widetilde{\Lambda}G_\sigma)
 \times(\widetilde{\Lambda}^+_* G_\sigma\times\widetilde{\Lambda}^- G_\sigma)  
 \to 
 \widetilde{\Lambda}G_\sigma\times\widetilde{\Lambda}G_\sigma
\end{array} 
\] 
are real analytic diffeomorphisms onto open subsets of 
$\widetilde{\Lambda}G_\sigma\times\widetilde{\Lambda}G_\sigma$, respectively.    
\end{theorem}

\begin{theorem}[Birkhoff decomposition of $\widetilde{\Lambda}G_\sigma$]
\label{thm-3.1.6}  
   The multiplication maps 
\[
\begin{array}{ll}
 \widetilde{\Lambda}^-_*G_\sigma\times\widetilde{\Lambda}^+G_\sigma 
 \to\widetilde{\Lambda}G_\sigma, 
& 
 \widetilde{\Lambda}^+_*G_\sigma\times\widetilde{\Lambda}^-G_\sigma 
 \to\widetilde{\Lambda}G_\sigma
\end{array}
\] 
are real analytic diffeomorphisms onto the open subsets 
$\widetilde{\mathcal{B}}_\mp:=
 \widetilde{\Lambda}^\mp_*G_\sigma\cdot\widetilde{\Lambda}^\pm G_\sigma$ 
of $\widetilde{\Lambda}G_\sigma$, respectively. 
   In particular, each element 
$A_\lambda\in\widetilde{\mathcal{B}}:=\widetilde{\mathcal{B}}_-
 \cap\widetilde{\mathcal{B}}_+$ 
can be uniquely factorized$:$  
\[
\begin{array}{lrr}
  A_\lambda=A^-_\lambda\cdot B^+_\lambda=A^+_\lambda\cdot B^-_\lambda, 
& A^\pm_\lambda\in\widetilde{\Lambda}^\pm_* G_\sigma, 
& B^\pm_\lambda\in\widetilde{\Lambda}^\pm G_\sigma.   
\end{array}
\]    
\end{theorem}

\begin{remark}\label{rem-3.1.7}
   Throughout this paper, we consider that for 
$A_\lambda\in\widetilde{\Lambda}G_\sigma$, its variable $\lambda$ varies  
not only in $S^1$ but also in $\mathbb{R}^+$ (or more generally in 
$\mathbb{C}^*$).     
\end{remark}

   We end this subsection with showing the following lemma:
\begin{lemma}\label{lem-3.1.8} 
   Each element $C_\lambda\in\widetilde{\Lambda}G_\sigma$ satisfies 
$C_\theta\in G:=\operatorname{Fix}(G^\mathbb{C},\nu)$ for all 
$\theta\in\mathbb{R}^+$.     
\end{lemma}
\begin{proof} 
   Since 
$C_\lambda\in\widetilde{\Lambda}G_\sigma\subset\Lambda G_\sigma
 =\operatorname{Fix}(\Lambda G^\mathbb{C}_\sigma,\nu_S)$, 
it satisfies $\nu(C_{\overline{\lambda}})=\nu_S(C_\lambda)=C_\lambda$ for all 
$\lambda\in S^1$. 
   Hence, one has $\nu(C_{\overline{\mu}})=C_\mu$ for all $\mu\in\mathbb{C}^*$; 
and therefore $\nu(C_\theta)=C_\theta$ for all $\theta\in\mathbb{R}^+$. 
\end{proof}

\subsection{Para-pluriharmonic maps and the loop group method}
\label{subsec-3.2}
   In this subsection, we will study the relation between para-pluriharmonic 
maps and the loop group method. 

\subsubsection{}\label{subsec-3.2.1} 
   Let $G^\mathbb{C}$ be a simply connected, simple, complex linear algebraic 
subgroup of $SL(m,\mathbb{C})$, let $\sigma$ be a holomorphic involution of 
$G^\mathbb{C}$, and let $\nu$ be an antiholomorphic involution of 
$G^\mathbb{C}$ such that $[\sigma,\nu]=0$.  
   Define subgroups $H^\mathbb{C}$, $G$ and $H$ by the same conditions as in 
Subsection \ref{subsec-2.2.3}, respectively---that is, 
\[
\begin{array}{lll} 
   H^\mathbb{C}:=\operatorname{Fix}(G^\mathbb{C},\sigma), 
&  G:=\operatorname{Fix}(G^\mathbb{C},\nu), 
&  H:=\operatorname{Fix}(G,\sigma)=\operatorname{Fix}(H^\mathbb{C},\nu). 
\end{array}  
\]
   We will conclude that the extended framing $F_\theta$ of a 
para-pluriharmonic map belongs to the loop group $\widetilde{\Lambda}G_\sigma$ 
(see \eqref{eq-3.1.3} for $\widetilde{\Lambda}G_\sigma$). 
   Let $(M,I)$ be a simply connected para-complex manifold, and let $F_\theta$ 
be the extended framing of a para-pluriharmonic map 
$f=\pi\circ F:(M,I)\to(G/H,\nabla^1)$ with $F(p_o)=\operatorname{id}$ and 
$[\alpha_\frak{m}^\pm\wedge\alpha_\frak{m}^\pm]=0$, where $p_o$ is a base point 
in $(M,I)$. 
   Then it follows from (2.2.15) and (2.2.16) that $F_\lambda$ belongs to 
$\Lambda G^\mathbb{C}_\sigma$. 
   Moreover, the variable $\lambda$ of $F_\lambda$ can vary in all of  
$\mathbb{C}^*$ (cf.\ Subsection \ref{subsec-2.2.3}).  
   Accordingly one can assert that the framing $F_\lambda$ belongs to 
$\widetilde{\Lambda}G_\sigma$, if it satisfies  
\begin{equation}\label{eq-3.2.1}
 \nu_S(F_\lambda)=F_\lambda
\end{equation}
(see \eqref{eq-3.1.2} for $\nu_S$). 
   Let us show \eqref{eq-3.2.1}. 
   From (2.2.17) we know $\nu(F_\theta)=F_\theta$ for any 
$\theta\in\mathbb{R}^+$. 
   This yields that $\nu_S(F_\lambda)=\nu(F_{\overline{\lambda}})=F_\lambda$ 
for any $\lambda\in S^1$ because $\nu(F_{\overline{\mu}})=F_\mu$ for any 
$\mu\in\mathbb{C}^*$ follows from $\nu(F_\theta)=F_\theta$ for any 
$\theta\in\mathbb{R}^+$. 
   Hence, we have shown \eqref{eq-3.2.1}. 
   Consequently the framing $F_\lambda$ belongs to 
$\widetilde{\Lambda}G_\sigma$.

\subsubsection{Para-pluriharmonic potentials}\label{subsec-3.2.2} 
   We have just shown that $F_\lambda$ belongs to $\widetilde{\Lambda}G_\sigma$, 
where $F_\lambda$ is the extended framing of a para-pluriharmonic map 
$f=\pi\circ F:(M,I)\to(G/H,\nabla^1)$ with $F(p_o)=\operatorname{id}$ and 
$[\alpha_\frak{m}^\pm\wedge\alpha_\frak{m}^\pm]=0$. 
   To $F_\lambda\in\widetilde{\Lambda}G_\sigma$, one can apply the Birkhoff 
decomposition theorem (cf.\ Theorem \ref{thm-3.1.6}). 
   We will obtain a pair of $\frak{m}$-valued $1$-forms $\eta_\theta$ and 
$\tau_\theta$ on $(M,I)$ parameterized $\theta\in\mathbb{R}^+$, from the 
framing $F_\theta$.\par  

   Since $F_\lambda(p_o)\equiv\operatorname{id}\in\widetilde{\mathcal{B}}$, 
one can perform a Birkhoff decomposition of the framing 
$F_\lambda\in\widetilde{\Lambda}G_\sigma$:  
\[
\begin{array}{lrr}
  F_\lambda=F^-_\lambda\cdot L^+_\lambda=F^+_\lambda\cdot L^-_\lambda, 
& F^\pm_\lambda\in\widetilde{\Lambda}^\pm_* G_\sigma, 
& L^\pm_\lambda\in\widetilde{\Lambda}^\pm G_\sigma,    
\end{array}
\]
on an open neighborhood $U$ of $M$ at $p_o$ (cf.\ Theorem \ref{thm-3.1.6}). 
   Define $\eta_\theta$ and $\tau_\theta$ by 
\[
\begin{array}{ll}
  \eta_\theta:=(F^-_\theta)^{-1}\cdot dF^-_\theta, 
& \tau_\theta:=(F^+_\theta)^{-1}\cdot dF^+_\theta,
\end{array}
\]
respectively. 
   Then for any $\theta\in\mathbb{R}^+$, both $\eta_\theta$ and $\tau_\theta$ 
become $\frak{m}$-valued $1$-forms on the para-complex manifold $(U,I)$; and 
furthermore, $\eta_\theta$ is para-holomorphic and $\tau_\theta$ is 
para-antiholomorphic. 
   Indeed, it is immediate from $F_\theta^{-1}\cdot dF_\theta=\alpha^\theta$ 
that 
\[
\begin{split}
 \alpha_\frak{h}+\theta^{-1}\cdot\alpha_\frak{m}^+
 +\theta\cdot\alpha_\frak{m}^- 
 =\alpha^\theta
&=(L^+_\theta)^{-1}\cdot((F^-_\theta)^{-1}\cdot dF^-_\theta)\cdot L^+_\theta  
 +(L^+_\theta)^{-1}\cdot dL^+_\theta\\
&=(L^-_\theta)^{-1}\cdot((F^+_\theta)^{-1}\cdot dF^+_\theta)\cdot L^-_\theta  
 +(L^-_\theta)^{-1}\cdot dL^-_\theta,
\end{split} 
\]
and that $\eta_\theta=\theta^{-1}\cdot\operatorname{Ad}(L^+_0)\alpha_\frak{m}^+$ 
and $\tau_\theta=\theta\cdot\operatorname{Ad}(L^-_0)\alpha_\frak{m}^-$, 
where $L^\pm_\lambda=\sum_{\pm k\geq 0}L_k^\pm\lambda^k$. 
   Here, we remark that $L^\pm_0\in H$ by Lemma \ref{lem-3.1.8}.\par   

   From the extended framing $F_\theta$, we have obtained the pair 
$(\eta_\theta,\tau_\theta)$ of an $\frak{m}$-valued para-holomorphic $1$-form 
and an $\frak{m}$-valued para-antiholomorphic $1$-form on $(U,I)$ parameterized 
by $\theta\in\mathbb{R}^+$. 
   In the next subsection, we will see that the pair 
$(\eta_\theta,\tau_\theta)$ is a {\it para-pluriharmonic potential} (cf.\ 
Definition \ref{def-3.2.1}).

\subsubsection{}\label{subsec-3.2.3}
   We are going to introduce the notion of a para-pluriharmonic potential. 
   Consider two linear subspaces 
$\widetilde{\Lambda}_{-1,\infty}\frak{g}_\sigma$ and 
$\widetilde{\Lambda}_{-\infty,1}\frak{g}_\sigma$ of 
$\widetilde{\Lambda}\frak{g}_\sigma$:   
\[
\begin{array}{l}
 \widetilde{\Lambda}_{-1,\infty}\frak{g}_\sigma
 :=\{ X_\lambda\in\widetilde{\Lambda}\frak{g}_\sigma
   \,|\, X_\lambda=\sum_{i=-1}^\infty X_i\lambda^i \},\\
 \widetilde{\Lambda}_{-\infty,1}\frak{g}_\sigma
 :=\{ Y_\lambda\in\widetilde{\Lambda}\frak{g}_\sigma
   \,|\, Y_\lambda=\sum_{j=-\infty}^1 Y_j\lambda^j \}, 
\end{array}
\]
where $\widetilde{\Lambda}\frak{g}_\sigma$ denotes the Lie algebra of 
$\widetilde{\Lambda}G_\sigma$ (see \eqref{eq-3.1.3} for 
$\widetilde{\Lambda}G_\sigma$). 
   Let $\widetilde{\mathcal{P}}_+=\widetilde{\mathcal{P}}_+(\frak{g})$ and 
$\widetilde{\mathcal{P}}_-=\widetilde{\mathcal{P}}_-(\frak{g})$ denote the 
sets of all $\widetilde{\Lambda}_{-1,\infty}\frak{g}_\sigma$-valued 
para-holomorphic and $\widetilde{\Lambda}_{-\infty,1}\frak{g}_\sigma$-valued 
para-antiholomorphic $1$-forms on a simply connected para-complex manifold 
$(M,I)$, respectively. 
\begin{definition}\label{def-3.2.1} 
   An element 
$(\eta_\lambda,\tau_\lambda)
 \in\widetilde{\mathcal{P}}_+\times\widetilde{\mathcal{P}}_-$ 
is called a {\it para-pluriharmonic potential} (or a {\it potential}, 
for short) on $(M,I)$.    
\end{definition}

\begin{remark}\label{rem-3.2.2} 
   (1) For each potential 
$(\eta_\lambda,\tau_\lambda)
 \in\widetilde{\mathcal{P}}_+\times\widetilde{\mathcal{P}}_-$,  
one may assume that the variable $\lambda$ of $\eta_\lambda$ (resp.\ 
$\tau_\lambda$) varies in $\mathbb{R}^+$ by virtue of 
$\eta_\lambda\in\widetilde{\Lambda}\frak{g}_\sigma$ (resp.\ 
$\tau_\lambda\in\widetilde{\Lambda}\frak{g}_\sigma$).\par

   (2) Note that we has just obtained a para-pluriharmonic potential 
$(\eta_\theta,\tau_\theta)$ from the extended framing $F_\theta$ of a 
para-pluriharmonic map $f:(M,I)\to(G/H,\nabla^1)$ with 
$F(p_o)=\operatorname{id}$ and 
$[\alpha_\frak{m}^\pm\wedge\alpha_\frak{m}^\pm]=0$.\par

   (3) It is unfortunate that the condition 
$[\alpha_\frak{m}^\pm\wedge\alpha_\frak{m}^\pm]=0$ for the applicability of 
the loop group method is necessary, but not always satisfied as shown by 
Krahe \cite{Kr}. 
   The condition is always true for surfaces and if the pseudo metric of the 
target space is positive definite. 
   Other natural conditions for the existence of 
$[\alpha_\frak{m}^\pm\wedge\alpha_\frak{m}^\pm]=0$ are not known.    
\end{remark}

   We has just obtained a para-pluriharmonic potential 
$(\eta_\theta,\tau_\theta)$ from the extended framing $F_\theta$ of a 
para-pluriharmonic map $f:(M,I)\to(G/H,\nabla^1)$ with 
$F(p_o)=\operatorname{id}$ and 
$[\alpha_\frak{m}^\pm\wedge\alpha_\frak{m}^\pm]=0$.  
   The converse statement is also true---that is, one can obtain a 
para-pluriharmonic map and its extended framing from any para-pluriharmonic 
potential and this framing satisfies 
$[\alpha_\frak{m}^\pm\wedge\alpha_\frak{m}^\pm]=0$: 
\begin{proposition}\label{prop-3.2.3} 
   Let 
$(\eta_\theta,\tau_\theta)
 \in\widetilde{\mathcal{P}}_+(\frak{g})
     \times\widetilde{\mathcal{P}}_-(\frak{g})$ 
be a para-pluriharmonic potential on the para-complex manifold $(M,I)$. 
   Then, the following steps provide an $\mathbb{R}^+$-family 
$\{f_\theta\}_{\theta\in\mathbb{R}^+}$ of para-pluriharmonic maps$:$ 
\begin{enumerate}
\item[{\rm (S1)}] 
   Solve the two initial value problems$:$ 
$(A^-_\theta)^{-1}\cdot dA^-_\theta=\eta_\theta$ and 
$(A^+_\theta)^{-1}\cdot dA^+_\theta=\tau_\theta$ with 
$A^\pm_\theta(p_o)\equiv\operatorname{id}$, where $p_o$ is a base point in $(M,I)$. 
\item[{\rm (S2)}]   
  Factorize 
$(A^-_\theta,A^+_\theta)
 \in\widetilde{\Lambda}G_\sigma\times\widetilde{\Lambda}G_\sigma$ 
in the Iwasawa decomposition $($cf.\ Theorem {\rm \ref{thm-3.1.5}}$):$   
$(A^-_\theta,A^+_\theta)=(C_\theta,C_\theta)\cdot(B^+_\theta,B^-_\theta)$, 
where $C_\theta\in\widetilde{\Lambda}G_\sigma$, 
$B^+_\theta\in\widetilde{\Lambda}^+_*G_\sigma$ and 
$B^-_\theta\in\widetilde{\Lambda}^- G_\sigma$. 
\item[{\rm (S3)}] 
   Then, $f_\theta:=\pi\circ C_\theta:(W,I)\to(G/H,\nabla^1)$ becomes a 
para-pluriharmonic map for every $\theta\in\mathbb{R}^+$. 
   Here, $W$ is any open neighborhood of $M$ at $p_o$ such that both {\rm (S1)} 
and {\rm (S2)} are solved on $W$.  
\end{enumerate}
   In particular, $C_\theta(p_o)\equiv\operatorname{id}$ and $C_\theta$ is 
the extended framing of the para-pluriharmonic map 
$f_1=\pi\circ C_1:(W,I)\to(G/H,\nabla^1)$. 
\end{proposition}
\begin{proof}
   (S1), (S2): The solution $(A_\theta^-,A_\theta^+)$ to (S1) satisfies 
$A_\theta^\mp\in\widetilde{\Lambda}G_\sigma$ and 
$A_\theta^\mp(p_o)\equiv\operatorname{id}$. 
   Therefore, it belongs to the open subset of 
$\widetilde{\Lambda}G_\sigma\times\widetilde{\Lambda}G_\sigma$ locally. 
   Hence, one can factorize $(A_\theta^-,A_\theta^+)$ by means of (S2).\par

   (S3): Let $W$ be any open neighborhood of $M$ at $p_o$ such that both (S1) 
and (S2) are solved on $W$. 
   First, let us show $C_\theta(p_o)\equiv\operatorname{id}$. 
   Since $A_\theta^\mp(p_o)\equiv\operatorname{id}$ we have 
$\widetilde{\Lambda}^+_* G_\sigma\ni B^+_\theta(p_o)=C_\theta(p_o)^{-1}
 =B^-_\theta(p_o)\in\widetilde{\Lambda}^- G_\sigma$. 
   Hence, 
$C_\theta(p_o)\in(\widetilde{\Lambda}^+_* G_\sigma
 \cap\widetilde{\Lambda}^- G_\sigma)=\{\operatorname{id}\}$. 
   Now, let $\beta^\mu:=C_\mu^{-1}\cdot dC_\mu$ for $\mu\in\mathbb{C}^*$. 
   Lemma \ref{lem-3.1.8} implies that $\beta^\theta$ is a $\frak{g}$-valued 
$1$-form on $W$ for any $\theta\in\mathbb{R}^+$.
   Therefore, one can express it as 
$\beta^\theta=(\beta^\theta)_\frak{h}+(\beta^\theta)_\frak{m}
=(\beta^\theta)_\frak{h}+(\beta^\theta)_\frak{m}^++(\beta^\theta)_\frak{m}^-$ 
by taking $\frak{g}=\frak{h}\oplus\frak{m}$ into consideration (see (2.2.5) 
and (2.2.6) for $(\beta^\theta)_\frak{h}$ and $(\beta^\theta)_\frak{m}^\pm$).  
   Then we obtain the conclusion, if one has 
\begin{equation}\label{eq-3.2.2}
  \beta^\theta=(\beta^1)_\frak{h}+\theta^{-1}\cdot(\beta^1)_\frak{m}^+
    +\theta\cdot(\beta^1)_\frak{m}^-
\end{equation}
because $\beta^\theta=C_\theta^{-1}\cdot dC_\theta$ satisfies 
$d\beta^\theta+(1/2)\cdot[\beta^\theta\wedge\beta^\theta]=0$ 
for any $\theta\in\mathbb{R}^+$, and thus the proof of Proposition 
\ref{prop-2.2.4} and \eqref{eq-3.2.2} allow us to conclude that 
$f_\theta=\pi\circ C_\theta:(W,I)\to(G/H,\nabla^1)$ is a para-pluriharmonic 
map for every $\theta\in\mathbb{R}^+$. 
   Hence, it suffices to prove \eqref{eq-3.2.2}.   
   Direct computation, together with 
$C_\theta=A_\theta^-\cdot(B^+_\theta)^{-1}=A_\theta^+\cdot(B^-_\theta)^{-1}$, 
gives us 
\[
\begin{split}
  (\beta^\theta,\beta^\theta)
&=\bigl(C_\theta^{-1}\cdot dC_\theta,\,\,
        C_\theta^{-1}\cdot dC_\theta\bigr)\\
&=\bigl(B^+_\theta\cdot\eta_\theta\cdot(B^+_\theta)^{-1}
    +B^+_\theta\cdot d(B^+_\theta)^{-1},\,\, 
   B^-_\theta\cdot\tau_\theta\cdot(B^-_\theta)^{-1}
    +B^-_\theta\cdot d(B^-_\theta)^{-1}\bigr). 
\end{split}
\]
   Therefore, the Fourier series 
$\beta^\lambda=\sum_{k\in\mathbb{Z}}\beta_k\lambda^k$ has actually the simple 
form:
\[\label{a}\tag{a}
   \beta^\lambda
=\lambda^{-1}\cdot\beta_{-1}+\beta_0+\lambda\cdot\beta_{+1}
\]   
because the $n$-th and $m$-th Fourier coefficients of 
$B^+_\lambda\cdot\eta_\lambda\cdot(B^+_\lambda)^{-1}
    +B^+_\lambda\cdot d(B^+_\lambda)^{-1}$ 
and 
$B^-_\lambda\cdot\tau_\lambda\cdot(B^-_\lambda)^{-1}
    +B^-_\lambda\cdot d(B^-_\lambda)^{-1}$ 
are zero for all $n\leq -2$ and $2\leq m$, respectively. 
   Let us denote by $(\beta_j)^+$ and $(\beta_j)^-$ the para-holomorphic 
component and the para-antiholomorphic component of $\beta_j$, respectively 
(i.e.,   $(\beta_j)^\pm:=(1/2)\cdot(\beta_j\pm{}^tI(\beta_j))$) for $j=\pm 1$, and 
rewrite the above \eqref{a} as 
\[\label{a'}\tag{a$'$}
   \beta^\lambda
=\lambda^{-1}\cdot((\beta_{-1})^++(\beta_{-1})^-)
  +\beta_0+\lambda\cdot((\beta_{+1})^++(\beta_{+1})^-).
\]    
   Then, \eqref{a'} simplifies to 
\[\label{a''}\tag{a$''$}
   \beta^\lambda
=\lambda^{-1}\cdot(\beta_{-1})^++\beta_0+\lambda\cdot(\beta_{+1})^-
\]
because the $-1$st and $+1$st Fourier coefficients of 
$B^+_\lambda\cdot\eta_\lambda\cdot(B^+_\lambda)^{-1}
    +B^+_\lambda\cdot d(B^+_\lambda)^{-1}$ 
and 
$B^-_\lambda\cdot\tau_\lambda\cdot(B^-_\lambda)^{-1}
    +B^-_\lambda\cdot d(B^-_\lambda)^{-1}$ 
are para-holomorphic and para-antiholomorphic, respectively. 
   From \eqref{a''} and $\beta^\lambda\in\widetilde{\Lambda}\frak{g}_\sigma$ 
we see that $(\beta^1)_\frak{h}=\beta_0$ and 
$(\beta^1)_\frak{m}=(\beta_{-1})^++(\beta_{+1})^-$. 
   This implies $(\beta^1)_\frak{m}^+=(\beta_{-1})^+$, 
$(\beta^1)_\frak{m}^-=(\beta_{+1})^-$ and 
$\beta^\lambda=\lambda^{-1}\cdot(\beta^1)_\frak{m}^++(\beta^1)_\frak{h}
 +\lambda\cdot(\beta^1)_\frak{m}^-$; 
and \eqref{eq-3.2.2} follows.   
\end{proof}

\subsection{Pluriharmonic maps and the loop group method}\label{subsec-3.3} 
   We have explained the relation between para-pluriharmonic maps and the loop 
group method in Subsection \ref{subsec-3.2}. 
   In this subsection, we will explain the relation between pluriharmonic maps 
and the loop group method. 
   The arguments below will be similar to those in Subsection \ref{subsec-3.2}.

\subsubsection{}\label{subsec-3.3.1} 
   In Subsection \ref{subsec-3.2.1} we have learned that the extended 
framing $F_\theta'$ of a para-pluriharmonic map belongs to the almost split 
real form $\Lambda G_\sigma$---that is, it satisfies 
$\nu_S(F_\lambda')=F_\lambda'$ for the involution $\nu_S$ of {\it the first 
kind} (cf.\ \eqref{eq-3.1.2} for $\nu_S$). 
   In this subsection, we will first confirm that the extended framing 
$F_\lambda$ of a pluriharmonic map satisfies $\nu_C(F_\lambda)=F_\lambda$ for 
the involution $\nu_C$ of {\it the second kind} defined below.\par

   Let $G^\mathbb{C}$ be a simply connected, simple, complex linear algebraic 
subgroup of $SL(m,\mathbb{C})$, let $\sigma$ be a holomorphic involution of 
$G^\mathbb{C}$, and let $\nu$ be an antiholomorphic involution of 
$G^\mathbb{C}$ such that $[\sigma,\nu]=0$.  
   Denote by $H^\mathbb{C}$, $G$ and $H$, the subgroups defined in Subsection 
\ref{subsec-2.2.3}, respectively (cf.\ \eqref{eq-2.2.14}). 
   Now, let us define an antiholomorphic involution $\nu_C$ of 
$\Lambda G^\mathbb{C}_\sigma$ by  
\begin{equation}\label{eq-3.3.1}
\begin{array}{ll}
  \nu_C(A_\lambda):=\nu(A_{1/\overline{\lambda}}) 
& \mbox{for $A_\lambda\in\Lambda G^\mathbb{C}_\sigma$}. 
\end{array}   
\end{equation} 
   This involution $\nu_C$ is said to be of {\it the second kind}, and 
satisfies the following: 
\begin{equation}\label{eq-3.3.2}
\begin{array}{ll}
  \nu_C(\Lambda^\pm G^\mathbb{C}_\sigma)=\Lambda^\mp G^\mathbb{C}_\sigma, 
& \nu_C(\Lambda^\pm_* G^\mathbb{C}_\sigma)=\Lambda^\mp_* G^\mathbb{C}_\sigma.  
\end{array}   
\end{equation} 
   Let $p_o$ be a base point in a simply connected complex manifold $(M,J)$, 
and let $F_\lambda$ be the extended framing of a pluriharmonic map 
$f=\pi\circ F:(M,J)\to(G/H,\nabla^1)$ with $F(p_o)=\operatorname{id}$ and 
$[\alpha_\frak{m}'\wedge\alpha_\frak{m}']=0$. 
   From (2.3.9) it follows that $F_\lambda\in\Lambda G^\mathbb{C}_\sigma$. 
   In particular, (2.3.10) implies that $F_\lambda$ satisfies 
\begin{equation}\label{eq-3.3.3}
  \nu_C(F_\lambda)=F_\lambda.     
\end{equation} 

\subsubsection{Pluriharmonic potentials}\label{subsec-3.3.2}
   Since $F_\lambda(p_o)\equiv\operatorname{id}$ we perform a Birkhoff 
decomposition of the framing $F_\lambda$. 
   Therefore we obtain a pair of $\frak{m}^\mathbb{C}$-valued $1$-forms 
$\eta_\lambda$ and $\tau_\lambda$ on $(M,J)$ parameterized by $\lambda\in S^1$. 
   Here 
$\frak{m}^\mathbb{C}:=\operatorname{Fix}(\frak{g}^\mathbb{C},-d\sigma)$.  
   We will see later that the pair $(\eta_\lambda,\tau_\lambda)$ is a 
pluriharmonic potential (cf.\ Definition \ref{def-3.3.1}).    

   Since $F_\lambda(p_o)\equiv\operatorname{id}\in\mathcal{B}^\mathbb{C}$, 
we factorize the framing $F_\lambda\in\Lambda G^\mathbb{C}_\sigma$ in the 
Birkhoff decomposition:  
\[
\begin{array}{lrr}
  F_\lambda=F^-_\lambda\cdot L^+_\lambda=F^+_\lambda\cdot L^-_\lambda, 
& F^\pm_\lambda\in\Lambda^\pm_* G^\mathbb{C}_\sigma, 
& L^\pm_\lambda\in\Lambda^\pm G^\mathbb{C}_\sigma,    
\end{array}
\]
on an open neighborhood $U$ of $M$ at $p_o$ (cf.\ Theorem \ref{thm-3.1.2}). 
   Define $\eta_\lambda$ and $\tau_\lambda$ by 
\[
\begin{array}{ll}
  \eta_\lambda:=(F^-_\lambda)^{-1}\cdot dF^-_\lambda, 
& \tau_\lambda:=(F^+_\lambda)^{-1}\cdot dF^+_\lambda,
\end{array}
\]
respectively. 
   Then for any $\lambda\in S^1$, both $\eta_\lambda$ and $\tau_\lambda$ 
become $\frak{m}^\mathbb{C}$-valued $1$-forms on the complex manifold $(U,J)$. 
   In addition, $\eta_\lambda$ is holomorphic and $\tau_\lambda$ is 
antiholomorphic. 
   Indeed, $F_\lambda^{-1}\cdot dF_\lambda=\alpha^\lambda$ yields 
\[
\begin{split}
 \alpha_\frak{h}+\lambda^{-1}\cdot\alpha_\frak{m}'
 +\lambda\cdot\alpha_\frak{m}'' 
 =\alpha^\lambda
&=(L^+_\lambda)^{-1}\cdot((F^-_\lambda)^{-1}\cdot dF^-_\lambda)\cdot L^+_\lambda  
 +(L^+_\theta)^{-1}\cdot dL^+_\theta\\
&=(L^-_\lambda)^{-1}\cdot((F^+_\lambda)^{-1}\cdot dF^+_\lambda)\cdot L^-_\lambda  
 +(L^-_\lambda)^{-1}\cdot dL^-_\lambda,
\end{split} 
\]
and  
$\eta_\lambda=\lambda^{-1}\cdot\operatorname{Ad}(L^+_0)\alpha_\frak{m}'$ 
and $\tau_\lambda=\lambda\cdot\operatorname{Ad}(L^-_0)\alpha_\frak{m}''$, 
where $L^\pm_\lambda=\sum_{\pm k\geq 0}L_k^\pm\lambda^k$. 
   Now, it follows from \eqref{eq-3.3.2} and \eqref{eq-3.3.3} that 
$\nu_C(F^-_\lambda)=F^+_\lambda$. 
   This implies that $\eta_\lambda$ is related with $\tau_\lambda$ by  
the formula $d\nu_C(\eta_\lambda)=\tau_\lambda$.  
   Consequently we obtain from the extended framing $F_\lambda$ of 
a pluriharmonic map the pair $(\eta_\lambda,\tau_\lambda)$ of 
an $\frak{m}^\mathbb{C}$-valued holomorphic $1$-form and an 
$\frak{m}^\mathbb{C}$-valued antiholomorphic $1$-form on $(U,J)$ satisfying 
$d\nu_C(\eta_\lambda)=\tau_\lambda$.

\subsubsection{}\label{subsec-3.3.3} 
   Let us introduce the following subspaces 
$\Lambda_{-1,\infty}\frak{g}^\mathbb{C}_\sigma$ and 
$\Lambda_{-\infty,1}\frak{g}^\mathbb{C}_\sigma$ of 
$\Lambda\frak{g}^\mathbb{C}_\sigma$, in order to recall the notion of a 
pluriharmonic potential:   
\[
\begin{array}{l}
 \Lambda_{-1,\infty}\frak{g}^\mathbb{C}_\sigma
 :=\{ X_\lambda\in\Lambda\frak{g}^\mathbb{C}_\sigma
   \,|\, X_\lambda=\sum_{i=-1}^\infty X_i\lambda^i \},\\
 \Lambda_{-\infty,1}\frak{g}^\mathbb{C}_\sigma
 :=\{ Y_\lambda\in\Lambda\frak{g}^\mathbb{C}_\sigma
   \,|\, Y_\lambda=\sum_{j=-\infty}^1 Y_j\lambda^j \} 
\end{array}
\]        
(cf.\ \eqref{eq-3.1.1} for $\Lambda\frak{g}^\mathbb{C}_\sigma$). 
   Let $\mathcal{P}'=\mathcal{P}'(\frak{g}^\mathbb{C})$ and 
$\mathcal{P}''=\mathcal{P}''(\frak{g}^\mathbb{C})$ denote the set of all 
$\Lambda_{-1,\infty}\frak{g}^\mathbb{C}_\sigma$-valued holomorphic 
and $\Lambda_{-\infty,1}\frak{g}^\mathbb{C}_\sigma$-valued antiholomorphic 
$1$-forms on a simply connected complex manifold $(M,J)$, respectively. 

\begin{definition}\label{def-3.3.1}
   An element $(\eta_\lambda,\tau_\lambda)\in\mathcal{P}'\times\mathcal{P}''$ 
is called a {\it pluriharmonic potential} (or a {\it potential}, for short) on 
$(M,J)$, if it satisfies $d\nu_C(\eta_\lambda)=\tau_\lambda$ (cf.\ 
\eqref{eq-3.3.1} for $\nu_C$).    
\end{definition}

   In Subsection \ref{subsec-3.3.2} one has obtained a pluriharmonic potential 
$(\eta_\lambda,\tau_\lambda)$ from the extended framing $F_\lambda$ of a 
pluriharmonic map $f=\pi\circ F:(M,J)\to(G/H,\nabla^1)$ with 
$F(p_o)=\operatorname{id}$ and $[\alpha_\frak{m}'\wedge\alpha_\frak{m}']=0$.    
   Next we recall from \cite{Do-Es} that one can obtain a pluriharmonic map 
and its extended framing from a pluriharmonic potential: 
\begin{proposition}\label{prop-3.3.2}
   Let 
$(\eta_\lambda,\tau_\lambda)=(\eta_\lambda,d\nu_C(\eta_\lambda))
 \in\mathcal{P}'(\frak{g}^\mathbb{C})\times\mathcal{P}''(\frak{g}^\mathbb{C})$  
be any pluriharmonic potential on the complex manifold $(M,J)$. 
   Then, the following steps provide an $S^1$-family 
$\{f_\lambda\}_{\lambda\in S^1}$ of pluriharmonic maps$:$  
\begin{enumerate}
\item[{\rm (S1)}] 
   Solve the two initial value problems$:$ 
$A_\lambda^{-1}\cdot dA_\lambda=\eta_\lambda$, 
$B_\lambda^{-1}\cdot dB_\lambda=\tau_\lambda$ with 
$A_\lambda(p_o)\equiv\operatorname{id}\equiv B_\lambda(p_o)$, where $p_o$ is 
a base point in $(M,J)$. 
\item[{\rm (S2)}] 
   Factorize 
$(A_\lambda,B_\lambda)\in\Lambda G^\mathbb{C}_\sigma\times
 \Lambda G^\mathbb{C}_\sigma$ 
in the Iwasawa decomposition $($cf.\ Theorem $\ref{thm-3.1.1}):$ 
$(A_\lambda,B_\lambda)
  =(C_\lambda,C_\lambda)\cdot(B^+_\lambda,B^-_\lambda)$, 
where $C_\lambda\in\Lambda G^\mathbb{C}_\sigma$, 
$B^+_\lambda\in\Lambda^+_*G^\mathbb{C}_\sigma$ and  
$B^-_\lambda\in\Lambda^-G^\mathbb{C}_\sigma$. 
\item[{\rm (S3)}] 
   Take an open neighborhood $V$ of $M$ at $p_o$ and a smooth map 
$h^\mathbb{C}=h^\mathbb{C}(p): V\to H^\mathbb{C}$ such that 
 \begin{enumerate}
  \item[{\rm (1)}]  
   $C_\lambda'(p)\in G$ for all $(p,\lambda)\in V\times S^1$,
  \item[{\rm (2)}]  
   $C_\lambda'(p_o)\equiv\operatorname{id}$, where 
   $C_\lambda':=C_\lambda\cdot h^\mathbb{C}$.
  \end{enumerate}  
\item[{\rm (S4)}] 
   Then, $f_\lambda:=\pi\circ C_\lambda':(V,J)\to(G/H,\nabla^1)$ 
becomes an $S^1$-family of pluriharmonic maps.       
\end{enumerate} 
\end{proposition}
\begin{proof}
   (S1), (S2): For the solutions $A_\lambda$ and $B_\lambda$ to (S1), we 
deduce that they satisfy 
\begin{equation}\label{eq-3.3.4}
 \nu_C(A_\lambda)=B_\lambda
\end{equation}
in terms of $d\nu_C(\eta_\lambda)=\tau_\lambda$. 
   Since $A_\lambda(p_o)\equiv\operatorname{id}\equiv B_\lambda(p_o)$ and 
$(A_\lambda(p_o),B_\lambda(p_o))$ belongs to a suitable open subset of 
$\Lambda G^\mathbb{C}_\sigma\times\Lambda G^\mathbb{C}_\sigma$, one can 
factorize $(A_\lambda,B_\lambda)$ by means of (S2).\par

   (S3): Let us assume that both (S1) and (S2) hold on an open neighborhood 
$W$ of $M$ at $p_o$. 
   We will confirm that there exist an open neighborhood $V$ ($\subset W$) of 
$M$ at $p_o$ and a smooth map 
$h^\mathbb{C}=h^\mathbb{C}(p):V\to H^\mathbb{C}$ such that 
\begin{enumerate}
\item 
 $C_\lambda(p)\cdot h^\mathbb{C}(p)\in G=\operatorname{Fix}(G^\mathbb{C},\nu)$ 
for all $(p,\lambda)\in V\times S^1$; 
\item 
 $C_\lambda(p_o)\cdot h^\mathbb{C}(p_o)\equiv\operatorname{id}$
\end{enumerate}
---that is, we want to assert that (S3) holds.  
   First, let us verify  
\[
  C_\lambda(p_o)\equiv\operatorname{id}.
\] 
   By $A_\lambda(p_o)\equiv\operatorname{id}\equiv B_\lambda(p_o)$ we conclude 
$\Lambda^+_* G^\mathbb{C}_\sigma\ni
   B^+_\lambda(p_o)=C_\lambda(p_o)^{-1}
    =B^-_\lambda(p_o)\in\Lambda^- G^\mathbb{C}_\sigma$;   
so that 
$C_\lambda(p_o)
 \in(\Lambda^+_*G^\mathbb{C}_\sigma\cap\Lambda^- G^\mathbb{C}_\sigma)
 =\{\operatorname{id}\}$, 
and $C_\lambda(p_o)\equiv\operatorname{id}$. 
   Next, we will deduce that 
\begin{equation}\label{eq-3.3.5}
\begin{array}{ll}
  (C_\lambda(q))^{-1}\cdot\nu(C_\lambda(q))\in H^\mathbb{C} 
& \mbox{for any point $(q,\lambda)\in W\times S^1$}.   
\end{array}
\end{equation}
   Since \eqref{eq-3.3.4}, \eqref{eq-3.3.2} and 
$C_\lambda=A_\lambda\cdot(B^+_\lambda)^{-1}=B_\lambda\cdot(B^-_\lambda)^{-1}$, 
we obtain     
\[
 (C_\lambda)^{-1}\cdot\nu_C(C_\lambda)
=\bigl(B_\lambda\cdot(B^-_\lambda)^{-1}\bigr)^{-1}\cdot 
  \nu_C\bigl(A_\lambda\cdot(B^+_\lambda)^{-1}\bigr)
=B^-_\lambda\cdot\nu_C((B^+_\lambda)^{-1})\in\Lambda^- G^\mathbb{C}_\sigma.  
\]  
   The above also leads to 
$(C_\lambda)^{-1}\cdot\nu_C(C_\lambda)
 =\nu_C\bigl(\{ (C_\lambda)^{-1}\cdot\nu_C(C_\lambda) \}^{-1} \bigr)
\in\nu_C(\Lambda^- G^\mathbb{C}_\sigma)=\Lambda^+ G^\mathbb{C}_\sigma$. 
   Therefore we have  
$(C_\lambda)^{-1}\cdot\nu_C(C_\lambda)
 \in(\Lambda^- G^\mathbb{C}_\sigma\cap\Lambda^+ G^\mathbb{C}_\sigma)=H^\mathbb{C}$, 
and so \eqref{eq-3.3.5} follows. 
   It remains to show that there exist an open neighborhood $V$ of $M$ at $p_o$ 
and a smooth map $h^\mathbb{C}=h^\mathbb{C}(p): V\to H^\mathbb{C}$ satisfying 
the equations (1) and (2) above. 
   Let $U_H$ and $O_\frak{h}$ denote open neighborhoods of $H^\mathbb{C}$ at 
$\operatorname{id}$ and of $\frak{h}^\mathbb{C}$ at $0$ such that 
$\exp:O_\frak{h}\to U_H$ is a diffeomorphism and $\nu(U_H)\subset U_H$. 
   Since \eqref{eq-3.3.5} and 
$(C_\lambda(p_o))^{-1}\cdot\nu(C_\lambda(p_o))=\operatorname{id}\in U_H$, there 
exists an open neighborhood $V$ ($\subset W$) of $p_o$ in $M$ such that  
$(C_\lambda(p))^{-1}\cdot\nu(C_\lambda(p))\in U_H$ for all $p\in V$. 
   Hence,   
\begin{equation}\label{eq-3.3.6}
\begin{array}{ll}
 (C_\lambda(p))^{-1}\cdot\nu(C_\lambda(p))=\exp X(p) 
& \mbox{on $V$}, 
\end{array}
\end{equation}
where $X=X(p):V\to O_\frak{h}$ is a smooth map with $X(p_o)=0$. 
   This yields  
\[
  \exp d\nu(X(p))
=\nu\bigl((C_\lambda(p))^{-1}\cdot\nu(C_\lambda(p))\bigr)\\
=\nu(C_\lambda(p))^{-1}\cdot(C_\lambda(p)) 
=\exp(-X(p)) 
\] 
and $d\nu(X(p))=-X(p)$. 
   Accordingly we conclude that (1)
$\nu(C_\lambda(p)\cdot h^\mathbb{C}(p))=C_\lambda(p)\cdot h^\mathbb{C}(p)$ 
for all $(p,\lambda)\in V\times S^1$ and (2) 
$C_\lambda(p_o)\cdot h^\mathbb{C}(p_o)\equiv\operatorname{id}$, by setting 
$h^\mathbb{C}(p):=\exp((1/2)\cdot X(p))$ (cf.\ \eqref{eq-3.3.6}).\par
   
   (S4): The arguments below will be similar to those of the proof of (S3) in 
Proposition \ref{prop-3.2.3}. 
   Define a $\frak{g}$-valued $1$-form $\beta^\lambda$ on $(V,J)$ by 
$\beta^\lambda:=(C_\lambda')^{-1}\cdot dC_\lambda'$, and express it as 
$\beta^\lambda=(\beta^\lambda)_\frak{h}+(\beta^\lambda)_\frak{m}
=(\beta^\lambda)_\frak{h}+(\beta^\lambda)_\frak{m}'+(\beta^\lambda)_\frak{m}''$,  
where $\frak{g}=\frak{h}\oplus\frak{m}$ (see (2.3.6) for 
$(\beta^\lambda)_\frak{m}'$ and $(\beta^\lambda)_\frak{m}''$). 
   Then, it suffices to verify \eqref{eq-3.3.7}: 
\begin{equation}\label{eq-3.3.7}
  \beta^\lambda=(\beta^1)_\frak{h}+\lambda^{-1}\cdot(\beta^1)_\frak{m}'
    +\lambda\cdot(\beta^1)_\frak{m}''. 
\end{equation}
   Indeed, $\beta^\lambda=(C_\lambda')^{-1}\cdot dC_\lambda'$ satisfies 
$d\beta^\lambda+(1/2)\cdot[\beta^\lambda\wedge\beta^\lambda]=0$ for any 
$\lambda\in S^1$, and so Proposition \ref{prop-2.3.3} and \eqref{eq-3.3.7} 
allow us to conclude that 
$f_\lambda=\pi\circ C_\lambda':(V,J)\to(G/H,\nabla^1)$ is a pluriharmonic map 
for every $\lambda\in S^1$.    
   Direct computation, together with 
$C_\lambda=A_\lambda\cdot(B^+_\lambda)^{-1}=B_\lambda\cdot(B^-_\lambda)^{-1}$ 
and $C_\lambda'=C_\lambda\cdot h^\mathbb{C}$, gives us 
\[
\begin{split}
  (\beta^\lambda,\beta^\lambda)
&=\bigl((C_\lambda')^{-1}\cdot dC_\lambda',\,\,
        (C_\lambda')^{-1}\cdot dC_\lambda'\bigr)\\
&=\bigl(D^+_\lambda\cdot\eta_\lambda\cdot(D^+_\lambda)^{-1}
    +D^+_\lambda\cdot d(D^+_\lambda)^{-1},\,\, 
   D^-_\lambda\cdot\tau_\lambda\cdot(D^-_\lambda)^{-1}
    +D^-_\lambda\cdot d(D^-_\lambda)^{-1}\bigr), 
\end{split}
\]
where $(D^\pm_\lambda)^{-1}:=(B^\pm_\lambda)^{-1}\cdot h^\mathbb{C}$. 
  It follows from $h^\mathbb{C}\in H^\mathbb{C}$ that 
$D^\pm_\lambda\in\Lambda^\pm G^\mathbb{C}_\sigma$. 
   Therefore, the Fourier series 
$\beta^\lambda=\sum_{k\in\mathbb{Z}}\beta_k\lambda^k$ is actually a Laurent 
polynomial of the form 
\[\label{a}\tag{a}
   \beta^\lambda
=\lambda^{-1}\cdot\beta_{-1}+\beta_0+\lambda\cdot\beta_{+1}
=\lambda^{-1}\cdot((\beta_{-1})'+(\beta_{-1})'')
  +\beta_0+\lambda\cdot((\beta_{+1})'+(\beta_{+1})'')
\]   
because the $n$-th and $m$-th Fourier coefficients of 
$D^+_\lambda\cdot\eta_\lambda\cdot(D^+_\lambda)^{-1}
    +D^+_\lambda\cdot d(D^+_\lambda)^{-1}$ 
and 
$D^-_\lambda\cdot\tau_\lambda\cdot(D^-_\lambda)^{-1}
    +D^-_\lambda\cdot d(D^-_\lambda)^{-1}$ 
are zero for all $n\leq -2$ and $2\leq m$, respectively. 
   Moreover, \eqref{a} simplifies to 
\[\label{a'}\tag{a$'$}
   \beta^\lambda
=\lambda^{-1}\cdot(\beta_{-1})'+\beta_0+\lambda\cdot(\beta_{+1})''
\]
because the $-1$st and $+1$st Fourier coefficients of 
$D^+_\lambda\cdot\eta_\lambda\cdot(D^+_\lambda)^{-1}
    +D^+_\lambda\cdot d(D^+_\lambda)^{-1}$ 
and 
$D^-_\lambda\cdot\tau_\lambda\cdot(D^-_\lambda)^{-1}
    +D^-_\lambda\cdot d(D^-_\lambda)^{-1}$ 
are holomorphic and antiholomorphic, respectively. 
   In view of \eqref{a'} and 
$\beta^\lambda\in\Lambda\frak{g}^\mathbb{C}_\sigma$ it turns out that 
$(\beta^1)_\frak{h}=\beta_0$ and 
$(\beta^1)_\frak{m}=(\beta_{-1})'+(\beta_{+1})''$. 
    Therefore \eqref{eq-3.3.7} follows from $(\beta^1)_\frak{m}'=(\beta_{-1})'$ 
and $(\beta^1)_\frak{m}''=(\beta_{+1})''$.   
\end{proof}

\section{{\bf Relation between pluriharmonic maps and para-pluriharmonic maps}}
\label{sec-4} 
   In this section, by utilizing the loop group method, we interrelate 
pluriharmonic maps with para-pluriharmonic maps. 
   We consider two real subspaces $\mathbb{A}^{2n}$ and 
$\mathbb{B}^{2n}$ of $\mathbb{C}^{2n}$ (cf.\ Subsection \ref{subsec-4.1}), 
and two symmetric closed subspaces $G_1/H_1$ and $G_2/H_2$ of 
$G^\mathbb{C}/H^\mathbb{C}$ (cf.\ Subsection \ref{subsec-4.2}), and we  
investigate the relation between certain pluriharmonic maps 
$f_1:\mathbb{A}^{2n}\to G_1/H_1$ and certain para-pluriharmonic maps 
$f_2:\mathbb{B}^{2n}\to G_2/H_2$  (cf.\ Subsection \ref{subsec-4.3}).   

\begin{center}
\unitlength=1mm
\begin{picture}(75,21)
\put(1,17){$f_1:\mathbb{A}^{2n}\longrightarrow G_1/H_1$, pluriharmonic}
\put(7,13){$\cap$}
\put(25,13){$\cap$}
\put(7,9){$\mathbb{C}^{2n}$}
\put(21,9){$G^\mathbb{C}/H^\mathbb{C}$}
\put(7,5){$\cup$}
\put(25,5){$\cup$}
\put(1,1){$f_2:\mathbb{B}^{2n}\longrightarrow G_2/H_2$, para-pluriharmonic}
\end{picture}   
\end{center}

\subsection{The real subspaces $\mathbb{A}^{2n}$ and $\mathbb{B}^{2n}$ of 
$\mathbb{C}^{2n}$}\label{subsec-4.1}  
   Let $\mathbb{A}^{2n}$ and $\mathbb{B}^{2n}$ be the real subspaces of 
$\mathbb{C}^{2n}$ given by 
\[
\begin{split}
  \mathbb{A}^{2n}:
&=\{(z^1,\cdots,z^n,w^1,\cdots,w^n)\in\mathbb{C}^n\times\mathbb{C}^n 
  \,|\, \mbox{$\bar{z}^a=w^a$ for all $1\leq a\leq n$} \}\\
&=\{(z,w)\in\mathbb{C}^n\times\mathbb{C}^n \,|\, w=\bar{z}\},\\  
  \mathbb{B}^{2n}:
&=\{(z^1,\cdots,z^n,w^1,\cdots,w^n)\in\mathbb{C}^n\times\mathbb{C}^n 
  \,|\, \mbox{$z^a=\bar{z}^a$ and $w^a=\bar{w}^a$ for all $1\leq a\leq n$}\}\\
&=\mathbb{R}^n\times\mathbb{R}^n.  
\end{split} 
\] 
   Let $(x^1,\cdots,x^n,y^1,\cdots,y^n)$ denote the global coordinate system 
on $\mathbb{B}^{2n}$ defined by $x^a:={\rm Re}(z^a)$ and 
$y^a:={\rm Re}(w^a)$ for $1\leq a\leq n$.    
   Define smooth $(1,1)$-tensor fields $J$ on $\mathbb{A}^{2n}$ and $I$ on 
$\mathbb{B}^{2n}$ by  
\[
 J\Big(\frac{\partial}{\partial z^a}\Big)
  :=i\frac{\partial}{\partial z^a},\,
  J\Big(\frac{\partial}{\partial\bar{z}^a}\Big)
  :=-i\frac{\partial}{\partial\bar{z}^a}
\mbox{ and } 
 I\Big(\frac{\partial}{\partial x^a}\Big)
   :=\frac{\partial}{\partial x^a},\, 
  I\Big(\frac{\partial}{\partial y^a}\Big)
   :=-\frac{\partial}{\partial y^a}. 
\] 
   Then $(\mathbb{A}^{2n},J)$ and $(\mathbb{B}^{2n},I)$ are simply connected 
complex and para-complex manifolds, respectively. 
   Henceforth, for the natural coordinate systems 
$(z^1,\cdots,z^n,\bar{z}^1,\cdots,\bar{z}^n)$ on $\mathbb{A}^{2n}$, 
$(x^1,\cdots,x^n,y^1,\cdots,y^n)$ on $\mathbb{B}^{2n}$ and 
$(z^1,\cdots,z^n,w^1,\cdots,w^n)$ on $\mathbb{C}^{2n}$, we will use the 
notation $({\bf z},{\bf \bar{z}})$, $({\bf x},{\bf y})$ and $({\bf z},{\bf w})$, 
respectively.

\subsection{The symmetric subspaces $G_1/H_1$ and $G_2/H_2$ of 
$G^\mathbb{C}/H^\mathbb{C}$}\label{subsec-4.2} 
   In this subsection, we introduce two symmetric subspaces $G_1/H_1$ and 
$G_2/H_2$ of $G^\mathbb{C}/H^\mathbb{C}$. 
   Let $G^\mathbb{C}$ be a simply connected, simple, complex linear algebraic 
subgroup of $SL(m,\mathbb{C})$, let $\sigma$ be a holomorphic involution of 
$G^\mathbb{C}$, and let $\nu_1$ and $\nu_2$ be antiholomorphic involutions of 
$G^\mathbb{C}$ satisfying $[\sigma,\nu_1]=[\sigma,\nu_2]=[\nu_1,\nu_2]=0$. 
   Then we define $H^\mathbb{C}$, $G_i$, $H_i$, $\pi_i$ ($i=1,2$) and 
$\frak{g}_2$ as follows: 
\begin{enumerate}[(4.2.1)] 
\item 
  $H^\mathbb{C}:=\operatorname{Fix}(G^\mathbb{C},\sigma)$, 
\item 
  $G_i:=\operatorname{Fix}(G^\mathbb{C},\nu_i)$, 
\item 
  $H_i:=\operatorname{Fix}(G_i,\sigma)=\operatorname{Fix}(H^\mathbb{C},\nu_i)$, 
\item 
  $\pi_i$: the projection from $G_i$ onto $G_i/H_i$, 
\item 
  $\frak{g}_2:=\operatorname{Lie}G_2$. 
\end{enumerate} 
   Clearly, $(G^\mathbb{C}/H^\mathbb{C},\sigma)$ is an affine symmetric space, 
and both $G_1/H_1$ and $G_2/H_2$ are symmetric closed subspaces of 
$(G^\mathbb{C}/H^\mathbb{C},\sigma)$ (ref.\ \cite[p.\ 227]{Ko-No} for the 
definition of symmetric closed subspace). 
  In particular, $(G_i/H_i,\sigma|_{G_i})$, $i=1,2$, are affine symmetric 
spaces.

\subsection{The main result}\label{subsec-4.3} 
   With the notation in Subsections \ref{subsec-4.1} and \ref{subsec-4.2} we 
assert the following (see \eqref{eq-3.3.1} for $(\nu_1)_C$):  
\begin{theorem}\label{thm-4.3.1} 
   Let 
$(\eta_\theta,\tau_\theta)=(\eta_\theta({\bf x}),\tau_\theta({\bf y}))
 \in\widetilde{\mathcal{P}}_+(\frak{g}_2)\times\widetilde{\mathcal{P}}_-(\frak{g}_2)$ 
be a real analytic, para-pluriharmonic potential on $(\mathbb{B}^{2n},I)$, 
and let 
$(f_2)_\theta=\pi_2\circ C_\theta({\bf x},{\bf y}):(W,I)\to(G_2/H_2,\nabla^1)$ 
denote the $\mathbb{R}^+$-family of para-pluriharmonic maps constructed 
from $(\eta_\theta,\tau_\theta)$ in the neighborhood $W$ of 
$\mathbb{B}^{2n}$ at $({\bf 0},{\bf 0})$ in Proposition $\ref{prop-3.2.3}$.  
   Suppose that $(\eta_\theta,\tau_\theta)$ satisfies the morphing condition 
\[\label{M}\tag{M}
  d(\nu_1)_C(\eta_\lambda({\bf z}))=\tau_\lambda({\bf \bar{z}}).  
\]
   Then, there exist an open neighborhood $V$ of $\mathbb{A}^{2n}$ at 
$({\bf 0},{\bf 0})$ and a smooth map 
$h^\mathbb{C}({\bf z},{\bf \bar{z}}):V\to H^\mathbb{C}$ such that 
\begin{enumerate} 
\item[{\rm (1)}] 
  $C_\lambda'({\bf z},{\bf \bar{z}})\in G_1$ for all 
$({\bf z},{\bf \bar{z}};\lambda)\in V\times S^1;$ 
\item[{\rm (2)}] 
 $(f_1)_\lambda:=\pi_1\circ C_\lambda'({\bf z},{\bf \bar{z}}):
(V,J)\to(G_1/H_1,\nabla^1)$    
is an $S^1$-family of pluriharmonic maps with 
$C_\lambda'({\bf 0},{\bf 0})\equiv\operatorname{id}$, where 
$C_\lambda'({\bf z},{\bf \bar{z}})
 :=C_\lambda({\bf z},{\bf \bar{z}})\cdot h^\mathbb{C}({\bf z},{\bf \bar{z}})$. 
\end{enumerate}  
\end{theorem}

\begin{remark}\label{rem-4.3.2} 
   (i) Since both $\eta_\theta({\bf x})$ and $\tau_\theta({\bf y})$ are analytic on 
$\mathbb{B}^{2n}$ and $\mathbb{B}^{2n}$ is a totally real submanifold of 
$\mathbb{C}^{2n}$, one can uniquely extend them as holomorphic $1$-forms 
$\eta_\theta({\bf z})$ and $\tau_\theta({\bf w})$ to an open subset 
$\widetilde{W}$ of $\mathbb{C}^{2n}$ such that 
$\mathbb{B}^{2n}\subset\widetilde{W}$. 
    For this reason, the notation $\eta_\lambda({\bf z})$ and 
$\tau_\lambda({\bf \bar{z}})$ in Theorem \ref{thm-4.3.1} makes sense.\par
 
   (ii) Similarly, one can verify that the notation 
$C_\lambda({\bf z},{\bf \bar{z}})$, used in Theorem \ref{thm-4.3.1}, makes 
sense. 
\end{remark}

\begin{proof}[Proof of Theorem {\rm \ref{thm-4.3.1}}] 
   Let 
$(A_\lambda({\bf x}),B_\lambda({\bf y}))
 =(C_\lambda({\bf x},{\bf y}),C_\lambda({\bf x},{\bf y}))
  \cdot(B^+_\lambda({\bf x},{\bf y}),B^-_\lambda({\bf x},{\bf y}))$ 
denote the Iwasawa decomposition in (S2) of Proposition \ref{prop-3.2.3}. 
   Note that $A_\lambda({\bf x})$ and $B_\lambda({\bf y})$ satisfy 
\[ 
\begin{array}{lll}
  (A_\lambda^{-1}\cdot dA_\lambda)({\bf x})=\eta_\lambda({\bf x}), 
& (B_\lambda^{-1}\cdot dB_\lambda)({\bf y})=\tau_\lambda({\bf y}), 
& A_\lambda({\bf 0})\equiv\operatorname{id}\equiv B_\lambda({\bf 0}).
\end{array}
\] 
   Since $(\eta_\lambda({\bf x}),\tau_\lambda({\bf y}))$ is analytic, we deduce 
that $A_\lambda({\bf x})$, $B_\lambda({\bf y})$, $C_\lambda({\bf x},{\bf y})$ 
and $B^\pm_\lambda({\bf x},{\bf y})$ are analytic with respect to the variables 
${\bf x}$ and ${\bf y}$. 
   Therefore these matrices have unique analytic extensions 
$A_\lambda({\bf z})$, $B_\lambda({\bf w})$, $C_\lambda({\bf z},{\bf w})$ and 
$B^\pm_\lambda({\bf z},{\bf w})$ to an open neighborhood $\widetilde{W}$ of 
$\mathbb{C}^{2n}$ at $({\bf 0},{\bf 0})$, respectively, because 
$\mathbb{B}^{2n}$ is a totally real submanifold of $\mathbb{C}^{2n}$. 
   Then on the neighborhood $\widetilde{W}\cap\mathbb{A}^{2n}$ of 
$\mathbb{A}^{2n}$ at $({\bf 0},{\bf 0})$, we confirm that 
$A_\lambda({\bf z})$ and $B_\lambda({\bf \bar{z}})$ satisfy 
$(A_\lambda^{-1}\cdot dA_\lambda)({\bf z})=\eta_\lambda({\bf z})$, 
$(B_\lambda^{-1}\cdot dB_\lambda)({\bf \bar{z}})=\tau_\lambda({\bf \bar{z}})$ 
and 
$A_\lambda({\bf 0})\equiv\operatorname{id}\equiv B_\lambda({\bf 0})$; 
and furthermore, 
$(A_\lambda({\bf z}),B_\lambda({\bf \bar{z}}))
 =(C_\lambda({\bf z},{\bf \bar{z}}),C_\lambda({\bf z},{\bf \bar{z}}))
  \cdot(B^+_\lambda({\bf z},{\bf \bar{z}}),
      B^-_\lambda({\bf z},{\bf \bar{z}}))$ 
becomes the Iwasawa decomposition in (S2) of Proposition \ref{prop-3.3.2}, 
where we remark that $(\eta_\lambda({\bf z}),\tau_\lambda({\bf \bar{z}}))$ 
satisfy 
$(\eta_\lambda({\bf z}),\tau_\lambda({\bf \bar{z}}))
 \in\mathcal{P}'(\frak{g}^\mathbb{C})\times\mathcal{P}''(\frak{g}^\mathbb{C})$ 
and $d(\nu_1)_C(\eta_\lambda({\bf z}))=\tau_\lambda({\bf \bar{z}})$.   
   Consequently, the proof of Proposition \ref{prop-3.3.2} assures that there 
exist an open neighborhood $V$ $\subset\widetilde{W}\cap\mathbb{A}^{2n}$ of 
$\mathbb{A}^{2n}$ at $({\bf 0},{\bf 0})$ and a smooth map 
$h^\mathbb{C}({\bf z},{\bf \bar{z}}):V\to H^\mathbb{C}$ satisfying the 
conditions {\rm (1)} and {\rm (2)}.   
\end{proof}

\section{Appendix}\label{sec-5}
   We will interrelate concretely some pluriharmonic maps with para-pluriharmonic 
maps by means of Theorem \ref{thm-4.3.1}. 
   In Subsection \ref{subsec-5.2} we will focus on harmonic maps and Lorentz 
harmonic maps. 
   This will yield a relation between CMC-surfaces in $\mathbb{R}^3$ and 
CMC-surface in $\mathbb{R}^3_1$. 

\subsection{A relation between certain pluriharmonic maps and certain 
para-pluriharmonic maps}\label{subsec-5.1} 

\subsubsection{$f_1:\mathbb{A}^4\to Gr_{2,4}(\mathbb{C})$ $\Longleftrightarrow$ 
$f_2:\mathbb{B}^4\to Gr_{2,4}(\mathbb{C}')$}\label{subsec-5.1.1}    
   Following the main result of this paper, we construct in this subsection a 
pluriharmonic map 
$f_1(z^1,z^2,\bar{z}^1,\bar{z}^2):\mathbb{A}^4\to Gr_{2,4}(\mathbb{C})$ and 
a para-pluriharmonic map 
$f_2(x^1,x^2,y^1,y^2):\mathbb{B}^4\to Gr_{2,4}(\mathbb{C}')$ from one potential 
\eqref{eq-5.1.1} below, where $Gr_{2,4}(\mathbb{C})$ (resp.\ 
$Gr_{2,4}(\mathbb{C}')$) denotes a complex (resp.\ para-complex) Grassmann 
manifold.    
   In this subsection, we will use the following notation: 
\begin{enumerate}
\item[(5.1.1)]  
   $G^\mathbb{C}=SL(4,\mathbb{C})$, 
\item[(5.1.2)]  
   $\sigma(A):=I_{2,2}\cdot A\cdot I_{2,2}$ for $A\in G^\mathbb{C}$,  
where $I_{2,2}:=\operatorname{diag}(-1,-1,1,1)$, 
\item[(5.1.3)]  
   $\nu_1(A):={}^t(\overline{A})^{-1}$ for $A\in G^\mathbb{C}$, 
\item[(5.1.4)]  
   $\nu_2(A):=\overline{A}$ for $A\in G^\mathbb{C}$,  
\item[(5.1.5)]  
  $G^\mathbb{C}/H^\mathbb{C}
   =SL(4,\mathbb{C})/S(GL(2,\mathbb{C})\times GL(2,\mathbb{C}))$, 
\item[(5.1.6)]  
  $G_1/H_1=SU(4)/S(U(2)\times U(2))\simeq Gr_{2,4}(\mathbb{C})$, 
\item[(5.1.7)]  
  $G_2/H_2=SL(4,\mathbb{R})/S(GL(2,\mathbb{R})\times GL(2,\mathbb{R}))
  \simeq Gr_{2,4}(\mathbb{C}')$,    
\item[(5.1.8)]  
  $\pi_i$: the projection from $G_i$ onto $G_i/H_i$ ($i=1,2$), 
\item[(5.1.9)]  
  $\frak{g}_2:=\operatorname{Lie}G_2=\frak{sl}(4,\mathbb{R})$.     
\end{enumerate}\par 
\setcounter{equation}{9} 

   First, we define a 
$\widetilde{\Lambda}_{-1,\infty}(\frak{g}_2)_\sigma$-valued, real analytic 
para-holomorphic $1$-form $\eta_\theta(x^1,x^2)$ on $(\mathbb{B}^4,I)$ by 
\begin{equation}\label{eq-5.1.1}
\eta_\theta(x^1,x^2):=
 \theta^{-1}\begin{pmatrix}
             0 & 0 & 0 & 1 \\
             0 & 0 & 0 & 0 \\
             0 & 0 & 0 & 0 \\
            -1 & 0 & 0 & 0 
            \end{pmatrix}dx^1
+\theta^{-1}\begin{pmatrix}
             0 & 0 & 0 & 0 \\
             0 & 0 & 1 & 0 \\
             0 & 1 & 0 & 0 \\
             0 & 0 & 0 & 0 
            \end{pmatrix}dx^2.   
\end{equation}
   Taking the morphing condition \eqref{M} of Theorem \ref{thm-4.3.1} into 
consideration, we define a 
$\widetilde{\Lambda}_{-\infty,1}(\frak{g}_2)_\sigma$-valued, real analytic 
para-antiholomorphic $1$-form $\tau_\theta(y^1,y^2)$ on $(\mathbb{B}^4,I)$ by 
setting   
\[
\tau_\theta(y^1,y^2):=
 \theta\begin{pmatrix}
             0 & 0 & 0 & 1 \\
             0 & 0 & 0 & 0 \\
             0 & 0 & 0 & 0 \\
            -1 & 0 & 0 & 0 
        \end{pmatrix}dy^1
+\theta\begin{pmatrix}
             0 &  0 &  0 & 0 \\
             0 &  0 & -1 & 0 \\
             0 & -1 &  0 & 0 \\
             0 &  0 &  0 & 0 
       \end{pmatrix}dy^2. 
\]
   Hence we obtain the real analytic, para-pluriharmonic potential 
$(\eta_\theta(x^1,x^2),\tau_\theta(y^1,y^2))$.\par

   From Proposition \ref{prop-3.2.3} we obtain a para-pluriharmonic map 
$f_2:(\mathbb{B}^4,I)\to G_2/H_2\simeq Gr_{2,4}(\mathbb{C}')$.\par
 
   (S1): Solve the two initial value problems: 
\[
\begin{array}{lll}
  (A^-_\theta)^{-1}\cdot dA^-_\theta=\eta_\theta, 
& (A^+_\theta)^{-1}\cdot dA^+_\theta=\tau_\theta, 
& A^\pm_\theta(0,0)\equiv\operatorname{id}.
\end{array}
\]  
   The solutions are  
\[
\begin{array}{l}
  A^-_\theta(x^1,x^2)
=\begin{pmatrix}
  \cos(x^1/\theta) & 0 & 0 & \sin(x^1/\theta) \\
  0 & \cosh(x^2/\theta) & \sinh(x^2/\theta) & 0 \\
  0 & \sinh(x^2/\theta) & \cosh(x^2/\theta) & 0 \\
  -\sin(x^1/\theta) & 0 & 0 & \cos(x^1/\theta) \\  
 \end{pmatrix},\\
  A^+_\theta(y^1,y^2)
=\begin{pmatrix}
  \cos(\theta y^1) & 0 & 0 & \sin(\theta y^1) \\
  0 & \cosh(-\theta y^2) & \sinh(-\theta y^2) & 0 \\
  0 & \sinh(-\theta y^2) & \cosh(-\theta y^2) & 0 \\
  -\sin(\theta y^1) & 0 & 0 & \cos(\theta y^1) \\  
 \end{pmatrix}.
\end{array} 
\]
   (S2): Factorize 
$(A^-_\theta,A^+_\theta)
 \in\widetilde{\Lambda}^-_*(G_2)_\sigma
     \times\widetilde{\Lambda}^+_*(G_2)_\sigma$    
in the Iwasawa decomposition Theorem \ref{thm-3.1.5}:   
\[
 (A^-_\theta,A^+_\theta)=(C_\theta,C_\theta)\cdot(B^+_\theta,B^-_\theta),
\] 
where $C_\theta\in\widetilde{\Lambda}(G_2)_\sigma$, 
$B^+_\theta\in\widetilde{\Lambda}^+_*(G_2)_\sigma$ and 
$B^-_\theta\in\widetilde{\Lambda}^-(G_2)_\sigma$.  
   Here, $B^\pm_\theta$ and $C_\theta$ are given by 
$B^\pm_\theta=(A^\pm_\theta)^{-1}$ and 
\begin{multline*}
 C_\theta(x^1,x^2,y^1,y^2)\\
=\begin{pmatrix}
  \cos(x^1/\theta+\theta y^1) & 0 & 0 & \sin(x^1/\theta+\theta y^1) \\
  0 & \cosh(x^2/\theta-\theta y^2) & \sinh(x^2/\theta-\theta y^2) & 0 \\
  0 & \sinh(x^2/\theta-\theta y^2) & \cosh(x^2/\theta-\theta y^2) & 0 \\
  -\sin(x^1/\theta+\theta y^1) & 0 & 0 & \cos(x^1/\theta+\theta y^1) \\  
  \end{pmatrix}.
\end{multline*} 
   (S3): The last step of Proposition \ref{prop-3.2.3} assures  
\[
\mbox{
 $(f_2)_\theta:=\pi_2\circ C_\theta(x^1,x^2,y^1,y^2): 
  (\mathbb{B}^4,I)\to Gr_{2,4}(\mathbb{C}')$ 
  is para-pluriharmonic}
\]
for every $\theta\in\mathbb{R}^+$.\par

   We will construct a pluriharmonic map 
$f_1:(\mathbb{A}^4,J)\to G_1/H_1\simeq Gr_{2,4}(\mathbb{C})$ from 
$C_\theta(x^1,x^2,y^1,y^2)$ given above.  
   Substituting $\lambda$, $z^i$ and $\bar{z}^i$ for $\theta$, $x^i$ and $y^i$, 
respectively ($i=1,2$) we obtain 
\begin{multline*}
 C_\lambda(z^1,z^2,\bar{z}^1,\bar{z}^2)\\
=\begin{pmatrix}
  \cos(z^1/\lambda+\lambda\bar{z}^1) & 0 & 0 & \sin(z^1/\lambda+\lambda\bar{z}^1) \\
  0 & \cosh(z^2/\lambda-\lambda\bar{z}^2) & \sinh(z^2/\lambda-\lambda\bar{z}^2) & 0 \\
  0 & \sinh(z^2/\lambda-\lambda\bar{z}^2) & \cosh(z^2/\lambda-\lambda\bar{z}^2) & 0 \\
  -\sin(z^1/\lambda+\lambda\bar{z}^1) & 0 & 0 & \cos(z^1/\lambda+\lambda\bar{z}^1) \\  
  \end{pmatrix}
\end{multline*}
for $C_\theta(x^1,x^2,y^1,y^2)$. 
   Then $C_\lambda(z^1,z^2,\bar{z}^1,\bar{z}^2)\in G_1=SU(4)$ for all 
$(z^1,z^2,\bar{z}^1,\bar{z}^2;\lambda)\in\mathbb{A}^4\times S^1$ because 
$z^1/\lambda+\lambda\bar{z}^1$ is a real number and 
$z^2/\lambda-\lambda\bar{z}^2$ is a purely imaginary number. 
   Hence, we conclude that 
\[
\mbox{
 $(f_1)_\lambda:=\pi_1\circ C_\lambda(z^1,z^2,\bar{z}^1,\bar{z}^2): 
  (\mathbb{A}^4,J)\to Gr_{2,4}(\mathbb{C})$ 
  is a pluriharmonic map}
\]
for every $\lambda\in S^1$. 
   Consequently, we have constructed a pluriharmonic map 
$f_1:\mathbb{A}^4\to Gr_{2,4}(\mathbb{C})$ and a para-pluriharmonic map 
$f_2:\mathbb{B}^4\to Gr_{2,4}(\mathbb{C}')$ from the potential \eqref{eq-5.1.1}.

\begin{center}
\fbox{
\begin{minipage}{130mm}
$(f_1)_\lambda=\pi_1\circ C_\lambda(z^1,z^2,\bar{z}^1,\bar{z}^2): 
  (\mathbb{A}^4,J)\to Gr_{2,4}(\mathbb{C})$ 
is pluriharmonic\par
\centerline{$\Updownarrow$} 
$(f_2)_\theta=\pi_2\circ C_\theta(x^1,x^2,y^1,y^2): 
  (\mathbb{B}^4,I)\to Gr_{2,4}(\mathbb{C}')$ 
  is para-pluriharmonic 
\end{minipage}
}
\end{center}

\subsubsection{$f_1:\mathbb{A}^4\to S^4$ $\Longleftrightarrow$ 
$f_2:\mathbb{B}^4\to H^4$}\label{subsec-5.1.2} 
   In this subsection, we will construct a pluriharmonic map 
$f_1(z^1,z^2,\bar{z}^1,\bar{z}^2):\mathbb{A}^4\to S^4$ and a para-pluriharmonic 
map $f_2(x^1,x^2,y^1,y^2):\mathbb{B}^4\to H^4$ by arguments similar to those in 
Subsection \ref{subsec-5.1.1}. 
   Here $S^4$ and $H^4$ denote a sphere and a upper half space of dimension 
$4$, respectively.    
   Henceforth, we will use the following notation: 
\begin{enumerate} 
\item[(5.1.11)] 
   $G^\mathbb{C}=Sp(2,\mathbb{C})$ (see \cite[p.\ 445]{He} for 
$Sp(2,\mathbb{C})$), 
\item[(5.1.12)] 
   $\sigma(A):=K_{1,1}\cdot A\cdot K_{1,1}$ for $A\in G^\mathbb{C}$,  
where $K_{1,1}:=\operatorname{diag}(-1,1,-1,1)$, 
\item[(5.1.13)] 
   $\nu_1(A):={}^t(\overline{A})^{-1}$ for $A\in G^\mathbb{C}$, 
\item[(5.1.14)] 
   $\nu_2(A):=K_{1,1}\cdot{}^t(\overline{A})^{-1}\cdot K_{1,1}$ for 
$A\in G^\mathbb{C}$,  
\item[(5.1.15)] 
  $G^\mathbb{C}/H^\mathbb{C}
   =Sp(2,\mathbb{C})/(Sp(1,\mathbb{C})\times Sp(1,\mathbb{C}))$, 
\item[(5.1.16)] 
  $G_1/H_1=Sp(2)/(Sp(1)\times Sp(1))\simeq S^4$, 
\item[(5.1.17)] 
  $G_2/H_2=Sp(1,1)/(Sp(1)\times Sp(1))\simeq H^4$,    
\item[(5.1.18)] 
  $\pi_i$: the projection from $G_i$ onto $G_i/H_i$ ($i=1,2$), 
\item[(5.1.19)] 
  $\frak{g}_2:=\operatorname{Lie}G_2=\frak{sp}(1,1)$.     
\end{enumerate}
\setcounter{equation}{19}

   Define a $\widetilde{\Lambda}_{-1,\infty}(\frak{g}_2)_\sigma$-valued 
para-holomorphic $1$-form $\eta_\theta(x^1,x^2)$ on $(\mathbb{B}^4,I)$ by 
\[
  \eta_\theta(x^1,x^2):=\theta^{-1}
  \begin{pmatrix} 
  0 & 1 & 0 & 0 \\
  1 & 0 & 0 & 0 \\
  0 & 0 & 0 & -1 \\
  0 & 0 & -1 & 0 \\
  \end{pmatrix}dx^1 +\theta
  \begin{pmatrix} 
  0 & -1 & 0 & 0 \\
  -1 & 0 & 0 & 0 \\
  0 & 0 & 0 & 1 \\
  0 & 0 & 1 & 0 \\
  \end{pmatrix}dx^2.  
\]
   In view of the morphing condition \eqref{M}, it is natural that one defines 
a $\widetilde{\Lambda}_{-\infty,1}(\frak{g}_2)_\sigma$-valued 
para-antiholomorphic $1$-form $\tau_\theta(y^1,y^2)$ as follows:    
\[
  \tau_\theta(y^1,y^2):=\theta
  \begin{pmatrix} 
  0 & -1 & 0 & 0 \\
  -1 & 0 & 0 & 0 \\
  0 & 0 & 0 & 1 \\
  0 & 0 & 1 & 0 \\
  \end{pmatrix}dy^1 +\theta^{-1}
  \begin{pmatrix} 
  0 & 1 & 0 & 0 \\
  1 & 0 & 0 & 0 \\
  0 & 0 & 0 & -1 \\
  0 & 0 & -1 & 0 \\
  \end{pmatrix}dy^2.  
\]
   Let us solve the two initial value problems: 
$(A^-_\theta)^{-1}\cdot dA^-_\theta=\eta_\theta$ and 
$(A^+_\theta)^{-1}\cdot dA^+_\theta=\tau_\theta$ with 
$A^\pm_\theta(0,0)\equiv\operatorname{id}$, and factorize 
$(A^-_\theta,A^+_\theta)\in\widetilde{\Lambda}(G_2)_\sigma
  \times\widetilde{\Lambda}(G_2)_\sigma$    
in the Iwasawa decomposition (cf.\ Theorem \ref{thm-3.1.5}):   
$(A^-_\theta,A^+_\theta)=(C_\theta,C_\theta)\cdot(B^+_\theta,B^-_\theta)$,
where $C_\theta\in\widetilde{\Lambda}(G_2)_\sigma$, 
$B^+_\theta\in\widetilde{\Lambda}^+_*(G_2)_\sigma$ and 
$B^-_\theta\in\widetilde{\Lambda}^-(G_2)_\sigma$. 
   Then, it follows that 
\allowdisplaybreaks{
\begin{align*}
&
A^-_\theta(x^1,x^2)
 =\begin{pmatrix} 
 \cosh(\frac{x^1-\theta^2x^2}{\theta}) & \sinh(\frac{x^1-\theta^2x^2}{\theta}) & 0 & 0 \\
 \sinh(\frac{x^1-\theta^2x^2}{\theta}) & \cosh(\frac{x^1-\theta^2x^2}{\theta}) & 0 & 0 \\
 0 & 0 & \cosh(\frac{x^1-\theta^2x^2}{\theta}) & -\sinh(\frac{x^1-\theta^2x^2}{\theta}) \\
 0 & 0 & -\sinh(\frac{x^1-\theta^2x^2}{\theta}) & \cosh(\frac{x^1-\theta^2x^2}{\theta}) \\
 \end{pmatrix},\\
&
A^+_\theta(y^1,y^2)
 =\begin{pmatrix} 
 \cosh(\frac{\theta^2y^1-y^2}{\theta}) & -\sinh(\frac{\theta^2y^1-y^2}{\theta}) & 0 & 0 \\
 -\sinh(\frac{\theta^2y^1-y^2}{\theta}) & \cosh(\frac{\theta^2y^1-y^2}{\theta}) & 0 & 0 \\
 0 & 0 & \cosh(\frac{\theta^2y^1-y^2}{\theta}) & \sinh(\frac{\theta^2y^1-y^2}{\theta}) \\
 0 & 0 & \sinh(\frac{\theta^2y^1-y^2}{\theta}) & \cosh(\frac{\theta^2y^1-y^2}{\theta}) \\
 \end{pmatrix},\\
&
B^+_\theta(x^2,y^1)\\
&
 =\begin{pmatrix} 
 \cosh(\theta(x^2-y^1)) & -\sinh(\theta(x^2-y^1)) & 0 & 0 \\
 -\sinh(\theta(x^2-y^1)) & \cosh(\theta(x^2-y^1)) & 0 & 0 \\
 0 & 0 & \cosh(\theta(x^2-y^1)) & \sinh(\theta(x^2-y^1)) \\
 0 & 0 & \sinh(\theta(x^2-y^1)) & \cosh(\theta(x^2-y^1)) \\
 \end{pmatrix},\\
&
B^-_\theta(x^1,y^2)
 =\begin{pmatrix} 
 \cosh(\frac{x^1-y^2}{\theta}) & -\sinh(\frac{x^1-y^2}{\theta}) & 0 & 0 \\
 -\sinh(\frac{x^1-y^2}{\theta}) & \cosh(\frac{x^1-y^2}{\theta}) & 0 & 0 \\
 0 & 0 & \cosh(\frac{x^1-y^2}{\theta}) & \sinh(\frac{x^1-y^2}{\theta}) \\
 0 & 0 & \sinh(\frac{x^1-y^2}{\theta}) & \cosh(\frac{x^1-y^2}{\theta}) \\
 \end{pmatrix},\\ 
&
C_\theta(x^1,x^2,y^1,y^2)=
 \begin{pmatrix} 
 \cosh(\frac{x^1-\theta^2y^1}{\theta}) & \sinh(\frac{x^1-\theta^2y^1}{\theta}) & 0 & 0 \\
 \sinh(\frac{x^1-\theta^2y^1}{\theta}) & \cosh(\frac{x^1-\theta^2y^1}{\theta}) & 0 & 0 \\
 0 & 0 & \cosh(\frac{x^1-\theta^2y^1}{\theta}) & -\sinh(\frac{x^1-\theta^2y^1}{\theta}) \\
 0 & 0 & -\sinh(\frac{x^1-\theta^2y^1}{\theta}) & \cosh(\frac{x^1-\theta^2y^1}{\theta}) \\
 \end{pmatrix}.
\end{align*}
}   Substitute $\lambda$, $z^i$ and $\bar{z}^i$ for $\theta$, $x^i$ and $y^i$ 
($i=1,2$), respectively: 
\[
  C_\lambda(z^1,z^2,\bar{z}^1,\bar{z}^2)
=\begin{pmatrix}
 \cosh(\frac{z^1-\lambda^2\bar{z}^1}{\lambda}) & \sinh(\frac{z^1-\lambda^2\bar{z}^1}{\lambda}) & 0 & 0 \\
 \sinh(\frac{z^1-\lambda^2\bar{z}^1}{\lambda}) & \cosh(\frac{z^1-\lambda^2\bar{z}^1}{\lambda}) & 0 & 0 \\
 0 & 0 & \cosh(\frac{z^1-\lambda^2\bar{z}^1}{\lambda}) & -\sinh(\frac{z^1-\lambda^2\bar{z}^1}{\lambda}) \\
 0 & 0 & -\sinh(\frac{z^1-\lambda^2\bar{z}^1}{\lambda}) & \cosh(\frac{z^1-\lambda^2\bar{z}^1}{\lambda}) \\
 \end{pmatrix}
\] 
for $C_\theta(x^1,x^2,y^1,y^2)$. 
   Since $(z^1-\lambda^2\bar{z}^1)/\lambda$ is a purely imaginary number, one 
sees that $C_\lambda(z^1,z^2,\bar{z}^1,\bar{z}^2)\in G_1=Sp(2)$ for all 
$(z^1,z^2,\bar{z}^1,\bar{z}^2;\lambda)\in\mathbb{A}^4\times S^1$. 
   Accordingly, we obtain a pluriharmonic map $f_1$ and a para-pluriharmonic 
map $f_2$, 
\[
\begin{array}{ll} 
 (f_1)_\lambda=\pi_1\circ C_\lambda(z^1,z^2,\bar{z}^1,\bar{z}^2):
 (\mathbb{A}^4,J)\longrightarrow G_1/H_1\simeq S^4,  
& \lambda\in S^1,\\
 (f_2)_\theta=\pi_2\circ C_\theta(x^1,x^2,y^1,y^2):
 (\mathbb{B}^4,I)\longrightarrow G_2/H_2\simeq H^4, 
& \theta\in\mathbb{R}^+
\end{array}
\]
(ref.\ Subsection \ref{subsec-5.1.1}). 

\begin{center}
\fbox{
\begin{minipage}{130mm}
$(f_1)_\lambda=\pi_1\circ C_\lambda(z^1,z^2,\bar{z}^1,\bar{z}^2):
 (\mathbb{A}^4,J)\to S^4$ 
is pluriharmonic\par
\centerline{$\Updownarrow$} 
$(f_2)_\theta=\pi_2\circ C_\theta(x^1,x^2,y^1,y^2):(\mathbb{B}^4,I)\to H^4$ 
  is para-pluriharmonic 
\end{minipage}
}
\end{center}

\subsection{Harmonic maps, Lorentz harmonic maps and CMC-surfaces}
\label{subsec-5.2} 
   In this subsection we will interrelate some harmonic maps 
$f_1(z,\bar{z}):\mathbb{A}^2\to G_1/H_1$ with Lorentz harmonic maps 
$f_2(x,y):\mathbb{B}^2\to G_2/H_2$ by means of Theorem \ref{thm-4.3.1}; and in 
addition, we will interrelate CMC-surfaces with other CMC-surfaces in 
$\mathbb{R}^3$ or $\mathbb{R}^3_1$, by use of $f_1(z,\bar{z})$ and 
$f_2(x,y)$. 
   More precisely, we interrelate a cylinder in $\mathbb{R}^3$ with a 
hyperbolic cylinder in $\mathbb{R}^3_1$ (cf.\ Subsection \ref{subsec-5.2.1}), 
a two sheeted hyperboloid in $\mathbb{R}^3_1$ with a one sheeted hyperboloid 
in $\mathbb{R}^3_1$ (cf.\ Subsection \ref{subsec-5.2.2}),  a sphere in 
$\mathbb{R}^3$ with a one sheeted hyperboloid in $\mathbb{R}^3_1$ (cf.\ 
Subsection \ref{subsec-5.2.3}), a Smyth surface in $\mathbb{R}^3$ with a 
timelike Smyth surface in $\mathbb{R}^3_1$ (cf.\ Subsection 
\ref{subsec-5.2.4}), and a Delaunay surface in $\mathbb{R}^3$ with a 
$K$-surface of revolution in $\mathbb{R}^3$ (cf.\ Subsection 
\ref{subsec-5.2.5}).

\subsubsection{Cylinder in $\mathbb{R}^3$ $\Leftrightarrow$ Hyperbolic cylinder 
in $\mathbb{R}^3_1$}\label{subsec-5.2.1}  
   In this subsection we will use the following notation: 
\begin{enumerate}[(5.2.1)]
\item 
  $G^\mathbb{C}=SL(2,\mathbb{C})$, 
\item 
  $\sigma(A):=I_{1,1}\cdot A\cdot I_{1,1}$ for $A\in G^\mathbb{C}$, where 
$I_{1,1}:=\operatorname{diag}(-1,1)$, 
\item 
  $\nu_1(A):={}^t(\overline{A})^{-1}$ for $A\in G^\mathbb{C}$, 
\item 
  $\nu_2(A):=\overline{A}$ for $A\in G^\mathbb{C}$,
\item 
  $G^\mathbb{C}/H^\mathbb{C}
=SL(2,\mathbb{C})/S(GL(1,\mathbb{C})\times GL(1,\mathbb{C}))$,
\item 
  $G_1/H_1=SU(2)/S(U(1)\times U(1))\simeq S^2$, 
\item 
  $G_2/H_2=SL(2,\mathbb{R})/S(GL(1,\mathbb{R})\times GL(1,\mathbb{R}))
      \simeq S^2_1$, 
\item 
  $\pi_i$: the projection from $G_i$ onto $G_i/H_i$ ($i=1,2$), 
\item 
  $\frak{g}_2:=\operatorname{Lie}G_2=\frak{sl}(2,\mathbb{R})$. 
\end{enumerate}
\setcounter{equation}{9} 

   We will construct a harmonic map $f_1:(\mathbb{A}^2,J)\to S^2$ and a Lorentz 
harmonic map $f_2:(\mathbb{B}^2,I)\to S^2_1$ by means of Theorem \ref{thm-4.3.1}; 
and moreover, a cylinder in $\mathbb{R}^3$ and a hyperbolic cylinder in 
$\mathbb{R}^3_1$ from $f_1$ and $f_2$, respectively.  

   In the first place, we define a 
$\widetilde{\Lambda}_{-1,\infty}(\frak{g}_2)_\sigma$-valued, real analytic 
para-holomorphic $1$-form $\eta_\theta(x)$ on $(\mathbb{B}^2,I)$ by 
\begin{equation}\label{eq-5.2.1}
  \eta_\theta(x):=\theta^{-1}
    \begin{pmatrix} 
    0 & 1 \\
    1 & 0
    \end{pmatrix}dx. 
\end{equation}
   In the second place, we define a 
$\widetilde{\Lambda}_{-\infty,1}(\frak{g}_2)_\sigma$-valued 
para-antiholomorphic $1$-form $\tau_\theta(y)$ on $(\mathbb{B}^2,I)$ by taking 
the morphing condition \eqref{M} in Theorem \ref{thm-4.3.1} into 
consideration, i.e., 
\[
  \tau_\theta(y):=\theta
    \begin{pmatrix} 
     0 & -1 \\
    -1 & 0
    \end{pmatrix}dy.  
\] 
   In the third place, let us solve the two initial value problems: 
$A_\theta^{-1}\cdot dA_\theta=\eta_\theta(x)$, 
$B_\theta^{-1}\cdot dB_\theta=\tau_\theta(y)$ and 
$A_\theta(0)\equiv\operatorname{id}\equiv B_\theta(0)$.  
   In this case, one can obtain  
\[
\begin{array}{ll}
   A_\theta(x)
=\begin{pmatrix}
  \cosh(\theta^{-1}x) & \sinh(\theta^{-1}x) \\
  \sinh(\theta^{-1}x) & \cosh(\theta^{-1}x)  
 \end{pmatrix},  
& 
   B_\theta(y)
=\begin{pmatrix}
  \cosh(-\theta y) & \sinh(-\theta y) \\
  \sinh(-\theta y) & \cosh(-\theta y)  
 \end{pmatrix}
\end{array}     
\] 
and the Iwasawa decomposition:  
$(A_\theta(x),B_\theta(y))=(C_\theta(x,y),C_\theta(x,y))
 \cdot(B^+_\theta(x,y),B^-_\theta(x,y))$, 
where $B^+_\theta(x,y):=B_\theta(y)^{-1}\in\widetilde{\Lambda}^+_*(G_2)_\sigma$ 
and $B^-_\theta(x,y):=A_\theta(x)^{-1}\in\widetilde{\Lambda}^-_*(G_2)_\sigma$. 
   Here $C_\theta(x,y)$ is given as follows: 
\begin{equation}\label{eq-5.2.2}
  C_\theta(x,y)
=\begin{pmatrix}
  \cosh(\theta^{-1}x-\theta y) & \sinh(\theta^{-1}x-\theta y) \\
  \sinh(\theta^{-1}x-\theta y) & \cosh(\theta^{-1}x-\theta y)  
 \end{pmatrix}. 
\end{equation}
   This $C_\theta(x,y)$ provides us with an $\mathbb{R}^+$-family of Lorentz 
harmonic maps 
\[
\begin{array}{ll}
 (f_2)_\theta=\pi_2\circ C_\theta(x,y):(\mathbb{B}^2,I)
  \longrightarrow G_2/H_2\simeq S^2_1, 
& \theta\in\mathbb{R}^+ 
\end{array}  
\] 
(cf.\ Proposition \ref{prop-3.2.3}).   
   In the fourth place, we substitute $\lambda$, $z$ and $\bar{z}$ for $\theta$, 
$x$ and $y$, respectively: 
\begin{equation}\label{eq-5.2.3}
  C_\lambda(z,\bar{z})
=\begin{pmatrix}
  \cosh(\lambda^{-1}z-\lambda\bar{z}) & \sinh(\lambda^{-1}z-\lambda\bar{z}) \\
  \sinh(\lambda^{-1}z-\lambda\bar{z}) & \cosh(\lambda^{-1}z-\lambda\bar{z})  
 \end{pmatrix}
\end{equation}
for $C_\theta(x,y)$. 
   Remark that $C_\lambda(z,\bar{z})\in G_1=SU(2)$ for all 
$(z,\bar{z};\lambda)\in \mathbb{A}^2\times S^1$ because 
$(\lambda^{-1}z-\lambda\bar{z})$ is a purely imaginary number. 
   As a consequence, one can construct a harmonic map $f_1$ and a Lorentz 
harmonic map $f_2$, 
\[
\begin{array}{l@{\,}ll} 
 (f_1)_\lambda=\pi_1\circ C_\lambda(z,\bar{z})&:(\mathbb{A}^2,J)
  \longrightarrow G_1/H_1\simeq S^2,  
& \lambda\in S^1,\\
 (f_2)_\theta=\pi_2\circ C_\theta(x,y)&:(\mathbb{B}^2,I)
  \longrightarrow G_2/H_2\simeq S^2_1, 
& \theta\in\mathbb{R}^+,\\ 
\end{array}
\]
from the potential \eqref{eq-5.2.1} $\eta_\theta(x)$. 

\begin{center}
\fbox{
\begin{minipage}{100mm}
$(f_1)_\lambda=\pi_1\circ C_\lambda(z,\bar{z}):(\mathbb{A}^2,J)\to S^2$ 
is harmonic\par
\centerline{$\Updownarrow$} 
$(f_2)_\theta=\pi_2\circ C_\theta(x,y):(\mathbb{B}^2,I)\to S^2_1$ 
is Lorentz harmonic 
\end{minipage}
}
\end{center}
   Here $C_\lambda(z,\bar{z})$ and $C_\theta(x,y)$ are given by 
\eqref{eq-5.2.3} and \eqref{eq-5.2.2}, respectively.\par 

   Note that we have constructed the extended framing 
$C_\lambda(z,\bar{z}):\mathbb{A}^2\to S^2$ of a harmonic map and the extended 
framing $C_\theta(x,y):\mathbb{B}^2\to S^2_1$ of a Lorentz harmonic map.  
   For this reason, the Sym-Bobenko formula will enable us to obtain a 
CMC-surface $\phi_1(z,\bar{z}):\mathbb{A}^2\to\mathbb{R}^3$ and a timelike 
CMC-surface $\phi_2(x,y):\mathbb{B}^2\to\mathbb{R}^3_1$ from them. 

   For $C_\lambda(z,\bar{z})$, the Sym-Bobenko formula in 
\cite[p.\ 30]{Fu-Ko-Ro} yields 
\[
\begin{split}
  \phi_1(z,\bar{z}):
&=-
  \left\{i\cdot\lambda\cdot\frac{\partial C_\lambda}{\partial\lambda}
     \cdot C_\lambda^{-1} 
      +\frac{1}{2}\cdot\operatorname{Ad}(C_\lambda)\cdot
     \begin{pmatrix}
      i & 0 \\
      0 & -i \\
     \end{pmatrix}\right\}\bigg|_{\lambda=1}\\
&=\frac{-i}{2}
   \begin{pmatrix} 
    \cosh 2(z-\bar{z}) & -2(z+\bar{z})-\sinh 2(z-\bar{z})\\
    -2(z+\bar{z})+\sinh 2(z-\bar{z}) & -\cosh 2(z-\bar{z})
   \end{pmatrix}\\
&\simeq (-2(z+\bar{z}),i\cdot\sinh 2(z-\bar{z}),-\cosh 2(z-\bar{z})). 
\end{split}
\] 
   This CMC-surface $\phi_1(z,\bar{z}):\mathbb{A}^2\to\mathbb{R}^3$ is a 
cylinder.  
   For $C_\theta(x,y)$, the Sym-Bobenko formula in \cite{Do-In-To}\footnote{
We must change $(\partial \Phi/\partial t)$ into $-(\partial \Phi/\partial t)$ 
in the Sym-Bobenko formula \cite[Proposition 5.1]{Do-In-To} because the 
parameter $\lambda$ in this paper corresponds to the parameter $\lambda^{-1}$ 
in \cite{Do-In-To}.} is given as follows:    
\[
\begin{split}
  \phi_2(x,y):
&=-2
  \left\{-\theta\cdot\frac{\partial C_\theta}{\partial\theta}
     \cdot C_\theta^{-1} 
      +\frac{1}{2}\cdot\operatorname{Ad}(C_\theta)\cdot
     \begin{pmatrix}
     -1 & 0 \\
      0 & 1 \\
     \end{pmatrix}\right\}\bigg|_{\theta=1}\\
&=
   \begin{pmatrix} 
    \cosh 2(x-y) & -2(x+y)-\sinh 2(x-y)\\
    -2(x+y)+\sinh 2(x-y) & -\cosh 2(x-y)
   \end{pmatrix}\\
&\simeq(\sinh 2(x-y),-2(x+y),-\cosh 2(x-y)). 
\end{split}
\]  
   This timelike CMC-surface $\phi_2(x,y):\mathbb{B}^2\to\mathbb{R}^3_1$ is a 
hyperbolic cylinder, because 
$-(\sinh 2(x-y))^2+(-2(x+y))^2+(-\cosh 2(x-y))^2=4(x+y)^2+1$ (see Section 1.1 
in \cite{Do-In-To} for the metric on $\mathbb{R}^3_1$).  

\begin{center}
\fbox{
\begin{minipage}{95mm}
CMC-surface in $\mathbb{R}^3$: Cylinder\\
$\phi_1(z,\bar{z})=
   (-2(z+\bar{z}),i\cdot\sinh 2(z-\bar{z}),-\cosh 2(z-\bar{z}))$\par
\centerline{$\Updownarrow$} 
Timelike CMC-surface in $\mathbb{R}^3_1$: Hyperbolic cylinder\\ 
  $\phi_2(x,y)=(\sinh 2(x-y),-2(x+y),-\cosh 2(x-y))$ 
\end{minipage}
}
\end{center}

\subsubsection{Two sheeted hyperboloid in $\mathbb{R}^3_1$ $\Leftrightarrow$ 
One sheeted hyperboloid in $\mathbb{R}^3_1$}\label{subsec-5.2.2}  
   In this subsection we will use the following notation: 
\begin{enumerate}
\item[(5.2.13)] 
  $G^\mathbb{C}$: the same notation (5.2.1) as in Subsection \ref{subsec-5.2.1}, 
\item[(5.2.14)]  
  $\sigma$: the same notation (5.2.2) as in Subsection \ref{subsec-5.2.1}, 
\item[(5.2.15)]  
  $\nu_1(A):=I_{1,1}\cdot{}^t(\overline{A})^{-1}\cdot I_{1,1}$ for 
$A\in G^\mathbb{C}$, 
\item[(5.2.16)]  
  $\nu_2(A):=I_{1,1}\cdot\overline{A}\cdot I_{1,1}$ for $A\in G^\mathbb{C}$,
\item[(5.2.17)]  
  $G^\mathbb{C}/H^\mathbb{C}$: the same notation (5.2.5) as in Subsection 
\ref{subsec-5.2.1}, 
\item[(5.2.18)]  
  $G_1/H_1=SU(1,1)/S(U(1)\times U(1))\simeq H^2$, 
\item[(5.2.19)]  
  $G_2/H_2=SL_*(2,\mathbb{R})/S(GL(1,\mathbb{R})\times GL(1,\mathbb{R}))
\simeq S^2_1$, 
\item[(5.2.20)]  
  $\pi_i$: the projection from $G_i$ onto $G_i/H_i$ ($i=1,2$), 
\item[(5.2.21)]  
  $\frak{g}_2:=\operatorname{Lie}G_2=\frak{sl}_*(2,\mathbb{R})$.            
\end{enumerate}
where the above notation $SL_*(2,\mathbb{R})$ and $\frak{sl}_*(2,\mathbb{R})$ 
are the same as those in \cite{Ko}. 
\setcounter{equation}{21}

     The arguments below are similar to those in Subsection \ref{subsec-5.2.1}. 
   Define a $\widetilde{\Lambda}_{-1,\infty}(\frak{g}_2)_\sigma$-valued 
analytic para-holomorphic $1$-form $\eta_\theta(x)$ on $(\mathbb{B}^2,I)$ by   
\begin{equation}\label{eq-5.2.4}
\begin{array}{ll}
  \eta_\theta(x):=\theta^{-1}
    \begin{pmatrix} 
    0 & i \\
    0 & 0
    \end{pmatrix}dx. 
\end{array} 
\end{equation} 
   We want $\tau_\theta(y)\in\widetilde{\mathcal{P}}^-(\frak{g}_2)$ to satisfy 
the morphing condition \eqref{M} in Theorem \ref{thm-4.3.1}; and therefore we 
define $\tau_\theta(y)$ as follows:  
\[ 
  \tau_\theta(y):=\theta
    \begin{pmatrix} 
     0 & 0 \\
    -i & 0
    \end{pmatrix}dy.
\] 
   Solve the two initial value problems: 
$A_\theta^{-1}\cdot dA_\theta=\eta_\theta(x)$, 
$B_\theta^{-1}\cdot dB_\theta=\tau_\theta(y)$ and 
$A_\theta(0)\equiv\operatorname{id}\equiv B_\theta(0)$. 
   Then one has 
\[
\begin{array}{ll}
   A_\theta(x)
=\begin{pmatrix}
  1 & i\theta^{-1}x \\
  0 & 1  
 \end{pmatrix},  
& 
   B_\theta(y)
=\begin{pmatrix}
  1 & 0 \\
  -i\theta y & 1  
 \end{pmatrix}.
\end{array}     
\] 
   Let us factorize 
$(A_\theta,B_\theta)\in\widetilde{\Lambda}(G_2)_\sigma\times
 \widetilde{\Lambda}(G_2)_\sigma$ 
in the Iwasawa decomposition around $(0,0)$:  
$(A_\theta,B_\theta)=(C_\theta,C_\theta)\cdot(B^+_\theta,B^-_\theta)$, 
$C_\theta\in\widetilde{\Lambda}(G_2)_\sigma$ and 
$B^\pm_\theta\in\widetilde{\Lambda}^\pm(G_2)_\sigma$ (cf.\ Theorem 
\ref{thm-3.1.5}). 
   Here  $B^\pm_\theta$ and $C_\theta$ are given as follows: 
\[
\begin{array}{ll}
   B^+_\theta(x,y)
={\displaystyle \frac{1}{\sqrt{1-xy}}}
 \begin{pmatrix}
  1 & 0 \\
  i\theta y & 1-xy  
 \end{pmatrix}, 
& 
   B^-_\theta(x,y)
={\displaystyle \frac{1}{\sqrt{1-xy}}}
 \begin{pmatrix}
  1-xy & -i\theta^{-1} x \\
  0 & 1  
 \end{pmatrix}, 
\end{array}     
\]   
\begin{equation}\label{eq-5.2.5}
   C_\theta(x,y)
={\displaystyle \frac{1}{\sqrt{1-xy}}}
 \begin{pmatrix}
  1 & i\theta^{-1}x \\
  -i\theta y & 1  
 \end{pmatrix}.  
\end{equation}
  From $C_\theta(x,y)$ one obtains an $\mathbb{R}^+$-family of Lorentz 
harmonic maps 
\[
 (f_2)_\theta=\pi_2\circ C_\theta(x,y):(W,I)
  \longrightarrow G_2/H_2\simeq S^2_1, 
\]  
where $W:=\{(x,y)\in\mathbb{B}^2\,|\, xy\neq 1\}$.  
   Substituting $\lambda$, $z$ and $\bar{z}$ for $\theta$, $x$ and $y$, 
respectively, we have 
\begin{equation}\label{eq-5.2.6}
   C_\lambda(z,\bar{z})
={\displaystyle \frac{1}{\sqrt{1-|z|^2}}}
 \begin{pmatrix}
  1 & i\lambda^{-1} z \\
  -i\lambda\bar{z} & 1  
 \end{pmatrix} 
\end{equation}
for $C_\theta(x,y)$. 
   It is obvious that $C_\lambda(z,\bar{z})\in G_1=SU(1,1)$ for all 
$(z,\bar{z};\lambda)\in V\times S^1$, where $V:=\mathbb{A}^2\setminus S^1$. 
   Consequently, one can get a harmonic map $f_1(z,\bar{z})$ and a Lorentz 
harmonic map $f_2(x,y)$, 
\[
\begin{array}{ll} 
 (f_1)_\lambda=\pi_1\circ C_\lambda(z,\bar{z}):(V,J)
  \longrightarrow G_1/H_1\simeq H^2,  
& \lambda\in S^1,\\
 (f_2)_\theta=\pi_2\circ C_\theta(x,y):(W,I)
  \longrightarrow G_2/H_2\simeq S^2_1, 
& \theta\in\mathbb{R}^+,
\end{array}
\]
from the potential \eqref{eq-5.2.4}. 

\begin{center}
\fbox{
\begin{minipage}{100mm}
$(f_1)_\lambda=\pi_1\circ C_\lambda(z,\bar{z}):(V,J)\to H^2$ is harmonic\par
\centerline{$\Updownarrow$} 
$(f_2)_\theta=\pi_2\circ C_\theta(x,y):(W,I)\to S^2_1$ is Lorentz harmonic 
\end{minipage}
}
\end{center}
   Here $C_\lambda(z,\bar{z})$ and $C_\theta(x,y)$ are given by 
\eqref{eq-5.2.6} and \eqref{eq-5.2.5}, respectively; and 
$V=\mathbb{A}^2\setminus S^1$ and $W=\{(x,y)\in\mathbb{B}^2\,|\, xy\neq 1\}$.\par 

   Now, let us obtain a spacelike CMC-surface 
$\phi_1(z,\bar{z}):V\to\mathbb{R}^3_1$ and 
a timelike CMC-surface $\phi_2(x,y):W\to\mathbb{R}^3_1$ from the above 
$f_1(z,\bar{z})$ and $f_2(x,y)$, respectively.\par
   
   On the one hand, the Sym-Bobenko formula in \cite{Br-Ro-Sc}, together with 
\eqref{eq-5.2.6}, gives us 
\[
\begin{split}
  \phi_1(z,\bar{z}):
&=-
  \left\{i\cdot\lambda\cdot\frac{\partial C_\lambda}{\partial\lambda}
     \cdot C_\lambda^{-1} 
      +\frac{1}{2}\cdot\operatorname{Ad}(C_\lambda)\cdot
     \begin{pmatrix}
      i & 0 \\
      0 & -i \\
     \end{pmatrix}\right\}\bigg|_{\lambda=1}\\
&=
   \begin{pmatrix} 
    -i(1+3|z|^2)/2(1-|z|^2) & -2z/(1-|z|^2)\\
    -2\bar{z}/(1-|z|^2) & i(1+3|z|^2)/2(1-|z|^2)
   \end{pmatrix}. 
\end{split}
\]
   Thus we have a spacelike CMC-surface in $\mathbb{R}^3_1$,   
\[
\begin{array}{ll}
 \phi_1:V\longrightarrow\mathbb{R}^3_1, 
& {\displaystyle  
(z,\bar{z})\mapsto 
   \Bigl(-\frac{z+\bar{z}}{1-|z|^2},
          \frac{i(z-\bar{z})}{1-|z|^2},
         -\frac{1+3|z|^2}{2(1-|z|^2)}\Bigr)}
\end{array}
\]
(cf.\ Subsection 3.2.1 in \cite{Br-Ro-Sc}). 
   This $\phi_1(z,\bar{z})$ is a two sheeted hyperboloid centered at 
$(0,0,1/2)$ because 
\[
  \Bigl(-\frac{z+\bar{z}}{1-|z|^2}\Bigr)^2
 +\Bigl(\frac{i(z-\bar{z})}{1-|z|^2}\Bigr)^2 
 -\Bigl(-\frac{1+3|z|^2}{2(1-|z|^2)}-\frac{1}{2}\Bigr)^2
 =-1
\]
(see Subsection 3.2.1 in \cite{Br-Ro-Sc} for the metric on $\mathbb{R}^3_1$).
   One the other hand, the Sym-Bobenko formula in \cite{Ko}, combined with 
\eqref{eq-5.2.5}, gives us 
\[
\begin{split}
  \phi_2(x,y)
&:=-{\displaystyle \frac{1}{2}}
  \left\{\theta\cdot\frac{\partial C_\theta}{\partial\theta}
    \cdot C_\theta^{-1} 
      +\frac{1}{2}\cdot\operatorname{Ad}(C_\theta)\cdot
     \begin{pmatrix}
      1 & 0 \\
      0 & -1 \\
     \end{pmatrix}\right\}\bigg|_{\theta=1}\\
&=
   \begin{pmatrix} 
    -(1+3xy)/4(1-xy) & ix/(1-xy)\\
    iy/(1-xy) & (1+3xy)/4(1-xy)
   \end{pmatrix} 
\end{split}
\]
(ref.\ Proof of Corollary 3.4 in \cite{Ko}).  
   Then it turns out that 
\[
\begin{array}{ll}
 \phi_2:W\longrightarrow\mathbb{R}^3_1, 
& {\displaystyle 
 (x,y)\mapsto 
   \Bigl(-\frac{x+y}{1-xy},-\frac{x-y}{1-xy},-\frac{1+3xy}{2(1-xy)}\Bigr)}
\end{array}
\]
(cf.\ Subsection 3.1 in \cite{Ko}).  
   This $\phi_2(x,y)$ is a one sheeted hyperboloid centered at $(0,0,1/2)$. 
   Indeed, we deduce
\[
\Bigl(-\frac{x+y}{1-xy}\Bigr)^2
 -\Bigl(-\frac{x-y}{1-xy}\Bigr)^2
  -\Bigl(-\frac{1+3xy}{2(1-xy)}-\frac{1}{2}\Bigr)^2
 =-1
\]
by a direct computation (see Remark 3.2 in \cite{Ko} for the metric on 
$\mathbb{R}^3_1$). 

\begin{center}
\fbox{
\begin{minipage}{105mm}
Spacelike CMC-surface in $\mathbb{R}^3_1$: Two sheeted hyperboloid\\
$\phi_1(z,\bar{z})=\bigl(-\frac{z+\bar{z}}{1-|z|^2},
 \frac{i(z-\bar{z})}{1-|z|^2},-\frac{1+3|z|^2}{2(1-|z|^2)}\bigr)$\par
\centerline{$\Updownarrow$} 
Timelike CMC-surface in $\mathbb{R}^3_1$: One sheeted hyperboloid \\ 
$\phi_2(x,y)
 =\bigl(-\frac{x+y}{1-xy},-\frac{x-y}{1-xy},-\frac{1+3xy}{2(1-xy)}\bigr)$
\end{minipage}
}
\end{center}

\subsubsection{Sphere in $\mathbb{R}^3$ $\Leftrightarrow$ One sheeted 
hyperboloid in $\mathbb{R}^3_1$}\label{subsec-5.2.3}  
   In this subsection, we utilize the same potential as in Subsection  
\ref{subsec-5.2.2}, but we will obtain other CMC-surfaces. 
   For this we will use the following notation:  
\begin{enumerate}
\item[(5.2.25)] 
   $G^\mathbb{C}$: the same notation (5.2.1) as in Subsection \ref{subsec-5.2.1}, 
\item[(5.2.26)] 
   $\sigma$: the same notation (5.2.2) as in Subsection \ref{subsec-5.2.1}, 
\item[(5.2.27)] 
   $\nu_1$: the same notation (5.2.3) as in Subsection \ref{subsec-5.2.1}, 
\item[(5.2.28)] 
   $\nu_2$: the same notation (5.2.16) as in Subsection \ref{subsec-5.2.2},   
\item[(5.2.29)] 
   $G^\mathbb{C}/H^\mathbb{C}$: the same notation (5.2.5) as in Subsection 
\ref{subsec-5.2.1}, 
\item[(5.2.30)]
   $G_1/H_1$: the same notation (5.2.6) as in Subsection \ref{subsec-5.2.1}, 
\item[(5.2.31)] 
   $G_2/H_2$: the same notation (5.2.19) as in Subsection \ref{subsec-5.2.2}, 
\item[(5.2.32)] 
   $\pi_i$: the projection from $G_i$ onto $G_i/H_i$ ($i=1,2$), 
\item[(5.2.33)] 
   $\frak{g}_2$: the same notation (5.2.21) as in Subsection \ref{subsec-5.2.2}.           
\end{enumerate}
\setcounter{equation}{33}

   Let $\eta_\theta(x)$ denote the potential \eqref{eq-5.2.4}. 
   Define $\tau_\theta(y)\in\widetilde{\mathcal{P}}^-(\frak{g}_2)$ by    
\[
  \tau_\theta(y):=\theta
    \begin{pmatrix} 
     0 & 0 \\
     i & 0
    \end{pmatrix}dy.  
\]
   Here we remark that $(\eta_\theta(x),\tau_\theta(y))$ is a real analytic 
para-pluriharmonic potential on $(\mathbb{B}^2,I)$ satisfying the morphing 
condition \eqref{M}.   
   Solve the two initial value problems: 
$A_\theta^{-1}\cdot dA_\theta=\eta_\theta(x)$, 
$B_\theta^{-1}\cdot dB_\theta=\tau_\theta(y)$ and 
$A_\theta(0)\equiv\operatorname{id}\equiv B_\theta(0)$; and factorize 
$(A_\theta,B_\theta)\in\widetilde{\Lambda}(G_2)_\sigma\times
 \widetilde{\Lambda}(G_2)_\sigma$ 
in the Iwasawa decomposition (cf.\ Theorem \ref{thm-3.1.5}):  
$(A_\theta,B_\theta)=(C_\theta,C_\theta)\cdot(B^+_\theta,B^-_\theta)$, 
$C_\theta\in\widetilde{\Lambda}(G_2)_\sigma$ and 
$B^\pm_\theta\in\widetilde{\Lambda}^\pm(G_2)_\sigma$. 
   In this case it follows that 
\[
\begin{array}{ll}
   A_\theta(x)
=\begin{pmatrix}
  1 & i\theta^{-1}x \\
  0 & 1  
 \end{pmatrix},  
& 
   B_\theta(y)
=\begin{pmatrix}
  1 & 0 \\
  i\theta y & 1  
 \end{pmatrix};\\
   B^+_\theta(x,y)
={\displaystyle \frac{1}{\sqrt{1+xy}}}
 \begin{pmatrix}
  1 & 0 \\
  -i\theta y & 1+xy  
 \end{pmatrix}, 
& 
   B^-_\theta(x,y)
={\displaystyle \frac{1}{\sqrt{1+xy}}}
 \begin{pmatrix}
  1+xy & -i\theta^{-1} x \\
  0 & 1  
 \end{pmatrix};\\
   C_\theta(x,y)
={\displaystyle \frac{1}{\sqrt{1+xy}}}
 \begin{pmatrix}
  1 & i\theta^{-1}x \\
  i\theta y & 1  
 \end{pmatrix}. 
& 
\end{array} 
\]
   It is easy to see that $C_\lambda(z,\bar{z})\in G_1=SU(2)$ for all 
$(z,\bar{z};\lambda)\in \mathbb{A}^2\times S^1$. 
   Accordingly, we obtain a harmonic map $f_1$ and a Lorentz harmonic map 
$f_2$, 
\[
\begin{array}{ll} 
 (f_1)_\lambda=\pi_1\circ C_\lambda(z,\bar{z}):(\mathbb{A}^2,J)
  \longrightarrow G_1/H_1\simeq S^2,  
& \lambda\in S^1,\\
 (f_2)_\theta=\pi_2\circ C_\theta(x,y):(W,I)
  \longrightarrow G_2/H_2\simeq S^2_1, 
& \theta\in\mathbb{R}^+,\\
\end{array}
\]
from \eqref{eq-5.2.4}. 
   Here $W:=\{(x,y)\in\mathbb{B}^2\,|\,xy\neq -1\}$. 
   The above maps will provide us with a CMC-surface 
$\phi_1:\mathbb{A}^2\to\mathbb{R}^3$ and a timelike CMC-surface 
$\phi_2:W\to\mathbb{R}^3_1$. 
   The Sym-Bobenko formula in \cite{Fu-Ko-Ro}, combined with 
$C_\lambda(z,\bar{z})$, gives 
\[
\begin{split}
  \phi_1(z,\bar{z}):
&=-
  \left\{i\cdot\lambda\cdot\frac{\partial C_\lambda}{\partial\lambda}
     \cdot C_\lambda^{-1} 
      +\frac{1}{2}\cdot\operatorname{Ad}(C_\lambda)\cdot
     \begin{pmatrix}
      i & 0 \\
      0 & -i \\
     \end{pmatrix}\right\}\bigg|_{\lambda=1}\\
&={\displaystyle \frac{-i}{2}}
   \begin{pmatrix} 
    (1-3|z|^2)/(1+|z|^2) & -4iz/(1+|z|^2)\\
    4i\bar{z}/(1+|z|^2) & -(1-3|z|^2)/(1+|z|^2)
   \end{pmatrix}\\
&\simeq 
 \Bigl(\frac{-2i(z-\bar{z})}{1+|z|^2},\frac{-2(z+\bar{z})}{1+|z|^2},
           \frac{-1+3|z|^2}{1+|z|^2}\Bigr). 
\end{split}
\]
   This CMC-surface $\phi_1(z,\bar{z}):\mathbb{A}^2\to\mathbb{R}^3$ is a 
sphere centered at $(0,0,1)$. 
   By the above $C_\theta(x,y)$ and the Sym-Bobenko formula in \cite{Ko}, we 
obtain 
\[
\begin{split}
  \phi_2(x,y):
&=-{\displaystyle \frac{1}{2}}
  \left\{\theta\cdot\frac{\partial C_\theta}{\partial\theta}
     \cdot C_\theta^{-1} 
      +\frac{1}{2}\cdot\operatorname{Ad}(C_\theta)\cdot
     \begin{pmatrix}
      1 & 0 \\
      0 & -1 \\
     \end{pmatrix}\right\}\bigg|_{\theta=1}\\
&=
   \begin{pmatrix} 
    -(1-3xy)/4(1+xy) & ix/(1+xy)\\
    -iy/(1+xy) & (1-3xy)/4(1+xy)
   \end{pmatrix}\\
&\simeq \Bigl(-\frac{x-y}{1+xy},-\frac{x+y}{1+xy},-\frac{1-3xy}{2(1+xy)}\Bigr).  
\end{split}
\]
   This timelike CMC-surface $\phi_2(x,y):W\to\mathbb{R}^3_1$ is a one sheeted 
hyperboloid centered at $(0,0,1/2)$ because  
\[
\Bigl(-\frac{x-y}{1+xy}\Bigr)^2
 -\Bigl(-\frac{x+y}{1+xy}\Bigr)^2
  -\Bigl(-\frac{1-3xy}{2(1+xy)}-\frac{1}{2}\Bigr)^2
 =-1
\]
(see Remark 3.2 in \cite{Ko} for the metric on $\mathbb{R}^3_1$).

\begin{center}
\fbox{
\begin{minipage}{100mm}
CMC-surface in $\mathbb{R}^3$: Sphere\\
$\phi_1(z,\bar{z})= \bigl(\frac{-2i(z-\bar{z})}{1+|z|^2},
          \frac{-2(z+\bar{z})}{1+|z|^2},\frac{-1+3|z|^2}{1+|z|^2}\bigr)$\par
\centerline{$\Updownarrow$} 
Timelike CMC-surface in $\mathbb{R}^3_1$: One sheeted hyperboloid \\ 
$\phi_2(x,y)=\bigl(-\frac{x-y}{1+xy},-\frac{x+y}{1+xy},
 -\frac{1-3xy}{2(1+xy)}\bigr)$ 
\end{minipage}
}
\end{center}

\subsubsection{Smyth surface in $\mathbb{R}^3$ $\Leftrightarrow$ 
Timelike Smyth surface in $\mathbb{R}^3_1$}\label{subsec-5.2.4} 
   In this subsection we construct a timelike CMC-surface, 
$\phi_2(x,y):W\to\mathbb{R}^3_1$, from the potential of Smyth surface in 
$\mathbb{R}^3$ (cf.\ \eqref{eq-5.2.7}); and we study the relation between the 
Gau{\ss} equation for $\phi_2(x,y):W\to\mathbb{R}^3_1$ and the Painlev\'{e} 
equation of type (III).   
   Henceforth we will use the same notation as in Subsection \ref{subsec-5.2.1}. 
   
   Define a $\widetilde{\Lambda}_{-1,\infty}(\frak{g}_2)_\sigma$-valued, real 
analytic para-holomorphic $1$-form $\eta_\theta(x)$ on $(\mathbb{B}^2,I)$ by 
\begin{equation}\label{eq-5.2.7}
\begin{array}{ll}
  \eta_\theta(x):=\theta^{-1}
    \begin{pmatrix} 
    0 & 1 \\
    x^m & 0
    \end{pmatrix}dx, 
\end{array} 
\end{equation}
where $m\in\mathbb{N}$.
   Taking the morphing condition \eqref{M} into consideration, we define 
$\tau_\theta(y)$ as follows: 
\[
  \tau_\theta(x):=\theta
    \begin{pmatrix} 
    0 & -y^m \\
    -1 & 0
    \end{pmatrix}dy.
\]
   Solve the two initial value problems: 
$A_\theta^{-1}\cdot dA_\theta=\eta_\theta(x)$, 
$B_\theta^{-1}\cdot dB_\theta=\tau_\theta(y)$ and 
$A_\theta(0)\equiv\operatorname{id}\equiv B_\theta(0)$. 
   In terms of Theorem \ref{thm-3.1.5} we factorize 
$(A_\theta,B_\theta)\in\widetilde{\Lambda}^-_*(G_2)_\sigma\times
 \widetilde{\Lambda}^+_*(G_2)_\sigma$ 
as follows: 
$(A_\theta,B_\theta)
  =(C_\theta,C_\theta)\cdot(B^+_\theta,B^-_\theta)$,  
$C_\theta\in\widetilde{\Lambda}(G_2)_\sigma$ and 
$B^+_\theta\in\widetilde{\Lambda}^+_*(G_2)_\sigma$ and 
$B^-_\theta\in\widetilde{\Lambda}^-(G_2)_\sigma$. 
   Then Proposition \ref{prop-3.2.3} enables us to obtain an 
$\mathbb{R}^+$-family of Lorentz harmonic maps 
\[
 (f_2)_\theta=\pi_2\circ C_\theta(x,y):(W,I)
 \longrightarrow G_2/H_2\simeq S^2_1,   
\]
where $W$ is an open neighborhood of $\mathbb{B}^2$ at $(0,0)$. 
   Furthermore, Theorem \ref{thm-4.3.1} tells us that there is an $S^1$-family 
of harmonic maps 
\[
\begin{array}{ll}
 (f_1)_\lambda=\pi_1\circ C'_\lambda(z,\bar{z}):(V,J)
  \longrightarrow G_1/H_1\simeq S^2, 
&  C'_\lambda(0,0)\equiv\operatorname{id}, 
\end{array}  
\]
where 
$C'_\lambda(z,\bar{z}):=C_\lambda(z,\bar{z})\cdot h^\mathbb{C}(z,\bar{z})$ 
(see Theorem \ref{thm-4.3.1} for $V$ and $h^\mathbb{C}(z,\bar{z})$). 
   From the above harmonic map $f_1(z,\bar{z})$, the Sym-Bobenko formula 
enables us to obtain a CMC-surface $\phi_1(z,\bar{z}):V\to\mathbb{R}^3$ (ref.\ 
Subsection \ref{subsec-5.2.1}), which is called the {\it Smyth surface} (cf.\ 
\cite[p.\ 662]{Do-Pe-Wu}). 
   In addition, one can obtain a timelike CMC-surface  
$\phi_2(x,y):W\to\mathbb{R}^3_1$, from the above Lorentz harmonic map 
$f_2(x,y)$. 
   We end this subsection with clarifying an important property of 
$\phi_2(x,y):W\to\mathbb{R}^3_1$: 
\begin{proposition}\label{prop-5.2.1} 
     With the above setting and notation, the Gau{\ss} equation for 
$\phi_2(x,y):W\to\mathbb{R}^3_1$ is the Painlev\'{e} equation of type 
{\rm (III)}.      
\end{proposition}
\begin{proof}   
   Our first aim is to deduce \eqref{eq-5.2.9} below. 
   For 
$k=\operatorname{diag}(s,1/s)\in H_2=S(GL(1,\mathbb{R})\times GL(1,\mathbb{R}))$, 
let us define real numbers $a=a(k)$ and $b=b(k)$ by $a(k):=s^{-4/m}$ and 
$b(k):=s^{-(4+2m)/m}$, respectively.  
   Since 
$k\cdot\eta_\lambda(x)\cdot k^{-1}=\eta_{(b\cdot\lambda)}(a\cdot x)$, 
$k\cdot\tau_\lambda(y)\cdot k^{-1}
 =\tau_{(b\cdot\lambda)}(a^{-1}\cdot y)$ 
and $A_\lambda(0)\equiv\operatorname{id}\equiv B_\lambda(0)$, we understand 
that   
\begin{equation}\label{eq-5.2.8}
\begin{array}{ll}
k\cdot A_\lambda(x)\cdot k^{-1} = A_{b\cdot\lambda}(a\cdot x),
& k\cdot B_\lambda(y)\cdot k^{-1} = B_{b\cdot\lambda}(a^{-1}\cdot y). 
\end{array} 
\end{equation}
   It is immediate from 
$B_\lambda^{-1}\cdot A_\lambda=(B^-_\lambda)^{-1}\cdot B^+_\lambda$ and 
\eqref{eq-5.2.8} that 
$(k\cdot B^-_\lambda(x,y)\cdot k^{-1})^{-1}
 \cdot(k\cdot B^+_\lambda(x,y)\cdot k^{-1})
 =B^-_{b\cdot\lambda}(a\cdot x,a^{-1}\cdot y)^{-1}
    \cdot B^+_{b\cdot\lambda}(a\cdot x,a^{-1}\cdot y)$  
for any $k\in H_2$ and $\lambda\in S^1$. 
   Therefore, the uniqueness of the Birkhoff decomposition allows us to 
conclude  
\[
\begin{array}{ll}
  k\cdot B^+_\lambda(x,y)\cdot k^{-1}
=B^+_{b\cdot\lambda}(a\cdot x,a^{-1}\cdot y) 
& \mbox{for any $k\in H_2$ and $\lambda\in S^1$}. 
\end{array}   
\]
   The above and \eqref{eq-5.2.8} imply that  
\begin{multline*}
  k\cdot C_\lambda(x,y)\cdot k^{-1}
 =k\cdot A_\lambda(x)\cdot B^+_\lambda(x,y)^{-1}\cdot k^{-1}\\
 =A_{b\cdot\lambda}(a\cdot x)\cdot B^+_{b\cdot\lambda}(a\cdot x,a^{-1}\cdot y)^{-1}
 =C_{b\cdot\lambda}(a\cdot x,a^{-1}\cdot y)
\end{multline*} 
---that is, they imply that   
\begin{equation}\label{eq-5.2.9}
\begin{array}{ll}
  k\cdot C_\lambda(x,y)\cdot k^{-1}=C_{b\cdot\lambda}(a\cdot x,a^{-1}\cdot y) 
& \mbox{for any $k\in H_2$ and $\lambda\in S^1$}.  
\end{array} 
\end{equation}
   Now, let 
$U_\lambda(x,y):=C_\lambda(x,y)^{-1}\cdot\partial_x C_\lambda(x,y)$ and 
$V_\lambda(x,y):=C_\lambda(x,y)^{-1}\cdot\partial_y C_\lambda(x,y)$. 
   We express these Maurer-Cartan forms explicitly as follows: 
\begin{equation}\label{eq-5.2.10}
\begin{split}
& U_\lambda(x,y)
 =\begin{pmatrix}
   u_x(x,y)/4 & -(\lambda^{-1}/2)\cdot H\cdot e^{u(x,y)/2} \\
   \lambda^{-1}\cdot Q(x)\cdot e^{-u(x,y)/2} & -u_x(x,y)/4
  \end{pmatrix},\\
& V_\lambda(x,y)
 =\begin{pmatrix}
   -u_y(x,y)/4 & -\lambda\cdot R(y)\cdot e^{-u(x,y)/2} \\
   (\lambda/2)\cdot H\cdot e^{u(x,y)/2} & u_y(x,y)/4,
  \end{pmatrix}
\end{split}  
\end{equation}
where $H$ ($\neq0$) is constant (cf.\ (2.1.5) in \cite{Ko}). 
   Then, the Gau{\ss} equation for $\phi_2(x,y):W\to\mathbb{R}^3_1$ is  
\begin{equation}\label{eq-5.2.11}
  u_{xy}(x,y)-2\cdot Q(x)\cdot R(y)\cdot e^{-u(x,y)}
 +\frac{1}{2}\cdot H^2\cdot e^{u(x,y)}=0 
\end{equation}
(cf.\ (2.1.7) in \cite{Ko}). 
   This equation will become the Painlev\'{e} equation of type (III) 
later (cf.\ \eqref{eq-5.2.11''}). 
   It follows from \eqref{eq-5.2.9} that 
$\alpha^\lambda(x,y):=C_\lambda(x,y)^{-1}\cdot dC_\lambda(x,y)$ satisfies 
$\alpha^\lambda(x,y)=U_\lambda(x,y)dx+V_\lambda(x,y)dy$ and 
$k\cdot\alpha^\lambda(x,y)\cdot k^{-1}
 =\alpha^{b\cdot\lambda}(a\cdot x,a^{-1}\cdot y)$. 
   Hence   
\[
\begin{array}{ll}
  k\cdot U_\lambda(x,y)\cdot k^{-1}
 =a\cdot U_{b\cdot\lambda}(a\cdot x, a^{-1}\cdot y), 
& 
  k\cdot V_\lambda(x,y)\cdot k^{-1}
 =a^{-1}\cdot V_{b\cdot\lambda}(a\cdot x,a^{-1}\cdot y).  
\end{array}
\] 
   Accordingly we obtain  
\[
\begin{array}{lll}
 u(a\cdot x,a^{-1}\cdot y)=u(x,y), 
& Q(a\cdot x)=a^m\cdot Q(x), 
& R(a^{-1}\cdot y)=a^{-m}\cdot R(y)
\end{array}
\]
from \eqref{eq-5.2.10}. 
   Let $\Omega(x\cdot y):=u(1,x\cdot y)$. 
   Then $\Omega(x\cdot y)=u(x,y)$ follows from 
$u(a\cdot x,a^{-1}\cdot y)=u(x,y)$ and $a:=x^{-1}$. 
   Hence we conclude that 
\begin{equation}\label{eq-5.2.12}
  u_{xy}=\partial_x\partial_y\Omega(x\cdot y)
        =\partial_x(\Omega(x\cdot y)'\cdot x)
        =\Omega(x\cdot y)''\cdot x\cdot y+\Omega(x\cdot y)'.        
\end{equation}
   Since $Q(a\cdot x)=a^m\cdot Q(x)$ and $R(a^{-1}\cdot y)=a^{-m}\cdot R(y)$ 
one can express $Q(x)$ and $R(x)$ as $Q(x)=Q_0\cdot x^m$ and 
$R(y)=R_0\cdot y^m$, respectively, where both $Q_0$ and $R_0$ are constant. 
   Therefore we show  
\begin{equation}\label{eq-5.2.13}
\begin{split}
&  -2\cdot Q(x)\cdot R(y)\cdot e^{-u(x,y)}
 +\frac{1}{2}\cdot H^2\cdot e^{u(x,y)}\\
&=-2\cdot Q_0\cdot R_0\cdot (x\cdot y)^m\cdot e^{-\Omega(x\cdot y)}
 +\frac{1}{2}\cdot H^2\cdot e^{\Omega(x\cdot y)}.  
\end{split}
\end{equation}
   In terms of \eqref{eq-5.2.12} and \eqref{eq-5.2.13} we rewrite 
\eqref{eq-5.2.11} as follows: 
\[\tag{5.2.38$'$}\label{eq-5.2.11'}
  \Omega(t)''\cdot t+\Omega(t)'
 -2\cdot Q_0\cdot R_0\cdot t^m\cdot e^{-\Omega(t)}
 +\frac{1}{2}\cdot H^2\cdot e^{\Omega(t)}=0,
\]
where $t:=x\cdot y$.  
   Furthermore, one can rewrite \eqref{eq-5.2.11'} as follows: 
\[\tag{5.2.38$''$}\label{eq-5.2.11''}
  \frac{d^2v}{du^2}
=\frac{1}{v}\Bigl(\frac{dv}{du}\Bigr)^2-\frac{1}{u}\frac{dv}{du}
 +\frac{1}{u}\Bigl(-\frac{H^2}{2+m}v^2+4\frac{R_0\cdot Q_0}{2+m}\Bigr)
\]
by setting $u:=(2t^{(2+m)/2})/(2+m)$ and $v:=e^{\Omega(t)}\cdot t^{-m/2}$. 
   Consequently we assert that the Gau{\ss} equation \eqref{eq-5.2.11} for 
$\phi_2(x,y):W\to\mathbb{R}^3_1$ is the Painlev\'{e} equation \eqref{eq-5.2.11''} 
of type (III).  
\end{proof}

\subsubsection{Delaunay surface in $\mathbb{R}^3$ $\Leftrightarrow$ 
$K$-surface of revolution in $\mathbb{R}^3$}\label{subsec-5.2.5} 
   In this subsection we will use the following notation:   
\begin{enumerate}
\item[(5.2.41)] 
  $G^\mathbb{C}$: the same notation (5.2.1) as in Subsection \ref{subsec-5.2.1}, 
\item[(5.2.42)]  
  $\sigma$: the same notation (5.2.2) as in Subsection \ref{subsec-5.2.1}, 
\item[(5.2.43)]  
  $\nu_1$: the same notation (5.2.3) as in Subsection \ref{subsec-5.2.1}, 
\item[(5.2.44)]  
  $\nu_2:=\nu_1$,
\item[(5.2.45)]  
  $G^\mathbb{C}/H^\mathbb{C}$: the same notation (5.2.5) as in Subsection 
\ref{subsec-5.2.1}, 
\item[(5.2.46)]  
  $G_1/H_1$: the same notation (5.2.6) as in Subsection \ref{subsec-5.2.1},  
\item[(5.2.47)]  
  $G_2/H_2=G_1/H_1=SU(2)/S(U(1)\times U(1))\simeq S^2$, 
\item[(5.2.48)]  
  $\pi_i$: the projection from $G_i$ onto $G_i/H_i$ ($i=1,2$). 
\end{enumerate} 
\setcounter{equation}{48}

   The main purpose in this subsection is to interrelate a surface of 
revolution in $\mathbb{R}^3$ (i.e., a Delaunay surface in $\mathbb{R}^3$) with 
a $K$-surface of revolution in $\mathbb{R}^3$ by means of Theorem 
\ref{thm-4.3.1} (see Theorem \ref{thm-5.2.7}). 
   Here, a {\it $K$-surface} means a surface of constant negative 
curvature $K=-1$. 
   Such a surface is sometimes called a {\it pseudospherical surface}.\par
   
   According to Toda \cite{To} (see \cite{To2} also), one can characterize 
each $K$-surface $M$ in $\mathbb{R}^3$ by an arc length asymptotic line 
coordinate system $(x,y)$ on $M$ and the angle function $\omega(x,y)$ with 
respect to $(x,y)$ by the loop group method.  
   For our purpose, we need to specialize her way concretely to surfaces 
of revolution. 
   First we recall    
\begin{lemma}\label{lem-5.2.2}
   Let $f_{{\mbox{\tiny pseud}}}(u,v)$, $f_{\mbox{\tiny hyper}}(u,v)$ and 
$f_{\mbox{\tiny conic}}(u,v)$ denote the $K$-surfaces of revolution given in 
Gray \cite[Chapter 19.3]{Gray}\footnote{Erratum: p.\ 381, the equation (19.4) 
in \cite{Gray}, should be $-i\sqrt{a^2-b^2}E(iv/a,-b^2/(a^2-b^2))$ instead of 
$-i\sqrt{a^2-b^2}E(iv/a,b^2/(a^2-b^2))$.}, respectively$:$
\[
\begin{array}{ll}
{\displaystyle
   f_{\mbox{\tiny pseud}}(u,v)
    =\bigl(\cos u\sin v,\sin u\sin v,\cos v+\log(\tan v/2)\bigr)}; 
& \\
{\displaystyle
   f_{\mbox{\tiny hyper}}(u,v)
    =\bigl(b\cos u\cosh v,b\sin u\cosh v,
           \int^v_0\sqrt{1-b^2\sinh^2(t)}dt\bigr)}, 
& 0<b;\\
{\displaystyle
   f_{\mbox{\tiny conic}}(u,v)
    =\bigl(b\cos u\sinh v,b\sin u\sinh v,
           \int^v_0\sqrt{1-b^2\cosh^2(t)}dt\bigr)}, 
& 0<b<1.
\end{array}
\] 
   Then in each case, an arc length asymptotic line parametrization 
$(x,y)$ is given by 
\[
\begin{array}{lll}
\mbox{Pseudosphere}: 
& {\displaystyle u=x+y}, 
& {\displaystyle v=2\tan^{-1}(\exp(x-y))};\\
\mbox{Hyperboloid type}: 
& {\displaystyle u=\frac{x+y}{\sqrt{1+b^2}}}, 
& {\displaystyle v=-i\cdot{\rm am}\bigl(\frac{i(x-y)}{\sqrt{1+b^2}},ib\bigr)};\\
\mbox{Conic type}:
& {\displaystyle u=\frac{x+y}{\sqrt{1-b^2}}}, 
& {\displaystyle v=-i\cdot{\rm am}\bigl(i(x-y),\frac{ib}{\sqrt{1-b^2}}\bigr)} 
\end{array} 
\] 
and the first fundamental form $I^{\rm st}$ is expressed as  
\[
\begin{array}{ll}
\mbox{Pseudosphere}: & 
{\displaystyle  I^{\rm st}_{\mbox{\tiny pseud}}
=dx^2+2\bigl(-1+\frac{2}{\cosh^2(x-y)}\bigr)dxdy+dy^2};\\
\mbox{Hyperboloid type}: & 
{\displaystyle  I^{\rm st}_{\mbox{\tiny hyper}}
=dx^2+2\Bigl(1-\frac{2}{1+b^2}\cdot 
 {\rm dn}^2\bigl(\frac{i(x-y)}{\sqrt{1+b^2}},ib\bigr)\Bigr)dxdy+dy^2};\\
\mbox{Conic type}: & 
{\displaystyle  I^{\rm st}_{\mbox{\tiny conic}}
=dx^2+2\Bigl(1-2\cdot{\rm dn}^2\bigl(i(x-y),\frac{ib}{\sqrt{1-b^2}}\bigr)\Bigr)dxdy
 +dy^2}
\end{array}
\]
$($see Remark $\ref{rem-5.2.3}$ for ${\rm am}(u,k)$ and ${\rm dn}(u,k))$. 
\end{lemma} 
\begin{proof} 
   Gray \cite{Gray} presents a method of computing asymptotic line 
parametrizations by the software {\it Mathematica}.  
   His arguments \cite[p.\ 329--330]{Gray}, together with the program in 
\cite[p.\ 328]{Gray}, enable us to obtain the arc length asymptotic line 
parametrizations $(x,y)$ for the $K$-surfaces $f_{{\mbox{\tiny pseud}}}(u,v)$, 
$f_{\mbox{\tiny hyper}}(u,v)$ and   $f_{\mbox{\tiny conic}}(u,v)$, respectively.  
\end{proof}

\begin{remark}\label{rem-5.2.3} 
   Throughout this paper, we use the notation ${\rm am}(u,k)$, ${\rm sn}(u,k)$, 
${\rm cn}(u,k)$, ${\rm dn}(u,k)$ and ${\rm sd}(u,k)$ as in Byrd-Friedman 
\cite{By-Fr} for the Jacobi functions.    
\end{remark}

\begin{lemma}\label{lem-5.2.4} 
   Let $\omega_{\mbox{\tiny pseud}}(x,y)$, $\omega_{\mbox{\tiny hyper}}(x,y)$ 
and $\omega_{\mbox{\tiny conic}}(x,y)$ be the real analytic functions around 
$(0,0)$ defined by 
\[
\begin{array}{ll}
{\displaystyle \omega_{\mbox{\tiny pseud}}(x,y):=2\sin^{-1}(\tanh(x-y))}; 
& \\
{\displaystyle \omega_{\mbox{\tiny hyper}}(x,y)
:=2\sin^{-1}\Bigl(\frac{1}{\sqrt{1+b^2}}\cdot
    {\rm dn}\bigl(\frac{i(x-y)}{\sqrt{1+b^2}},ib\bigr)\Bigr)}, 
& 0<b;\\
{\displaystyle \omega_{\mbox{\tiny conic}}(x,y)
:=-2\sin^{-1}\Bigl(b\cdot{\rm sd}\bigl(\frac{x-y}{\sqrt{1-b^2}},
    \sqrt{1-b^2}\bigr)\Bigr)+\pi},   
& 0<b<1.    
\end{array} 
\]  
   Then, they satisfy 
\[   
\begin{array}{@{}lll}
{\rm (i.1)} & \mbox{Pseudosphere}: & 
I^{\rm st}_{\mbox{\tiny pseud}}
 =dx^2+2\cos(\omega_{\mbox{\tiny pseud}}(x,y))dxdy+dy^2;\\
{\rm (i.2)} & \mbox{Hyperboloid type}: & 
I^{\rm st}_{\mbox{\tiny hyper}}
 =dx^2+2\cos(\omega_{\mbox{\tiny hyper}}(x,y))dxdy+dy^2;\\
{\rm (i.3)} & \mbox{Conic type}: & 
I^{\rm st}_{\mbox{\tiny conic}}
 =dx^2+2\cos(\omega_{\mbox{\tiny conic}}(x,y))dxdy+dy^2;
\end{array}
\]  
and furthermore, they are solutions to the sine-Gordon equation 
$\partial_x\partial_y\omega=\sin\omega$.  
\end{lemma}
\begin{proof}  
   Both (i.1) and (i.2) are immediate from $\cos\omega=1-2\sin^2(\omega/2)$. 
   Let us show (i.3).  
   By direct computations we have  
$\cos(\omega_{\mbox{\tiny conic}}/2)
 =\sin\bigl((\pi/2)-(\omega_{\mbox{\tiny conic}}/2)\bigr)
  =b\cdot{\rm sd}\bigl((x-y)/\sqrt{1-b^2},\sqrt{1-b^2}\bigr)$, 
and 
\begin{equation}\label{eq-5.2.14}
   \cos^2\frac{\omega_{\mbox{\tiny conic}}}{2}
   =b^2\cdot{\rm sd}^2\bigl(\frac{x-y}{\sqrt{1-b^2}},\sqrt{1-b^2}\bigr).
\end{equation}
   Transformation formulas in \cite[p.\ 38]{By-Fr} lead to  
\[
{\rm sn}(iu,ib/\sqrt{1-b^2})
 =i\sqrt{1-b^2}\cdot{\rm sd}(u/\sqrt{1-b^2},\sqrt{1-b^2}).
\] 
   Therefore, \eqref{eq-5.2.14} and $k^2\cdot{\rm sn}^2(u,k)+{\rm dn}^2(u,k)=1$ 
yield that 
\[
  \frac{1+\cos\omega_{\mbox{\tiny conic}}}{2}
  =\cos^2\frac{\omega_{\mbox{\tiny conic}}}{2}
  =-\frac{b^2}{1-b^2}\cdot{\rm sn}^2\bigl(i(x-y),\frac{ib}{\sqrt{1-b^2}}\bigr)
  =1-{\rm dn}^2\bigl(i(x-y),\frac{ib}{\sqrt{1-b^2}}\bigr).
\] 
   Accordingly one deduces  
$\cos(\omega_{\mbox{\tiny conic}}(x,y))
 =1-2\cdot{\rm dn}^2(i(x-y),ib/\sqrt{1-b^2})$, 
and thus (i.3) follows. 
   Now, the rest of this proof is to demonstrate that $\omega_{\mbox{\tiny pseud}}$, 
$\omega_{\mbox{\tiny hyper}}$ and $\omega_{\mbox{\tiny conic}}$ are solutions 
to the sine-Gordon equation, respectively. 
   We will only prove that $\omega_{\mbox{\tiny conic}}$ is a solution to 
the equation, because one can consider the other cases in a similar way. 
   Note that ${\rm dn}((x-y)/\sqrt{1-b^2},\sqrt{1-b^2})$ is positive around 
$(0,0)$ because of $0<b<1$. 
   On the one hand, direct computations show  
\[
\begin{split}
&  \partial_x\omega_{\mbox{\tiny conic}}=-\frac{2}{\sqrt{1-b^2}}
 \cdot\frac{1}{{\rm dn}((x-y)/\sqrt{1-b^2},\sqrt{1-b^2})};\\
&  \partial_x\partial_y\omega_{\mbox{\tiny conic}}
=\frac{2}{\sqrt{1-b^2}}\cdot
 \frac{{\rm sn}((x-y)/\sqrt{1-b^2},\sqrt{1-b^2})
         \cdot{\rm cn}((x-y)/\sqrt{1-b^2},\sqrt{1-b^2})}
      {{\rm dn}^2((x-y)/\sqrt{1-b^2},\sqrt{1-b^2})}.
\end{split}      
\]  
   One the other hand, it follows from \eqref{eq-5.2.14} that 
\begin{multline*}
 \partial_x\omega_{\mbox{\tiny conic}}\cdot\sin\omega_{\mbox{\tiny conic}}
 =-2\cdot\frac{\partial}{\partial x}\cos^2\frac{\omega_{\mbox{\tiny conic}}}{2}
 =-2b^2\cdot\frac{\partial}{\partial x}
    {\rm sd}^2\bigl(\frac{x-y}{\sqrt{1-b^2}},\sqrt{1-b^2}\bigr)\\
 =-\frac{4b^2}{\sqrt{1-b^2}}\cdot
     \frac{{\rm sn}((x-y)/\sqrt{1-b^2},\sqrt{1-b^2})
            \cdot{\rm cn}((x-y)/\sqrt{1-b^2},\sqrt{1-b^2})}
      {{\rm dn}^3((x-y)/\sqrt{1-b^2},\sqrt{1-b^2})}\\
 =\partial_x\omega_{\mbox{\tiny conic}}\cdot\partial_x\partial_y\omega_{\mbox{\tiny conic}}. 
\end{multline*}
   Therefore, one has 
$\partial_x\partial_y\omega_{\mbox{\tiny conic}}
 =\sin\omega_{\mbox{\tiny conic}}$ 
by virtue of $\partial_x\omega_{\mbox{\tiny conic}}<0$.   
\end{proof}

\begin{remark}\label{rem-5.2.5} 
   (i) The solution $\omega_{\mbox{\tiny pseud}}(x,y)$ to the sine-Gordon 
equation in Lemma \ref{lem-5.2.4} can be 
rewritten as follows: 
\begin{equation}\label{eq-5.2.15}
 \omega_{\mbox{\tiny pseud}}(x,y)=4\tan^{-1}(\exp(x-y))-\pi.
\end{equation}
   Indeed, $f(x,y):=4\tan^{-1}(\exp(x-y))-\pi$ is analytic and satisfies  
$\omega_{\mbox{\tiny pseud}}(0,0)=f(0,0)$, 
$\partial_x\omega_{\mbox{\tiny pseud}}=\partial_x f$ 
and $\partial_y\omega_{\mbox{\tiny pseud}}=\partial_yf$. 
   (ii) From every solution $\omega(x,y)$ to the sine-Gordon equation, one can 
construct another solution $\omega'(x,y)$ to the sine-Gordon equation by 
setting 
\[
   \omega'(x,y):=\omega(x,-y)+\pi.
\]
   Consequently, $\omega_{\mbox{\tiny pseud}}'(x,y)=4\tan^{-1}(\exp(x+y))$ 
becomes a solution to the sine-Gordon equation by virtue of \eqref{eq-5.2.15}. 
   Toda \cite{To} uses this solution to study pseudospheres.     
\end{remark}

   For the $K$-surfaces $f_{\mbox{\tiny pseud}}$, $f_{\mbox{\tiny hyper}}$ 
and $f_{\mbox{\tiny conic}}$ in Lemma \ref{lem-5.2.2}, we have obtained arc length 
asymptotic line parametrizations $(x,y)$ and the angle functions $\omega(x,y)$ 
with respect to $(x,y)$, respectively (cf.\ Lemmas \ref{lem-5.2.2} and 
\ref{lem-5.2.4}). 
   According to Toda \cite{To} one can, up to an isometry of $\mathbb{R}^3$, 
reconstruct the $K$-surface from $\omega(x,y)$ and the following potential 
$(\eta_\theta(x),\tau_\theta(y))$: 
\begin{equation}\label{eq-5.2.16}
\begin{split}
& \eta_\theta(x):=\frac{i\theta^{-1}}{2}
  \begin{pmatrix}
  0 & e^{i(\omega(x,0)-\omega(0,0))} \\
  e^{-i(\omega(x,0)-\omega(0,0))} & 0 
  \end{pmatrix}dx,\\
& \tau_\theta(y):=\frac{-i\theta}{2}
  \begin{pmatrix}
  0 & e^{-i\omega(0,y)} \\
  e^{i\omega(0,y)} & 0 
  \end{pmatrix}dy. 
\end{split}   
\end{equation}

\begin{remark}\label{rem-5.2.6} 
   This is not difficult to verify that for 
$\omega(x,y)=\omega_{\mbox{\tiny conic}}(x,y)$ the above potential 
$(\eta_\theta(x),\tau_\theta(y))$ satisfies the morphing condition \eqref{M} 
in Theorem \ref{thm-4.3.1}, while this is not true for the angle functions 
$\omega_{\mbox{\tiny pseud}}$ and $\omega_{\mbox{\tiny hyper}}$ in 
Lemma \ref{lem-5.2.4}. 
\end{remark}

   By means of Theorem \ref{thm-4.3.1}, we will construct a harmonic map 
$f_1(z,\bar{z}):\mathbb{A}^2\to G_1/H_1\simeq S^2$ and a Lorentz harmonic map 
$f_2(x,y):\mathbb{B}^2\to G_2/H_2\simeq S^2$ from the angle function 
$\omega_{\mbox{\tiny conic}}$; and interrelate the associated CMC-surfaces  
$\phi_1(z,\bar{z}):\mathbb{A}^2\to\mathbb{R}^3$ and $K$-surfaces  
$\phi_2(x,y):\mathbb{B}^2\to\mathbb{R}^3$ using $f_1(z,\bar{z})$ and 
$f_2(x,y)$.  
   One will see that $\phi_1(z,\bar{z})$ is a Delaunay surface and 
$\phi_2(x,y)$ is a conic $K$-surface of revolution (cf.\ Theorem 
\ref{thm-5.2.7}).\par

   First, we define a real analytic, para-pluriharmonic potential 
$(\eta_\theta(x),\tau_\theta(y))$ on $(U,I)$ by \eqref{eq-5.2.16} with 
$\omega(x,y)=\omega_{\mbox{\tiny conic}}(x,y)$ as given in 
Lemma \ref{lem-5.2.4}. 
   Here $U$ denotes any open neighborhood of $\mathbb{B}^2$ at $(0,0)$ such 
that $\omega_{\mbox{\tiny conic}}(x,y)$ is analytic on $U$. 
   As remarked above, $(\eta_\theta(x),\tau_\theta(y))$ satisfies the morphing 
condition \eqref{M}.   
   Next, let us solve the two initial value problems: 
$(A_\theta)^{-1}\cdot dA_\theta=\eta_\theta$ and 
$(B_\theta)^{-1}\cdot dB_\theta=\tau_\theta$ with 
$A_\theta(0,0)\equiv\operatorname{id}\equiv B_\theta(0,0)$; and factorize 
$(A_\theta,B_\theta)
 \in\widetilde{\Lambda}(G_2)_\sigma\times\widetilde{\Lambda}(G_2)_\sigma$ 
in the Iwasawa decomposition (cf.\ Theorem \ref{thm-3.1.5}): 
\[ 
\begin{array}{llll}
 (A_\theta,B_\theta)=(C_\theta,C_\theta)\cdot(B^+_\theta,B^-_\theta), 
& C_\theta\in\widetilde{\Lambda}(G_2)_\sigma, 
& B^+_\theta\in\widetilde{\Lambda}^+_*(G_2)_\sigma, 
& B^-_\theta\in\widetilde{\Lambda}^-(G_2)_\sigma. 
\end{array}
\]
   Then Proposition \ref{prop-3.2.3} assures that there exists an open 
neighborhood $W$ of $U$ at $(0,0)$, and  
$(f_2)_\theta:=\pi_2\circ C_\theta(x,y):(W,I)\to G_2/H_2\simeq S^2$ is a 
Lorentz harmonic map for any $\theta\in\mathbb{R}^+$. 
   Moreover, by Theorem \ref{thm-4.3.1} there exist an open neighborhood $V$ of 
$\mathbb{A}^2$ at $(0,0)$ and a smooth map $h^\mathbb{C}:V\to H^\mathbb{C}$ 
such that 
$(f_1)_\lambda:=\pi_1\circ C_\lambda'(z,\bar{z}):(V,J)\to G_1/H_1\simeq S^2$ is 
a harmonic map for any $\lambda\in S^1$, where 
$C_\lambda':=C_\lambda\cdot h^\mathbb{C}$. 
   Accordingly we have obtained a harmonic map $f_1(z,\bar{z})$ and a Lorentz 
harmonic map $f_2(x,y)$ from the potential \eqref{eq-5.2.16}: 
\[
\begin{array}{lll}
 (f_1)_\lambda=\pi_1\circ C_\lambda'(z,\bar{z}):(V,J)\to G_1/H_1\simeq S^2,
& C_\lambda'(0,0)\equiv\operatorname{id}, & \lambda\in S^1;\\ 
 (f_2)_\theta=\pi_2\circ C_\theta(x,y):(W,I)\to G_2/H_2\simeq S^2, 
& C_\theta(0,0)\equiv\operatorname{id}, & \theta\in\mathbb{R}^+.\\ 
\end{array}
\]  
   From $f_1(z,\bar{z})_\lambda$ and $f_2(x,y)_\theta$ one obtains a Delaunay 
surface and a conic $K$-surface of revolution, respectively: 
\begin{theorem}\label{thm-5.2.7} 
   Let $(f_1)_\lambda=\pi_1\circ C_\lambda'(z,\bar{z}):(V,J)\to S^2$ and 
$(f_2)_\theta=\pi_2\circ C_\theta(x,y):(W,I)\to S^2$ be the above harmonic 
map and Lorentz harmonic map. 
   Let $\phi_1(z,\bar{z})_\lambda:V\to\mathbb{R}^3$ $($resp.\ 
$\phi_2(x,y)_\theta:W\to\mathbb{R}^3)$ denote the CMC-surface $($resp.\ 
$K$-surface$)$ determined by the Sym-Bobenko formula $($resp.\ the Sym 
formula$):$    
\[
\begin{split}
&
\phi_1(z,\bar{z})_\lambda:=
  i\cdot\lambda\cdot\frac{\partial C'_\lambda}{\partial\lambda}
     \cdot {C'_\lambda}^{-1} 
      +\frac{1}{2}\cdot\operatorname{Ad}(C'_\lambda)\cdot
     \begin{pmatrix}
      i & 0 \\
      0 & -i \\
     \end{pmatrix},\\
& 
\phi_2(x,y)_\theta:=\theta\cdot\frac{\partial C_\theta}{\partial\theta}
     \cdot C_\theta^{-1}.
\end{split}     
\]
   Then, $\phi_1(z,\bar{z})_\lambda:V\to\mathbb{R}^3$ is a Delaunay surface, 
and $\phi_2(x,y)_\theta:W\to\mathbb{R}^3$ is a conic $K$-surface of revolution.      
\end{theorem}
\begin{proof} 
   The $K$-surface $\phi_2(x,y)_\theta:W\to\mathbb{R}^3$ is endowed with the 
angle function $\omega_{\mbox{\tiny conic}}(x,y)$ (cf.\ \eqref{eq-5.2.16}). 
   Therefore, Toda \cite{To} assures that $\phi_2(x,y)_\theta:W\to\mathbb{R}^3$ 
coincides, up to an isometry of $\mathbb{R}^3$, with the $K$-surface 
$f_{\mbox{\tiny conic}}$ given in Lemma \ref{lem-5.2.2}. 
   Consequently, the rest of proof is to conclude that 
$\phi_1(z,\bar{z})_\lambda:V\to\mathbb{R}^3$ is a Delaunay surface.  
   First, let us verify that 
\begin{equation}\label{eq-5.2.17}
\begin{array}{ll}
 C_\theta(x+t,y+t)=\chi_\theta(t)\cdot C_\theta(x,y), 
& \mbox{for any $t\in\mathbb{R}$ with $(x+t,y+t)\in W$}, 
\end{array}   
\end{equation}
where $\chi_\theta(t):=C_\theta(t,t)$. 
   By the proof of Lemma \ref{lem-5.2.4} we have 
$(\partial_x\omega_{\mbox{\tiny conic}})(x+t,y+t)
 =(\partial_x\omega_{\mbox{\tiny conic}})(x,y)$ 
and $\omega_{\mbox{\tiny conic}}(x+t,y+t)=\omega_{\mbox{\tiny conic}}(x,y)$. 
   Therefore 
$(C_\theta^{-1}\cdot dC_\theta)(x+t,y+t)=(C_\theta^{-1}\cdot d C_\theta)(x,y)$
follows from the equation (6) in \cite{To2}; and thus 
\[
(C_\theta^{-1}\cdot dC_\theta)(x+t,y+t)=(C_\theta^{-1}\cdot d C_\theta)(x,y)
=\bigl((\chi_\theta(t)\cdot C_\theta)^{-1}
    \cdot d(\chi_\theta(t)\cdot C_\theta)\bigr)(x,y). 
\] 
  In view of $C_\theta(0,0)\equiv\operatorname{id}$ one sees that 
$C_\theta(0+t,0+t)=\chi_\theta(t)\cdot C_\theta(0,0)=\chi_\theta(t)$. 
  Hence, one concludes \eqref{eq-5.2.17}.  
  From \eqref{eq-5.2.17} it follows that 
\[
\begin{array}{ll}
 C_\lambda(z+t,\bar{z}+t)=\chi_\lambda(t)\cdot C_\lambda(z,\bar{z}), 
& \lambda\in S^1, 
\end{array}   
\]
where we remark that the variable $\theta$ of $\chi_\theta(t)$ can vary 
in the whole $\mathbb{C}^*$ because of $\chi_\theta(t)=C_\theta(t,t)$. 
   Since 
$C'_\lambda(z,\bar{z})=C_\lambda(z,\bar{z})\cdot h^\mathbb{C}(z,\bar{z})$, we 
deduce that   
\begin{equation}\label{eq-5.2.18}
  C'_\lambda(z+t,\bar{z}+t)
  =\chi_\lambda(t)\cdot C'_\lambda(z,\bar{z})\cdot k^\mathbb{C}(t,z,\bar{z}),  
\end{equation} 
where 
$k^\mathbb{C}(t,z,\bar{z})
 :=h^\mathbb{C}(z,\bar{z})^{-1}\cdot h^\mathbb{C}(z+t,\bar{z}+t)$.  
   If $k^\mathbb{C}(t,z,\bar{z})$ belongs to $H_1$ ($\subset G_1$), then it is 
immediate from $C'_\lambda(z+t,\bar{z}+t),C'_\lambda(z,\bar{z})\in G_1$ that 
$\chi_\lambda(t)\in G_1$; so that $\phi_1(z,\bar{z})_\lambda:V\to\mathbb{R}^3$ 
admits a one-parameter group of isometries, which implies that 
$\phi_1(z,\bar{z})_\lambda:V\to\mathbb{R}^3$ is a Delaunay surface (cf.\ 
Theorem \cite[p.\ 127]{Do-Ha}). 
   Thus it suffices to confirm  
\[ 
 k^\mathbb{C}(t,z,\bar{z})\in H_1. 
\]  
   For the extended framing $C'_\lambda(z,\bar{z})$ of the harmonic map 
$(f_1)_\lambda:(V,J)\to S^2$, we have the Maurer-Cartan form:    
\[
\begin{array}{ll}
  {C'_\lambda}^{-1}\cdot\partial_z C'_\lambda=U, 
& {C'_\lambda}^{-1}\cdot\partial_{\bar{z}} C'_\lambda=V,\\
 U
 =\begin{pmatrix}
   u_z/4 & -(\lambda^{-1}/2)\cdot H\cdot e^{u/2} \\
   \lambda^{-1}\cdot Q\cdot e^{-u/2} & -u_z/4
  \end{pmatrix},
& V
 =\begin{pmatrix}
   -u_{\bar{z}}/4 & -\lambda\cdot R\cdot e^{-u/2} \\
   (\lambda/2)\cdot H\cdot e^{u/2} & u_{\bar{z}}/4
  \end{pmatrix}, 
\end{array}  
\] 
where $H$ ($\neq 0$) is constant. 
   Since 
$k^\mathbb{C}(t,z,\bar{z})
 =h^\mathbb{C}(z,\bar{z})^{-1}\cdot h^\mathbb{C}(z+t,\bar{z}+t)
 \in H^\mathbb{C}=S(GL(1,\mathbb{C})\times GL(1,\mathbb{C}))$ 
is a diagonal matrix, we can express it as  
\[
  k^\mathbb{C}(t,z,\bar{z})
                =\begin{pmatrix}
                 d(t,z,\bar{z}) & 0 \\
                 0 & d(t,z,\bar{z})^{-1}
                 \end{pmatrix}.
\]
   Then, the Maurer-Cartan form on the right hand side of \eqref{eq-5.2.18} is  
\begin{equation}\label{eq-5.2.19} 
\begin{array}{l} 
U(z,\bar{z})
 =\begin{pmatrix}
   u_z(z,\bar{z})/4+d^{-1}\cdot d_z 
 & -(\lambda^{-1}/2)\cdot H\cdot e^{u(z,\bar{z})/2}\cdot d^{-2} \\
   \lambda^{-1}\cdot Q(z)\cdot e^{-u(z,\bar{z})/2}\cdot d^2 
 & -u_z(z,\bar{z})/4-d^{-1}\cdot d_z
  \end{pmatrix},\\
V(z,\bar{z})
 =\begin{pmatrix}
   -u_{\bar{z}}(z,\bar{z})/4+d^{-1}\cdot d_{\bar{z}} 
  & -\lambda\cdot R(\bar{z})\cdot e^{-u(z,\bar{z})/2}\cdot d^{-2} \\
   (\lambda/2)\cdot H\cdot e^{u(z,\bar{z})/2}\cdot d^2 
  & u_{\bar{z}}(z,\bar{z})/4-d^{-1}\cdot d_{\bar{z}}
  \end{pmatrix}.
\end{array}
\end{equation} 
   The Maurer-Cartan form on the left hand side of \eqref{eq-5.2.18} is   
\begin{equation}\label{eq-5.2.20}
\begin{array}{@{}l@{}} 
U(z+t,\bar{z}+t)
 =\begin{pmatrix}
   u_z(z+t,\bar{z}+t)/4 
 & -(\lambda^{-1}/2)\cdot H\cdot e^{u(z+t,\bar{z}+t)/2} \\
   \lambda^{-1}\cdot Q(z+t)\cdot e^{-u(z+t,\bar{z}+t)/2} 
 & -u_z(z+t,\bar{z}+t)/4
  \end{pmatrix},\\
V(z+t,\bar{z}+t)
 =\begin{pmatrix}
   -u_{\bar{z}}(z+t,\bar{z}+t)/4 
  & -\lambda\cdot R(\bar{z}+t)\cdot e^{-u(z+t,\bar{z}+t)/2} \\
   (\lambda/2)\cdot H\cdot e^{u(z+t,\bar{z}+t)/2} 
  & u_{\bar{z}}(z+t,\bar{z}+t)/4
  \end{pmatrix}.
\end{array}
\end{equation}
   Let us compare the 12-entry of $U$ with the 21-entry of $V$ in 
\eqref{eq-5.2.19} and \eqref{eq-5.2.20}. 
   Then one has $d^4=1$, whence $d=\pm i$, $\pm 1$. 
   This means that $k^\mathbb{C}(t,z,\bar{z})\in S(U(1)\times U(1))=H_1$.    
\end{proof} 


\end{document}